 \documentclass[12pt, reqno]{amsart}
 
 \usepackage{hyperref,mathtools,color}
 \usepackage{amssymb,amsmath,graphicx,amsthm}
 \usepackage{pdfsync}
 \def\M{\mathcal M}
\def\G{\mathcal G}
\def\Id{\text{Id}}
 
\newcommand{\blue}[1]{\textcolor{blue}{#1}}

 \def\epf{ \hskip15cm$\Box$}
 \def\bpf{{\it Proof. }\hskip0.2cm}
 \def\bt{\begin{theorem}}
 	\def\el{\end{lemma}}
 \def\bl{\begin{lemma}}
 	\def\et{\end{theorem}}
 \def\bp{\begin{proposition}}
 	\def\ep{\end{proposition}}
 \def\bd{\begin{definition}}
 	\def\ed{\end{definition}}
 \def\br{\begin{remark}}
 	\def\er{\end{remark}}

 \def\T{\text}

 \def\La{\Lambda}
 \def\supp{\T{supp}}

 \def\T{\text}
 
 \newcommand{\om}{\omega}
 
 \newcommand{\bom}{\bar{\omega}}

 \newcommand{\we}{\wedge}
 
 \newcommand{\no}[1]{\|{#1}\|}
 
 \def\R{{\mathbb R}}
 
 \def\I{{\mathcal I}}
 
 \def\C{\mathbb C}

 \def\N{{\mathcal N}}
 
 \def\la{\langle}
 \def\ra{\rangle}
 \def\di{\partial}
 \def\dib{\bar\partial}
 \def\label#1{\label{#1}}

 \numberwithin{equation}{section}
 
 \def\T{\text}

 \theoremstyle{plain}

 \newtheorem{theorem}{Theorem}[section]
 
 \newtheorem{corollary}[theorem]{Corollary}
 
 \newtheorem{lemma}[theorem]{Lemma}
 
 \newtheorem{proposition}[theorem]{Proposition}

 \theoremstyle{definition}
 
 \newtheorem{definition}[theorem]{Definition}
 
 \theoremstyle{remark}

 \newtheorem{remark}[theorem]{Remark}

 \newcommand{\dbar}{\bar\partial}

 \newcommand{\dbarb}{\bar\partial_b}
 \newcommand{\Boxb}{\Box_b}
 \newcommand{\vp}{\varphi}
 
 \newcommand{\atopp}[2]{\genfrac{}{}{0pt}{2}{#1}{#2}}
 
 \newcommand{\eps}{\epsilon}
 
 \newcommand{\F}{\mathcal F}

 \DeclareMathOperator{\Tr}{Tr}

 \renewcommand{\H}{\mathcal H}

 \def\L{\mathcal L }
 \def\Re{\text{Re}}
  \def\Im{\text{Im}}

\begin{document}
 	
 	\title{The Kohn-Laplace equation on abstract CR manifolds: local regularity}         
 \author[T.V.~Khanh]{Tran Vu Khanh}
 		\address{Tran Vu Khanh}
 		\address{School of Mathematics and Applied Statistics, University of Wollongong, NSW, Australia,  2522}
 		\email{tkhanh@uow.edu.au}
 	
 		\thanks{Research was supported by the Australian Research Council
 			DE160100173. This work was done in part while the author was  a visiting member at the Vietnam Institute for Advanced Study in Mathematics (VIASM). He would like to thank the institution for its hospitality and support.}

 	\begin{abstract} 	The purpose of this paper is to establish local regularity of the solution operator to the Kohn-Laplace equation, called the complex Green operator, on abstract CR manifolds of hypersurface type. For a cut-off function $\sigma$, we introduce the $\sigma$-superlogarithmic property, a potential theoretical condition on CR manifolds. We prove that if the given datum is smooth on an open set containing the support of $\sigma$ then the  solution is smooth on the interior of $\{x\in M:\sigma(x)=1\}$. Furthermore, we also study the smoothness of the integral kernel of the complex Green operator.
 	\end{abstract}
  	\maketitle
  	\begin{flushright}   {\it	In memory of my teacher Beppe Zampieri} \end{flushright}

 \tableofcontents 	
 	\section{Introduction}
 	\label{s1}
 	\subsection{Introduction and motivation}
 	Let $M$ be an abstract $(2n+1)$-dimensional CR manifold of hypersurface type (or CR manifold for short) equipped with a CR structure $T^{1,0}M$. Let $\theta$ be a contact form of $M$ which is a  nonvanishing purely imaginary $1$-form that annihilates 
	$T^{1,0}M\oplus T^{0,1}M$.  
	Then the Levi form on $M$ with respect to $\theta$ is the Hermitian form on $T^{1,0}M$ given by $d\theta(L\we\bar L')$ where $L,L'\in T^{1,0}M$. We say that $M$ is  pseudoconvex  if  the Levi form $d\theta(L\we\bar L)$ is nonnegative 
	for any $L\in T^{1,0}M$. 
 	For a $C^2$ function $\lambda$ on $M$, the alternating $(1,1)$-form $\L_\lambda:=\frac{1}{2}(\di_b\dib_b-\dib_b\di_b)\lambda$ on $T^{1,0}M\times T^{0,1}M$ is called the Levi form of $\lambda$.

 	For $0\le q\le n$, the bundle of $(0,q)$-forms of smooth coefficients on $M$ is denoted by $C^\infty_{0,q}(M)$. The induced $L^2$-inner product and norm on $C^\infty_{0,q}(M)$ is defined by
 	$$
 	(u,v)_{L^2}=\int_M\la u,  v\ra dV,\qquad \no{u}_{L^2}^2=(u,u)_{L^2}.
 	$$
 	where  $dV$ is the element of volume on $M$ and $\la\cdot,\cdot\ra $ is a Hermitian metric on $T^{1,0}M\oplus T^{0,1}M$.
 	In general, $\no{\cdot}_{H^s}$ denotes the Sobolev $H^s$-norm that $H^0:=L^2$. We also define $L_{0,q}^2(M)$  and $H^s_{0,q}(M)$ to be the Hilbert space obtained by completing $C^\infty_{0,q}(M)$ under the $L^2$ and $H^s$ norms, respectively. The tangential Cauchy-Riemann operator $\dib_b:\, C_{0,q}^\infty(M)\to C_{0,q+1}^\infty(M)$ is the natural restriction of the de-Rham exterior complex of derivatives $d$ to $C_{0,q+1}^\infty(M)$. We denote by $\dib^*_b:\, C_{0,q+1}^\infty(M)\to C_{0,q}^\infty(M)$ the $L^2$-adjoint of $\dib_b$ and define the Kohn-Laplacian by 
 	$$\Box_b:=\dib_b\dib_b^*+\dib_b^*\dib_b:\, C_{0,q}^\infty(M)\to C_{0,q}^\infty(M).$$
We can extend $\dib_b$, $\dib_b^*$ and $\Box_b$ to unbounded operators in $L^2$-spaces with domains consisting of forms in $L^2$-spaces where the results (computed in the sense of distributions) are actually in $L^2$-spaces.
 	Let $\H_{0,q}(M)=\{u\in L^2_{0,q}(M):\Box_b u=0\}$ be the space of harmonic $(0,q)$-forms, that coincides with $\ker(\dib_b)\cap \ker(\dib_b^*)$. 
 	
 	 Given $\varphi\in L^2_{0,q}(M)$ with $\varphi\perp \H_{0,q}(M)$, to solve the Kohn-Laplace equation with datum $\varphi$ consists in finding a $(0,q)$-form $u\in \T{Dom}(\Box_b)$ such that 
 	\begin{eqnarray}
 	\label{KLe}\Box_b u=\varphi.
 	\end{eqnarray}
 	Observe that if a solution of \eqref{KLe} exists then there is a unique solution $u$ of \eqref{KLe} such that $u\perp \H_{0,q}(M)$, we denote the operator that maps $\varphi \mapsto u$ by $G_q$.
 	 The operator $G_q$ is called the  complex Green operator acting $(0,q)$-forms. We extend the operator $G_q$ to a linear operator by setting it to be equal to zero on $\H_{0,q}(M)$. If there exists $c>0$ so that for any 
 	 $\vp\in L^2_{0,q}(M) \cap \H_{0,q}^\perp(M)$, $\|u\|_{L^2} \leq c \|\vp\|_{L^2}$ then $G_q$ is bounded and self adjoint in $L^2_{0,q}(M)$. Furthermore, if $G_q$ exists, then $\dib_b^* G_q$ and  $G_q\dib^*$ are the canonical solution operators to the $\dib_b$-equation on $(0,q)$ and $(0,q+1)$-forms, respectively;  $\dib_b G_q$ and $G_q\dib_b$ are the canonical solution operators to the $\dib_b^*$-equation on $(0,q)$ and $(0,q-1)$-forms, respectively; $I-\dib_b^* \dib_bG_{q}$ and $I-\dib_b^*G_{q}\dib_b$ are the Szeg\"o projections---the projections onto $\ker(\dib_b)$;  and $I-\dib_b \dib_b^*G_{q}$ and $I-\dib_bG_{q}\dib_b^*$ are the anti Szeg\"o projections--- the projections onto $\ker(\dib_b^*)$. The fundamental questions regarding the Kohn-Laplace equation can be asked, at the level of understanding the regularity of these operators, as follows:
 	\begin{enumerate}
 		\item the $L^2$ boundedness;
 		\item the global regularity in $C^\infty$ and $H^s$ spaces;
 		\item the local regularity  in $C^\infty$ and $H^s$ spaces;
 		\item the smoothness of the integral kernels.
 	\end{enumerate}

 	A great deal of work has been done relatively to Question 1  (cf. \cite{Bar12, HaRa11, HaRa15a, Koh86, KoNi06, Nic06, Sha85}) and Question 2 (cf. \cite{RaSt08, Rai10, KhPiZa12a, StZe15, Str12, HaPeRa15}) on some classes of CR manifolds. In the parallel with this work in \cite{KhRa16a}, with Raich we give a positive answer to the first two questions on abstract CR manifolds equipped  with ``good" CR-plurisubharmonic functions, which is a  potential theoretical condition
	that we call a $P$-property. 
	In particular, we prove that if $M$ is a compact, pseudoconvex-oriented, $(2n+1)$-dimensional CR manifold equipped with a strictly CR-plurisubharmonic function on $(0,q_0)$-forms, i.e.,
	 there exist a global function $\lambda$ and constant $\alpha>0$ such that
 	$$\la(\L_\lambda+d\theta)\lrcorner u, u\ra\ge \alpha|u|^2 \quad\T{holds on $M$ for all smooth $(0,q)$-forms $u$,} $$ 
 	then the operators $G_{q}$, $\dib_b^*G_q$, $G_q\dib_b^*$, $\dib_bG_q$ $G_q\dib_b$, $I-\dib_b^*\dib_bG_q$, $I-\dib_b^*G_q\dib_b$, $I-\dib_b\dib_b^*G_q$, $I-\dib_bG_{q}\dib_b^*$, $\dib_bG^2_q\dib_b^*$ and $\dib_b^*G^2_q\dib_b$ are $L^2$-bounded for all  degrees $q_0\le q\le n-q_0$. Here $\lrcorner$ is the contraction operator defined in Section \ref{s2} below. If the manifold also admits a covering that satisfies ``a weak compactness property'' on  $(0,q_0)$-forms, 
	then these operators are both $C^\infty$-globally regular and exactly regular in $H^s$-spaces \cite{KhRa16a}.
	Our ``weak compactness" property acting $(0,q_0)$-forms on a covering  $\{U_\eta\}_\eta$ of $M$ is defined as follows: for any $U_\eta$ and any  $t\ge 1$ there exist a vector field $T_t$ transversal $T^{1,0}M\oplus T^{0,1}M$, a 
	uniformly bounded, smooth function $\lambda_t^\eta$, and a function $f(t)\nearrow \infty$ as $t\nearrow \infty$  such that 
 	\begin{eqnarray}\label{cmptness}
 		\begin{cases}
 			\left\la(\L_{\lambda^\eta_t}+td\theta)\lrcorner u, u\right\ra\ge f(t)|\T{\{\it Lie\}}_{T_t}(\theta)|^2|u|^2,\\
 			0<c_1\le \theta(T_t)\le c_2 \quad \T{uniformly in $t$,}
 		\end{cases} 
 	\end{eqnarray} hold on $U_\eta$ for all smooth $(0,q_0)$-forms $u$.
 	
 	The purpose of this paper is to answer Question 3 and 4. The study of the local $C^\infty$ regularity and the smoothness of the distribution kernels of $G_q$ and the relative operators are preliminary steps for the derivation of $L^p$ and H\"older estimates for these operators (cf. \cite{Chr88, Koe02, FeKo88, FeKo88a, NaRoStWa89, NaSt06, FoSt74e, ChNaSt92, Mac88}).
 	  For the question on local regularity,  
 	we first recall an operator is called locally regular on a given open set $U$
 	 if it preserve $C^\infty$ on $U$. The local regularity of $G_q$ is equivalent  to the local hypoellipticity of $\Box_b$ in the sense that  if the restriction of $\Box_bu$ to  $U$ is in
 	$C_{0,q}^\infty(U)$, then the restriction of $u$ to $U$ is also in $C_{0,q}^\infty(U)$.  

 		It is well known in the general theory of partial differential equations that a subelliptic estimate implies local hypoellipticity \cite{KoNi65}. For the $\dib$-Neumann problem, Kohn \cite{Koh79} gave a sufficient condition for subellipticity over pseudoconvex domains in whose  boundaries are  real analytic and of finite type by introducing a sequence of ideals of subelliptic multipliers. For example,  for the hypersurface in $\C^{n+1}$ defined by
 	\begin{equation}
 	\label{example0}
 	\Im\,{z}_{n+1}=\sum_{k=1}^m|h_k(z_1, \dots,z_n)|^2, 
 	\end{equation}  
 	where $h_k$'s are holomorphic functions and whose  common zeroes is only the origin of $\C^n$, the origin is a point of
 	finite type. Hence, by \cite[Theorem 1.19]{Koh79},  a subelliptic estimate for $\Box$ holds (see \cite{Koh79} for the analysis of this example). 
	In \cite{Cat83, Cat87}, Catlin proved, regardless whether the boundary is real analytic or not, the equivalence of the finite type condition and the subelliptic estimate by establishing that each of these two conditions is equivalent to a third technical condition called the {\it potential-theoretical condition} on domains.
 For a CR manifold which is the boundary of a bounded pseudoconvex domain in $\C^{n+1}$, the subelliptic estimate for $\Box_b$ and $\Box$ are equivalent \cite{Koh02}. Consequently, $\Box_b$ is locally hypoelliptic	 
 on the set of points of finite type. 
   An amazing result by Kohn \cite{Koh02} showed that superlogarithmic estimate for $\Box_b$  still implies local hypoellipticity. It should be noted that superlogarithmicity is a very weak estimate compared with subellipticity and that there are many classes of domains of infinite type (flat boundary) for which the superlogarithmic estimate holds (cf.\cite{KhZa10}). Superlogarithmicity is also a necessary condition for local hypoellipticity of  $\Box_b$ on some hypersurfaces. For example, Christ \cite{Chr02} considered  the model in $\C^2$ defined by
 \begin{eqnarray}
 \Im\,{z_2}=\exp\left(-\frac{1}{|\Re z_1|^\alpha}\right),
 \end{eqnarray} for some $\alpha>0$, and showed that  in a neighborhood of the origin, $\Box_b$ is locally hypoelliptic  if and only if superlogarithmic estimate holds, i.e., $\alpha<1$ (see also \cite{KhZa10, BaKhZa12}). The problem is to find conditions under which $\Box_b$ is locally hypoelliptic but superlogarithmicity fails. 	
 	It is a recent discovery \cite{Koh00, Chr01, BaPiZa15} that there exist many classes of hypersurfaces that superlogarithmicity might not hold but $\Box_b$ is still locally hypoelliptic. 
	For example, $\Box_b$ is locally hypoelliptic on hypersurface models in $\C^{n+1}$ either defined by 
 	\begin{eqnarray}
 \label{model2}
  \Im\, z_{n+1} =\exp\left(-\frac{1}{\left(\sum_{k=1}^m |h_k(z_1,\dots,z_n)|^2\right)^{\beta}}\right),
 	\end{eqnarray}
 	where $h_k$'s are holomorphic functions whose common zeroes is  only at the origin of $\C^n$, and $\beta>0$ (see \cite{Koh00, Chr01}); or
 	\begin{eqnarray}\label{model3}
 	 \Im\, z_{n+1}=\sum_{j=1}^n |\Re z_j|^{2m_j}\exp\left(-\frac{1}{|z_j|^{\beta_j}}\right),
 	\end{eqnarray}
 	where  $m_j\in \mathbb N$ and $\beta_j>0$, $j=1,\dots,n$ (see \cite{BaPiZa15}). The problem of local hypoellipticity of $\Boxb$, raised by Zampieri, remains open for the model, defined  by
 	\begin{eqnarray}\label{model4}
 		 \Im\, z_{n+1}=\sum_{j=1}^n \exp\left(-\frac{1}{|\Re z_j|^{\alpha_j}}\right)\exp\left(-\frac{1}{|z_j|^{\beta_j}}\right)
 	\end{eqnarray}
 	where $0<\alpha_j<1$ and $\beta_j>0$, $j=1,\dots,n$. 
 	
 The problem raised by Zampieri has motivated my work on this paper. 	
 The main goal of this paper consists in establishing the local hypoellipticity of $\Box_b$ and hence the regularity of $G_q$ and related operators. The technique is to establish a 
 sufficient potential theoretic condition on abstract CR manifolds.  Surprisingly, not only
this shows local hypoellipticity of $\Box_b$ on the hypersurface defined by \eqref{model4}, but also we discover a new class of hypersurfaces in $\C^{n+1}$ in which superlogarithmic estimate for $\Box_b$ might not hold but $\Box_b$ is locally hypoelliptic. For example, a generalization of Kohn's example in \cite{Koh79}:
	\begin{equation}
	\label{example1}
	\Im\,{z}_{n+1}=\sum_{j=1}^mH_j(|h_j(z_1, \dots,z_n)|), 
	\end{equation}  
	where $H_j:\R^+\to\R^+$ are increasing convex functions with $H_j(0)=0$ and $h_j$'s are holomorphic functions and have the common zero at only the origin of $\C^n$; or a generalization of Christ's example in \cite{Chr02}:
	\begin{equation}
	\label{example2}
	\Im\,{z}_{2}=\exp\left(-\frac{1}{|\Re\,( z_1^m)|^\alpha}\right)
	\end{equation}  
	where $m\in\mathbb N$  and $0< \alpha<1$.	
	  In particular,   all mentioned examples \eqref{example0}-\eqref{example2} can be unified  to a class of  hypersurfaces in $\C^{n+1}$  having the form
  \begin{eqnarray}\label{model5}
  \Im\, z_{n+1}=\sum_{k=1} H_k\left(\sum_{j}|h_{kj}|^2\right)F_k\left(\sum_{j}|Re\,h_{kj}|^2\right)
  \end{eqnarray}
 where $H_k,F_k:\R^+\to \R^+$  are increasing and convex functions such that  $H_k(0)F_k(0)=0$; and the $h_{kj}$'s are holomorphic functions in $\C^n$ with an isolated zero at the origin. The crucial condition for local hypoellipticity of $\Box_b$  is only that $\lim_{\delta\to 0}\delta \ln\left( F_k(\delta^2)\right)=0$ and no matter how small the value of $H_k$'s are.
 \subsection{The main theorems} Throughout of this paper, we  consider $M$ to be a $(2n+1)$-dimensional, pseudoconvex CR manifold of hypersurface type with a fixed purely imaginary contact form $\theta$ and its dual vector field $T$;  $U$ to be an open set of $M$ on which a smooth basis of $\C TM$ exists. We do not require $M$ to be compact, however, we assume that $\dib_b$ has closed range in $L^2$ for all degree of forms (see \cite{KhRa16a} for discussion on this condition). Let $\sigma$ be a cutoff function, we denote by $V^\sigma_0$  the interior of the set $\{x\in M:\sigma(x)=1\}$. If $\zeta$ is another cutoff function,  then we use notation $\sigma\prec \zeta$ if $\zeta=1$ on the support of $\sigma$.  For two functions $f,g:\R^+\to \R^+$, we write $f\gg g$ if $\lim_{t\to\infty }f(t)/g(t)=\infty$.
 
 Now we introduce a potential theoretic condition  that is a sufficient condition for local hypoellipticity of $\Box_b$. 
 	\begin{definition}\label{def1}
 		Let $\sigma$ be a cutoff function in  $U$. We say that the $\sigma$-superlogarithmic property on $(0,q_0)$-forms holds if there exist a pair of rate functions $f,f_0:\R^+\to\R^+$ with $f,f_0\gg 1$ and $f(t), f_0(t)\le t^{1/2}$ for $t\ge 1$,  and  a family of smooth real-valued function $\{\lambda_t\}_{t\ge1}$ defined on $U$ such that 
 		\begin{eqnarray}\label{definition1}
 		\begin{cases}
 		\left\la(\L_{\lambda_t}+td\theta)\lrcorner u, u\right\ra\ge f^2_0(t)\left(\ln t|\la\L_\sigma\lrcorner u,u\ra|+\ln^2t|\di_b\sigma \lrcorner u|^2\right)+f^2(t)|u|^2,\\
 		|T^m(\lambda_t)|\le c_mt^{\frac{m}{2}}, \T{for all } m\in \mathbb N,
 		\end{cases}
 		\end{eqnarray}
		holds at any $x\in U$ for all smooth $(0,q_0)$-forms $u$. The functions $f$, $f_0$ are called the independent and dependent rate, respectively. 
 		\end{definition}
 		This definition is inspired by the  $P$-property defined by Catlin in \cite{Cat84, Cat87, KhZa10} and Property $(CR\T-P_q)$ defined by Raich in \cite{Rai10, KhPiZa12a, Str12}.
 		The $\sigma$-superlogarithmic property is a combination of the potential theoretical conditions  $(f\T-Id)$, $(f_0\ln\T-\di_b\sigma\we\dib_b\sigma)$ and $(f_0\sqrt{\ln}\T-\L_\sigma)$-properties where the $(f\T-\M)$-property is defined by the first line of \eqref{definition1} replaced by 
 			\begin{eqnarray}\label{definition2}
 			\left\la(\L_{\lambda_t}+td\theta)\lrcorner u, u\right\ra\ge f^2(t)|\la \M\lrcorner u,u\ra|.
 				\end{eqnarray}
Here, $\M$ is a $(1,1)$ form and is called a multiplier of the $f$-property.  It is automatic that $(t^{1/2}\T-d\theta)$-property holds on any CR manifold.  A byproduct from  the proof of the main result in \S \ref{s3.1}, we have a general estimate for $\Box_b$ by the $(f\T-\M)$-property.
 				\begin{theorem}\label{t0}
 					Let $M$ be a $(2n+1)$-dimensional, pseudoconvex CR manifold and $x_o\in M$, let $f:\R^+\to\R^+$, and $\M$ be a nonnegative $(1,1)$-form. Suppose  $(f\T-\M\T)$-property on $(0,q_0)$-forms holds  at in a neighborhood $U$ of  $x_o$. Then, the estimate
 					\begin{eqnarray}
 					\label{fM+}(\M \lrcorner f(\La)u^+,f(\La)u^+)_{L^2}\le c\left(\no{\dib_b u^+}_{L^2}^2+\no{\dib_b^* u^+}_{L^2}^2+\no{u}^2_{L^2}\right)
 					\end{eqnarray}
 					holds for all $u\in C^\infty_{0,q}(M)$ supported in $U$ with $q\ge q_0$; and dually, the estimate
 					\begin{eqnarray}	\label{fM-}\begin{aligned}\no{\sqrt{\Tr(\M)}\times f(\La)u^-}_{L^2}^2&-(\M\lrcorner f(\La)u^-,f(\La)u^-)_{L^2}\\
 					&\le  c\left(\no{\dib_b u^-}_{L^2}^2+\no{\dib_b^* u^-}_{L^2}^2+\no{ u}_{L^2}^2\right)\end{aligned}	\end{eqnarray}
 					holds for all  $u\in C^\infty_{0,q}(M)$ supported in $U$ with $q\le n-q_0$. Here $u^\pm$ are the $\pm$-Kohn's microlocalizations of $u$ and $f(\La)$ is the pseudodifferential operator of symbol $f(\sqrt{1+|\xi|^2})$. 
 				\end{theorem}
 				We remark that if $\M$ is a positive $(1,1)$-form, then \eqref{fM+}, \eqref{fM-} and the elliptic estimate for the $0$-Kohn's microlocalization $u^0$ imply the full estimate
 				\begin{eqnarray}
 				\label{fM}\no{f(\La)u}^2_{L^2}\le c\left(\no{\dib_b u}_{L^2}^2+\no{\dib_b^* u}_{L^2}^2+\no{u}_{L^2}^2\right)\end{eqnarray}
 				for all $u \in C^\infty_{0,q}(M)$ supported in  $U$ with $q_0\le q\le n-q_0$. In this situation, if  $f(t)=t^{\epsilon}$, \eqref{fM} becomes a subelliptic estimate; if $f\gg\ln$, \eqref{fM} implies the superlogarithmic estimate; and if $f\gg1$, \eqref{fM} implies the (local) compactness estimate. These estimates are well-known in theory of hypoellipticity of $\Box_b$. The compactness estimate for $\Box_b$ obtained by potential-theoretical property has been proved in \cite{Rai10, Str12, KhPiZa12a}. However, this is the first time that estimates for $\Box_b$ by potential-theoretical properties have shown a ``gain" in derivatives.
 				Coming back to Definition~\eqref{def1}, it is obvious  that if the independent rate $f\gg\ln$ then the $\sigma$-superlogarithmic property holds for any cutoff function $\sigma$ in $U$. In fact, in this case the superlogarithmic estimate for $\Box_b$ holds by Theorem~\ref{t0} and hence $\Box_b$ is local hypoellipticity by  \cite[Theorem 1.5]{Koh02}.\\

 	The $\sigma$-superlogarithmic property implies the following results concerning  hypoellipticity of $\Box_b$ and regularity of $G_q$ and its related operators on CR manifolds with dimension at least five.

 		 		\bt
 		 \label{maintheorem1}
 		Let $M$ be a pseudoconvex CR manifold of dimension $(2n+1)$ with $n\ge 2$ 
 		such that  $\dib_b$ has closed range in $L^2$ spaces for all degrees of forms. 
 		  Suppose that  the $\sigma$-superlogarithmic property  on $(0,q_0)$-forms holds. 	Then, for any $q_0\le q\le n-q_0$, if  $u\in L^2_{0,q}(M)$ such that $\Box_b u\in L^2_{0,q}(M)\cap C^\infty_{0,q}(V_1^\sigma)$ then  $u\in C^\infty_{0,q}(V_0^\sigma) $, where $V_0^\sigma$ the interior of $\{x\in M: \sigma(x)=1\}$ and $V_1^\sigma$ be an open set containing $\supp(\sigma)$. 
 		  
 		  Furthermore, for any $s\ge0$, and $\zeta_0\prec\sigma\prec \zeta_1$, the following holds.
 			\begin{enumerate}
 			\item[(i)] If  $\varphi\in L^2_{0,q}(M)\cap H_{0,q}^s(\supp(\zeta_1))$, then 
 			 	\begin{eqnarray}
 			 	\label{main-est1}
 			 	\begin{aligned}
 			 	&\no{f^2(\Lambda)\zeta_0G_q\varphi}_{H^{s}}+\no{f(\Lambda)\zeta_0\dib_bG_q\varphi}_{H^{s}}+\no{f(\Lambda)\zeta_0\dib_b^*G_q\varphi}_{H^{s}}\\
 			 	&+\no{\zeta_0(I-\dib_b\dib_b^*G_q)\varphi}_{H^{s}}+\no{\zeta_0(I-\dib_b^*\dib_bG_q)\varphi}_{H^{s}}
 			 	\le c_{s,\zeta_0,\zeta_1}\left( \no{ \zeta_1 \varphi}_{H^{s}}+\no{\varphi}_{L^2}\right).
 			 	\end{aligned}
 			 	\end{eqnarray}
 				\item[(ii)] If $\varphi\in L^2_{q-1}(M)\cap H^s_{0,q-1}(\supp(\zeta_1))$, then
 				\begin{eqnarray}
 				\label{main-est2}
 				\begin{aligned}
  \no{f^2(\Lambda)\zeta_0\dib_b^*G_q^2\dib_b\varphi}_{H^{s}}+&\no{f(\Lambda)\zeta_0G_{q}\dib_b\varphi}_{H^{s}} 				+\no{\zeta_0(I-\dib_b^*G_{q}\dib_b)\varphi}_{H^{s}}\\
 				&\le c_{s,\zeta_0,\zeta_1}\left( \no{ \zeta_1 \varphi}_{H^{s}}+\no{\varphi}_{L^2}\right).
 				\end{aligned}
 				\end{eqnarray}
 			\item[(iii)] If  $\varphi\in L^2_{0,q+1}(M)\cap H^s_{0,q
 				+1}(\supp(\zeta_1))$, 
 			\begin{eqnarray}
 			\label{main-est3}
 			\begin{aligned}
 		 \no{f^2(\Lambda)\zeta_0\dib_bG_q^2\dib^*_b\varphi}_{H^{s}}+&\no{f(\Lambda)\zeta_0G_{q}\dib_b^*\varphi}_{H^{s}} 				+\no{\zeta_0(I-\dib_bG_{q}\dib_b^*)\varphi}_{H^{s}}\\
 			&\le c_{s,\zeta_0,\zeta_1}\left( \no{ \zeta_1 \varphi}_{H^{s}}+\no{\varphi}_{L^2}\right).
 			\end{aligned}
 			\end{eqnarray}
 			\end{enumerate} 	
 				Here, $f$ is the independent rate of the $\sigma$-superlogarithmic property.	
 		 \et

The proof of Theorem~\ref{maintheorem1} is given in \S\ref{s4.1}. The crucial idea in the proof is to combine the elliptic regularization method and a-priori estimates in the intermediate norm $\no{\cdot}_{f\T- {\mathcal A}^{s\sigma}_\zeta}$ between two classical norms $\no{f(\Lambda)\zeta_0\cdot}_{H^{s}}$ and $\no{f(\Lambda)\zeta_1\cdot}_{H^{s}}$. In particular, we will prove that 
 		\begin{eqnarray*}
 	\begin{aligned}
 	\no{u}_{f\T- {\mathcal A}^{s\sigma}_\zeta}\le c_{s,\sigma,\zeta}\left(\no{\dib_bu}_{{\mathcal A}^{s\sigma}_\zeta}+\no{\dib_b^*u}_{{\mathcal A}^{s\sigma}_\zeta}+\no{u}_{L^2}\right),
 	\end{aligned}
 	\end{eqnarray*}
 	and
 		\begin{eqnarray*}
 			\begin{aligned}
 				\no{u}_{f^2\T- {\mathcal A}^{s\sigma}_\zeta}+&\no{\dib_bu}_{f\T- {\mathcal A}^{s\sigma}_\zeta}+\no{\dib_b^*u}_{f\T- {\mathcal A}^{s\sigma}_\zeta}+\no{\dib_b\dib_b^*u}_{{\mathcal A}^{s\sigma}_\zeta}+\no{\dib_b^*\dib_bu}_{{\mathcal A}^{s\sigma}_\zeta}\\
 				&\le c_{s,\sigma,\zeta}\left( \no{\Box^\delta_b u}_{{\mathcal A}^{s\sigma}_\zeta}+\no{u}_{L^2}\right)
 			\end{aligned}
 		\end{eqnarray*}
 		hold for all   $u\in L^2_{0,q}(M)\cap C_{0,q}^\infty(\supp(\zeta)) $. Here $\Box_b^\delta$ is an elliptic perturbation of $\Box_b$.  It is should be emphasized that with the stronger hypotheses in \cite{Chr01, Chr02, Koh02, Koh00, BaPiZa15} the difficulty of regularization methods can be avoided.
 	
 			\begin{remark}For the top degrees,  it is known that $G_0:=\dib_b^*G_1^2\dib_b$ and   $G_n=\dib_bG_{n-1}^2\dib_b^*$. Thus, if the $\sigma$-superlogarithmic property holds on $(0,1)$-forms then by Theorem~\ref{maintheorem1}
 				(ii)  and  (iii) we get 
 			 				\begin{eqnarray}
 				\label{main-est4}
 				\begin{aligned}
 				\no{f^2(\Lambda)\zeta_0G_{0}\varphi}_{H^{s}}+\no{f(\Lambda)\zeta_0\dib_bG_0\varphi}_{H^{s}}
 				\le c_{s,\zeta_0,\zeta_1}\left( \no{ \zeta_1 \varphi}_{H^{s}}+\no{\varphi}_{L^2}\right).
 				\end{aligned}
 				\end{eqnarray}
 			holds 	for all $\varphi\in L^2(M)\cap H^s(\supp(\zeta_1))$, and 
 				\begin{eqnarray}
 				\label{main-est5}
 				\begin{aligned}
 				\no{f^2(\Lambda)\zeta_0G_{n}\varphi}_{H^{s}}+\no{f(\Lambda)\zeta_0\dib^*_bG_n\varphi}_{H^{s}}\le c_{s,\zeta_0,\zeta_1}\left( \no{ \zeta_1 \varphi}_{H^{s}}+\no{\varphi}_{L^2}\right).
 				\end{aligned}
 				\end{eqnarray}
 					holds 	for all $\varphi\in L_{0,n}^2(M)\cap H_{0,n}^s(\supp(\zeta_1))$. However, this works only for CR manifolds of  dimension at least five, i.e, $n\ge 2$. For the case $n=1$, we add an extra assumption that $G_0$ is $C^\infty$ 
					globally regular.  
 				Here is our result for this case.
 			\end{remark}	
 		\begin{theorem}\label{maintheorem2}
 			Let $M$ be a $3$-dimensional pseudoconvex CR manifold such that $\dib_b$ has closed range on functions and $G_0$ is globally regular. Suppose that 
			the $\sigma$-superlogarithmic property with dependent rate $f$ holds for $(0,1)$-forms. Then, for $q=0,1$, if $u\in L^2_{0,q}(M)$ such that $\Box_bu\in L^2_{0,q}(M)\cap C^\infty_{0,q}(V_1^\sigma)$ then  $u\in C^\infty_{0,q}(V_0^\sigma) $.  Furthermore, \eqref{main-est4} and \eqref{main-est5} hold.		
 		\end{theorem}
 		
 	Let $\G_q(\hat x,x)$ be the integral kernel of the complex Green operator $G_q$. This means $G_q$ has the integral representation 
 	$$G_q\varphi(x)=\int_M\la\varphi(\hat x),\G_q(\hat x,x)\ra dV(\hat x)=(\varphi,\G_q(\cdot,x)))$$
 	provided the integral exists; this imposes regularity conditions on $\varphi$ and $\G_q$.
 	 Thus, $G$ is a double differential form of degree $(0,q;q,0)$ on the product manifold $M\times M$.   The second goal of this paper is to study the smoothness of $\mathcal G_{q}$. 
Using the  idea of Kerzman \cite{Ker72} for proving smoothness of the Bergman kernel of strongly pseudoconvex domains in $\C^{n+1}$ , the smoothness of the Szeg\"o kernels and the kernel of the complex Green operator have been obtained on some cases of pseudoconvex hypersurfaces in $\C^{n+1}$ of finite type (see \cite{McNSt97, Chr88, NaRoStWa89, PhSt77, Boa87s, NaSt06, Koe02, FeKoMa90}).  However, when $M$ is of infinite type, the only result available is that of Halfpap, Nagel and Wainger \cite{HaNaWa10} who classified  subsets of  $M\times M$ in which the Szeg\"o kernel is smooth by comparing  $\alpha$ with $1$  for the case that $M$ is defined by $\Im\,{z_2}=\exp\left(-\frac{1}{|\Re z_1|^\alpha}\right)$. Halfpap et. al. use a harmonic analysis approach.

It is obvious that the smoothness of $\G_q$ implies the smoothness of the Szeg\"o kernels of $I-\dib_b^*\dib_bG_q$ and $I-\dib_b^*G_q\dib_b$ and also the kernels of all related operators such as $\dib_bG_q$, $G_q\dib_b$,  $\dib_b^*G_q$, $G_q\dib_b^*$, $I-\dib_b\dib_b^*G_q$, and $I-\dib_bG_q\dib_b^*$. So we only need to state the result on the smoothness of $G_q$.
	\begin{theorem}\label{maintheorem3} Let $M$ be a pseudoconvex, CR manifold of dimension $(2n+1)$ such that $\dib_b$ has closed range in $L^2$ spaces for all degrees of forms. Suppose that the $\sigma$-superlogarithmic property on $(0,q_0)$ forms holds .
	\begin{enumerate}
		\item[(i)] If $\hat\sigma$ is another cutoff function  such that $\supp(\sigma)\cap\sup(\hat\sigma)=\emptyset$ and the $\hat\sigma$ -superlogarithmic property on $(0,q_0)$-forms also holds.  Then 
		$$\mathcal G_{q}\in C^\infty_{0q;q,0}\left(\left( V_0^\sigma\times V_0^{\hat\sigma}\right)\cup \left( V_0^{\hat\sigma}\times V_0^{\sigma}\right)\right)\quad\T{	 for all $q_0\le q\le n-q_0$}.$$	If $q_0=1$ and $n\ge 2$, or $q_0=n=1$ with the extra assumption that $G_0$ is globally regular, then this conclusion also holds for $q=0$ and $q=n$. 
			\item[(ii)] If $G_q$ is globally regular for some $q_0\le q\le n-q_0$ then 
			$$\mathcal G_{q}\in C^\infty_{0,q;q,0}\left(\left(V^\sigma_0\times (M\setminus \overline{\supp(\sigma)})\right)\cup\left(  (M\setminus \overline{\supp(\sigma)})\times V_0^\sigma \right)\right).$$	If $q_0=q=1$ (resp. $q_0=1, q=n$), this conclusion also holds for $G_0$ and (resp. $G_n$).
 			\end{enumerate}	
 		\end{theorem}
 		We prove this theorem in \S\ref{s4.2}. The idea of the proof is based on the estimates of regularity for $G_q\varphi_{\hat x_0}$ and $G_q^\delta\varphi_{\hat x_0}$ (where $G^\delta_q$ is the inverse of $\Box_b^\delta$-- a elliptic perturbation of $\Box_b$ ) where $\varphi_{\hat x_0}$ is a Dirac-delta distribution $(0,q)$-form supported in $\hat x_0$. 
 		
 		We notice that the regularity of $G_q$ and the smoothness of $\G_q$ in Theorem~\ref{maintheorem1}, \ref{maintheorem2} and \ref{maintheorem3} depends on the set  $V_0^\sigma=\{x\in M: \sigma(x)=1\}$ and the support of $\sigma$. To give the precise result for local regularity defined above, we define the point of weak superlogarithmicity as follows.
 	
 	 \begin{definition}\label{d2}
 	 	The point $x_0\in M$ is called a weakly superlogarithmic point if given any  small neighborhood $W$ of  $x_0$ then  there exists a  cutoff function $\sigma$  such that 
 	 	\begin{enumerate}
 	 		\item[(i)] $\supp(\sigma)\subset W$;
 	 		\item[(ii)] $\sigma=1$ in an open set containing $x_o$;
 	 	\item[(iii)]	the $\sigma$-superlogarithmic property  on $(0,1)$-forms holds.
 	 	\end{enumerate}
 	 \end{definition}	
   Denote by $S$ the set of all weakly superlogarithmic points. If $S\subset\subset M$ then $S$ is open. The following corollary follows immediately from
    Theorems~\ref{maintheorem1}, \ref{maintheorem2} and \ref{maintheorem3}.
  \begin{corollary}\label{Cor1}
  	Let $M$ be a pseudoconvex, CR manifold of dimension $(2n+1)$ such that $\dib_b$ has closed range in $L^2$ spaces for all degrees of forms (here $n\ge 2$ or in the case $n = 1$ with the additional hypothesis that $G_0$ is globally regular). Let
  	$S$ be the set of all weakly superlogarithmic points in $M$.  Then,  the following properties hold:
  	\begin{enumerate}
  		\item[(i)] For all $0\le q\le n$, $G_q$ is locally regular in $S$ in the sense that for any open set $V\subset S$, if $\varphi\in L^2_{0,q}(M)\cap C^\infty_{0,q}(V)$ then $G_q\varphi\in C^\infty_{0,q}(V)$. 
  		\item[(ii)] The kernel $\mathcal G_q\in C^\infty_{0,q;q,0}\left(S\times S\setminus \{\T{Diagonal}\}\right)$ for all $0\le q\le n$.
  		\item[(iii)] If $G_q$ is global regular for some $q$, then $\mathcal G_q\in C^\infty_{0,q;q,0}\left((S\times M)\cup(M\times S)\setminus\{\T{Diagonal}\}\right)$. 
  	\end{enumerate}
  \end{corollary}
As our motivation, we will show any point of the hypersurface defined by \eqref{model5} is a weakly superlogarithmic point and then $\Box_b$ is locally hyopelliptic. 
 	\begin{theorem}\label{t1.9}
 	Let $M^{2n+1}$ be the hypersurface defined by \eqref{model5}. If 
 	$$\lim_{\delta\to 0}\delta \ln\left( F_k(\delta^2)\right)=0\quad  \T{for all  }k,$$ then any point $x\in M$ is a weakly superlogarithmic point. Consequently, $\Box_b$'s defined on forms of any degree are local hypoelliptic.  
 	\end{theorem}  
 The hypoellipticity of $\Box_b$ acting on functions implies the hypoellipticity of the sub-laplacian $\Delta_b:=\Re(\Box_b)$. The study of $\Delta_b$ is related to the work on degenerate second-order operators. By the celebrated theorem of H\"ormander \cite{Hor67} the finite type 
 condition of $\{X_0,\dots, X_r\}$ is sufficient for  hypoellipticity of  $\sum_{j=1}^rX^2_1+X_0+c$ in $\R^N$ (see also in \cite{FoKo72}, \cite{RoSt76}). However, the finite type  is not a necessary condition for local hypoellipticity (see \cite{Fed71}, \cite{Koh98}, \cite{BeMo95}, \cite{KuSt87}, \cite{Mor87}, \cite{Him95}).  For example, Fedii \cite{Fed71} showed that  the operator $\frac{\di^2}{\di x^2}+b(x)\frac{\di^2}{\di t^2}$ in $\R^2$ with $b\in C^\infty$ is hypoelliptic if $b(x)>0$ when $x\not=0$; or the example of Kusuoka-Stroock \cite{KuSt87} illustrates that  the operator $\frac{\di^2}{\di x^2}+\frac{\di^2}{\di y^2}+a^2(x)\frac{\di^2}{\di t^2}$ in $\R^3$ is hypoelliptic if and only if $\lim_{x\to 0}x\ln a(x)=0$, where $a\in C^\infty$, $a(x)\not=0$ when $x\not=0$. Applying Theorem~\ref{t1.9}  to the model $\Im  z_2=H(|z_1|^2)F(|\Re\,z_1|^2)$ in $\C^2$, we obtain a generalized version of Fedii and Kusuoka-Stroockand's examples. In fact, we can conclude that the operator $\frac{\di^2}{\di x^2}+\frac{\di^2}{\di y^2}+a^2(x)b(x^2+y^2)\frac{\di^2}{\di t^2}$ in $\R^3$ is hypoelliptic if $\lim_{x\to 0}x\ln a(x)=0$.\\

 The rest of the paper is organized as follows. The establishment of a basic estimate and the development of the microlocal are given in the first and second parts of Section \ref{s2}. The remainder of the section introduces a special norm between the classical $H^s$-norms. In Section~\ref{s3}, we establish the a-priori estimates for the system $(\dib_b,\dib_b^*)$, the Kohn-Laplacian $\Box_b$ and the elliptic perturbation $\Box_b^\delta$. We prove Theorems \ref{maintheorem1}, \ref{maintheorem2} and \ref{maintheorem3} in Section \ref{s4}. In Section~\ref{s5}, we analyze the hypersurface model \ref{model5} and prove Theorem~\ref{t1.9}.

 \subsection*{Acknowledgements} I am grateful to Joseph J. Kohn for his scientific discoveries. This research is inspired from his work. 
I wish to express my thanks to Luca Baracco and Stefano Pinton, communication with them
 led to the research presented here. I also thank Andy Raich 
 for going through the original manuscript and suggesting several improvements
 and clarifications.
	\section{Technical preliminaries}
 		\label{s2}
 		In this section we prepare some inequalities which are needed for 
		the {\it a-priori} estimates that prove Theorems~\ref{maintheorem1} and ~\ref{maintheorem2}.  The key technical tool of our discussion 
		is the Kohn-Morrey-H\"ormander inequality on abstract CR manifolds, the Kohn-microlocalizations and a special norm between $\no{\zeta_0\cdot}^2_{H^s}$ and $\no{\zeta_1\cdot}^2_{H^s}$.
 		\subsection{The Kohn-Morrey-H\"ormander inequality on CR manifolds}\label{s2.1}
 		Let $M$ be a smooth CR manifold of dimension $(2n+1)$ with the CR structure $T^{1,0}M$ and contact form $\theta$. Let $U$ be an open set of $M$ so that we can choose a $C^\infty(U)$ local basis $\{L_1,\dots,L_n,\bar L_1,\dots,\bar L_n,T\}$ of $\C TM$. Here,  $\{L_1, \dots, L_n\}$ is a basis of $T^{1,0}M$, $\{\bar L_1, \dots, \bar L_n\}$ is the basis of the conjugate of $T^{1,0}M$, denoted by $T^{0,1}M$. The vector field
		 $T$ is dual to $\theta$. Then the coefficients of the Levi form can be written as $c_{ij}=d\theta(L_i\we\bar L_j)$, and by the Cartan formula it is identified with the $T$-component of $[L_i, \bar L_j]$. 
 		For a smooth function $\lambda$, the matrix $(\lambda_{ij})$ of the Hermitian form $\L_\lambda:=\frac{1}{2}\big(\di_b\dib_b-\dib_b\di_b\big)\lambda$ is expressed by 
 		\begin{equation}\label{phiij}
 			\begin{aligned}
 				\lambda_{ij}=\bar L_jL_i(\lambda)+\frac12c_{ij}T(\lambda)+\sum_kc_{ij}^kL_k(\lambda).
 			\end{aligned}
 		\end{equation}
 		where $c^k_{ij}$ is the $L_k$-component of $[L_i,\bar L_j]$ (see \cite{KhRa16a} for more details).
 		
 		To express a form in  local coordinates, we let
 		$\I_q = \{(j_1,\dots,j_q) \in \mathcal N^q : 1 \leq j_1 < \cdots < j_q \leq n\}$, and for $J\in \I_q$, $I\in\I_{q-1}$, and $j\in\mathbb N$, $\eps^{jI}_J$ be the sign of the permutation $\{j,I\}\to J$
 		if $\{j\} \cup I = J$ as sets, and $0$ otherwise.
 		If $u\in C_{0,q}^\infty(M)$, then $u$ is locally expressed as a combination
 		\begin{eqnarray}
 			\label{form:u}
 			u= \sum_{J\in\I_q} u_J\, \bar\omega_J,
 		\end{eqnarray}
 		of basis forms $\bar\omega_J=\bar\omega_1\wedge...\wedge\bar\omega_{j_k}$ for ordered indices $j_1<...<j_k$ with $C^\infty$-coefficients $u_J$.
 		If $\lambda=\sum_{j}\lambda_j\om_j$ is a $(1,0)$ form defined in $U$, the contraction operator $\lrcorner$ acting $(0,q)$-form $u$ as in \eqref{form:u} is defined by
 		\begin{eqnarray}
 			\label{contraction01}\lambda\lrcorner u=\sum_{I\in\I_{q-1}} \left(\sum_{j=1}^n \lambda_j u_{jI}\right)\bom_I,		
 		\end{eqnarray}
 		and consequently
 		 	\begin{eqnarray}
 	\no{\lambda\lrcorner u}^2=\sum_{I\in\I_{q-1}} \no{\sum_{j=1}^n \lambda_j u_{jI}}^2,		
 		 	\end{eqnarray}
 			If $\lambda=\sum_{ij}\lambda_{ij}\om_i\we \bom_j$ is a $(1,1)$ form in $U$, the contraction operator $\lrcorner$ acting a $(0,q)$-form $u$ as in \eqref{form:u} is defined by
 			\begin{eqnarray}
 			\label{contraction11}\lambda\lrcorner u=\sum_{I\in\I_{q-1}} \left(\sum_{i,j=1}^n \lambda_{ij} u_{iI}\right)\bom_j\we \bom_I.	\end{eqnarray}
 		Thus, the Levi forms $d\theta$ and $\L_\lambda$ acting on $(0,q)$-forms $u,v$ defined in $U$ can be expressed as
 		$$\la d\theta\lrcorner u,v\ra =\sum_{I\in\I_{q-1}}\sum_{i,j=1}^nc_{ij}u_{iI}\overline{v_{jI}}\quad \T{and } \la\L_\lambda\lrcorner u,v\ra=\sum_{I\in\I_{q-1}}\sum_{i,j=1}^n\lambda_{ij}u_{iI}\overline{v_{jI}}.$$
We can also express the operator  $\dib_b : C^\infty_{0,q}(M)\to  C^\infty_{0,q+1}(M)$ and $\dib^*_b :  C^\infty_{0,q}(M)\to  C^\infty_{0,q-1}(M)$ in the local basis as follows:
 		\begin{align}\label{dib}
 			\dib_b u&=  \sum_{\atopp{J\in\I_q}{K\in\I_{q+1}}} \sum_{k=1}^n \eps^{kJ}_K \bar L_k u_J\, \bom_K + \sum_{\atopp{J\in\I_q}{K\in\I_{q+1}}} b_{JK} u_J\, \bom_K 
 		\end{align}
 		and
 		\begin{eqnarray}\label{dib*}
 			\dib_b^*u=- \sum_{J\in\I_{q-1}}\left(\sum_{j=1}^n L_ju_{jJ}+ \sum_{K\in\I_{q}}a_{JK}u_K\right)\bom_J
 		\end{eqnarray}
 		where $a^J_{ijI}, a_{JK}\in C^\infty(M)$ and $u$ is as \eqref{form:u}.
 		
 		We will use the the Kohn-Morrey-H\"ormander  inequality for CR manifolds  with the weighted norm $\no{\cdot}_{L^2_\lambda}$ defined by
 		$$\no{u}_{L^2_\lambda}^2=\int_M\la u, u\ra e^{-\lambda}\,dV.$$
 		Let $\dib^*_{b,\lambda}$ be the $L^2_\lambda$-adjoint  of $\dib_b$. It is easy to see that in the local basis
 		\begin{equation}\label{dib*phi}
 			\dib_{b,\lambda}^*u=- \sum_{J\in\I_{q-1}}\left(\sum_{j=1}^n L_j^\lambda u_{jK}+ \sum_{K\in\I_{q}}a_{JK}u_K\right)\bom_J
 		\end{equation}
 		where $L^\lambda_j \varphi:=e^\lambda L_j(e^{-\lambda}\varphi)$ and $u$ as in \eqref{form:u}. For such $u$, 
 		$$
 		\di_b(\lambda)\lrcorner u= - [\dib_b^*,\lambda]u=[\dbarb,\lambda]^*u=\sum_{I \in \I_{q-1}} \sum_{j=1}^n L_j(\lambda)u_{jI}\bom_I,
 		$$
 		and hence
 		\begin{equation}
 			\label{dbar*}
 			\dib_{b,\lambda}^*u=\dib^*_bu-\di_b\lambda\lrcorner u.
 		\end{equation}
 		Now we can state the Kohn-Morrey-H\"ormander inequality type for CR manifold which has been proven  in \cite{KhRa16a}.
 		\begin{proposition}\label{kmh} Let $\lambda$ be a real-valued $C^2$ function defined on $U$. Then there exists a constant $c$  independent  of $\lambda$ such that 
 			
 			\begin{eqnarray}\label{KMH1}
 				\begin{aligned}
 					\no{\bar{\partial}_b u}^2_{L^2_\lambda}+\no{\bar{\partial}^*_{b,\lambda} u}^2_{L^2_\lambda}+&c\no{u}^2_{L^2_\lambda} \geq \frac{1}{2}\sum^{n}_{j=1}\no{\bar{L}_ju}^2_{L^2_\lambda}\\
 					&
 					+ (\L_\lambda\lrcorner u,u)_{L^2_\lambda}+\Re\left\{(d\theta  \lrcorner Tu,u)_{L^2_\lambda}\right\} 
 				\end{aligned}
 			\end{eqnarray}
 			holds for all $u\in C^{\infty}_{0,q}(M)$ supported in $U$.
 		
 		\end{proposition}
 		To estimate the the Levi form containing $T$ in  the inequality \eqref{KMH1}, we use Kohn's microlocalizations and G{\aa}rding's inequality. G{\aa}rding's inequality is formulated as follows (see \cite{Nic06}).

 		\begin{lemma}\label{Garding} Let $R$ be a first-order pseudodifferential operator such that the symbol $\mathcal S(R)\ge 0$  and let $\lambda$ be a nonnegative $(1,1)$-form. Then there exists a constant $c$ such that 
 			$$ \Re\left\{\left(\lambda\lrcorner R u,u\right)_{L^2}\right\} \ge -c\no{u}_{L^2}^2$$
 			for any $u\in C^\infty_{0,q}(M)$.
 		\end{lemma}	
 		\subsection{Microlocalization}\label{s2.2}
 		We now recall the microlocal framework developed
 		by Kohn \cite{Koh86, Koh02} and present its main consequences for our work.
 		
 		We choose real coordinates $x=(x',x_{2n+1})=(x_1,...x_{2n},x_{2n+1})$ with the origin some point $x_o\in U$ such that $T=-\sqrt{-1}\frac{\di}{\di x_{2n+1}}$ and if we set $z_j=x_j+\sqrt{-1}x_{j+n}$ for $j=1,\dots,n$, then we have $L_j|_{x_o}=\frac{\di}{\di z_j}$ for $j=1,\dots,n$.  Let $\xi=$ $(\xi_1,...,\xi_{2n+1})$ $=(\xi',\xi_{2n+1})$ be the dual coordinates to the $(x_1,...,x_{2n+1})$'s. Let $\psi$ be a smooth function in the unit sphere $\{|\xi|=1\}$ with range in $[0,1]$, such that 
 		$$ \psi=1 \T{~~in~~ } \{\xi \big |\xi_{2n+1}> \dfrac{1}{3} |\xi'| \}\quad\T{ and}\quad \supp(\psi)\subset \subset \{\zeta_{2n+1}\ge \frac{1}{4}|\xi'|\}.$$ 
 		We extend this function to $\R^{2n+1}$ by homogeneity. Set
		$\psi(\xi):=\psi(\frac{\xi}{|\xi|})$ if $|\xi|\ge 2$ and  $\psi(\xi)=0$ if $|\xi|\le 1$. Set $\psi^+(\xi):=\psi(\xi)$; $\psi^-(\xi):=\psi(-\xi)$ and $\psi^0(\xi):=1-\psi^+(\xi)-\psi^-(\xi)$. Define $\tilde\psi^0$ so that  
 		$\tilde\psi^0=1$ on a neighborhood of $\T{supp}\psi^0\cup \T{supp}(d\psi^+) \cup \T{supp}(d\psi^-)$ and $\supp(\tilde{\psi}^0)\subset \mathcal C^0:=\{\zeta_{2n+1}\le \frac{3}{4}|\zeta'|\}\cup \{|\xi|<2\}$.\\

 		Associated to a function $\psi$ is a pseudodifferential operator $\Psi$ with symbol $\mathcal S(\Psi)=\psi$ whose action on a function $v$ supported in $U$ is defined by
 		$$\widehat{\Psi v}(\xi)=\psi(\xi)\hat{v}(\xi),$$
 		where $\hat{}$ denotes the Fourier transform. 
 		We say that a pseudodifferential operator $\Psi$ is dominated by a pseudodifferential operator $\tilde\Psi$ and denote by $\Psi\prec\tilde\Psi$ if  $\mathcal S(\Psi)\prec \mathcal S(\tilde \Psi)$.
 		The operators $\Psi^+$, $\Psi^-$, $\Psi^0$ and $\tilde\Psi^0$ are defined as above with symbols $\psi^+$, $\psi^-$, $\psi^0$ and $\tilde\psi^0$, respectively. 	
 		By the definition of $\psi^+, \psi^-, \psi^0$ and $\tilde\psi^0$, it follows that
 		$\Psi^++\Psi^-+\Psi^0=I$, and both operators
 		$[\Psi^{\underset{0}{\pm}},\alpha]$ and $\Psi^0$  are dominated by $\tilde\Psi^0$, where $\alpha\in C^\infty(M)$.

	Let $\gamma:\R\to [0,1]$ be a cutoff function such that 
	$$\gamma(t)=\begin{cases}
	1& \T{if  } |t|\in[e^{2},e^3],\\
	0& \T{if  } |t|\in[0,e]\cup[e^4,+\infty).
	\end{cases}$$
	Set $\gamma_k(\xi):=\gamma(e^{-k}|\xi_{2n+1}|)$ for $k=1,2,\cdots$, and also let $\gamma_0(\zeta)=1$ if $|\zeta_{2n+1}|\le e$ and is zero if $|\zeta_{2n+1}|\ge e^2$. Furthermore, we can choose $\gamma$ and $\gamma_0$ such that $\sum_{k=0}^\infty\gamma_k(\xi) =1$ for any $\xi\in\R^{2n+1}$.
	 For $k=0,1,2,\dots$, denote by $\Gamma_k$ the pseudodifferential operator with symbol $\gamma_k$. Let $\zeta$ be a cutoff function in $U$ and set 
	$$u^+_{\zeta,k}:=\zeta \Gamma_k\Psi^+\zeta u, \quad\T{and}\quad u^-_{\zeta,k}:=\zeta \Gamma_k\Psi^-\zeta u,$$
	where $u$ is a distribution function on $M$.
	\begin{remark}\label{1r1}
For $u\in H^{-s_o}(M)$, Plancherel's theorem gives us
	\[	\begin{aligned}\no{u^{\pm}_{\zeta,k}}^2_{H^s}\le& c\int_{\xi\in\R^{2n+1}}(1+|\xi|^2)^s\left(\gamma(e^{-k}|\xi_{2n+1}|)\psi^\pm(\xi)\right)^2|\widehat{\zeta u}(\xi)|^2 d\xi\\
	\le &c_{k,s,r}\int_{\xi\in\R^{2n+1}}(1+|\xi|^2)^{-s_o}|\widehat{\zeta u}(\xi)|^2 d\xi
	\\
	=&c_{k,s,r}\no{\zeta u}^2_{H^{-r}}\le c_{k,s,s_o}\no{ u}^2_{H^{-s_o}},
	\end{aligned}\]
	 for all $s,r\in \R$ with $r\ge s_o$. Consequently, by the Sobolev lemma, the functions $u_{\zeta,k}^\pm\in C^\infty(M)$ if $u\in H^{-s_o}(M)$. By this remark we will not worry about the regularity of $u^\pm_{\zeta,k}$.
	\end{remark}
We finish this subsection by an estimate for the Levi form involving $T$  on the positive microlocalizations $u^+_{\zeta,k}$. The estimate on the negative microlocalizations $u^-_{\zeta,k}$ will be obtained by a modified Hodge-* constructed in \S\ref{s3} below.
\begin{proposition} 
	\label{Local3.2} Let  $\lambda:=\lambda_{k,s}$ be a $C^\infty$ real-valued  function defined on $U$ such that $|\lambda|\le c_{0,s}k$ and $|T^m(\lambda)|\le c_{m,s}e^{\frac{km}{2}}$ for all $m=1,2,\dots$. Then,  for $u\in H^{-s_o}_{0,q}(M)$, $r\ge s_0$ and $k\ge 1$, we have
	\begin{eqnarray}\label{Mplus}
	\begin{aligned}
	\Re\left(d\theta\lrcorner Tu^+_{\zeta,k},u^+_{\zeta,k}\right)_{L^2_\lambda}
	\ge \frac{1}{2}e^k(d\theta\lrcorner u^+_{\zeta,k},u^+_{\zeta,k})_{L^2_\lambda} -c\no{u^+_{\zeta,k}}^2_{L^2_\lambda}-c_{s,r}e^{-k}\no{\zeta u}_{H^{-r}}^2.
	\end{aligned}
	\end{eqnarray}
	where $c$ and $c_{s,r}$ are independent of $k$.
\end{proposition}
\begin{proof} 
	Since  $T$ is a purely imaginary operator and $\lambda$ is a real-valued function, it follows $ \Re\left\{\left(T(\lambda)d\theta\lrcorner u^+_{\zeta,k},u^+_{\zeta,k}\right)_{L^2_\lambda}\right\}=0$.
	Hence, we can bring $e^{-\frac{\lambda}{2}}$ inside $T$ and rewrite the LHS of \eqref{Mplus} as follows
	\begin{eqnarray}
	\label{m1}\begin{aligned}
	\Re\left\{\left(d\theta\lrcorner Tu^+_{\zeta,k},u^+_{\zeta,k}\right)_{L^2_\lambda}\right\}=&\Re\left\{\left(d\theta\lrcorner Te^{-\frac{\lambda}{2}}u^+_{\zeta,k},e^{-\frac{\lambda}{2}}u^+_{\zeta,k}\right)_{L^2}\right\}\\
	=&\Re\left\{\left(d\theta\lrcorner (T-e^k)e^{-\frac{\lambda}{2}}u^+_{\zeta,k},e^{-\frac{\lambda}{2}}u^+_{\zeta,k}\right)_{L^2}\right\}+e^k(d\theta\lrcorner u^+_{\zeta,k},u^+_{\zeta,k})_{L^2_\lambda}\\
	\end{aligned}	\end{eqnarray}
	To estimate the first term in the second line of \eqref{m1}, we define $\tilde \gamma^+$ a cutoff function  on the real line such that $$\tilde \gamma^+(t)=\begin{cases}1& \T{if } t\in [e,e^{4}],\\0&\T{if } t\in(-\infty,1]\cup[e^5,+\infty).\end{cases}$$  We define $\tilde \gamma_k^+(\xi)=\tilde \gamma^+(e^{-k}\xi_{2n+1})$ then $\tilde\gamma_k^+=1$ on $\T{supp}(\gamma_k\psi^+)$ and denote by $\tilde \Gamma_k^+$ the pseudodifferential operator of symbol $\tilde \gamma_k^+$.  Then we have 
	\begin{eqnarray}
	\begin{aligned}
	(T-e^k)e^{-\frac{\lambda}{2}}u^+_{\zeta,k}=&(T-e^k)e^{-\frac{\lambda}{2}}\zeta \Gamma_k\Psi^+\zeta u\\
	=&(T-e^k)e^{-\frac{\lambda}{2}}\zeta \tilde{\Gamma}_k^+\Gamma_k\Psi^+\zeta u\\
	=&(T-e^k)\tilde{\Gamma}_k^+ e^{-\frac{\lambda}{2}}\zeta \Gamma_k\Psi^+\zeta u+(T-e^k)[e^{-\frac{\lambda}{2}}\zeta,\tilde{\Gamma}_k^+] \Gamma_k\Psi^+\zeta u\\
	=&(T-e^k)\tilde \Gamma_k^+e^{-\frac{\lambda}{2}}u^+_{\zeta,k}+w,\\
	\end{aligned}
	\end{eqnarray}
	where $w:=(T-e^k)[e^{-\frac{\lambda}{2}}\zeta,\tilde{\Gamma_k^+}] \Gamma_k\Psi^+\zeta u$;
	and hence
	\begin{eqnarray}
\begin{aligned}\label{d1}
&	\left(d\theta\lrcorner (T-e^k)e^{-\frac{\lambda}{2}}u^+_{\zeta,k},e^{-\frac{\lambda}{2}}u^+_{\zeta,k}\right)_{L^2}\\&=\left(d\theta\lrcorner(T-e^k)\tilde\Gamma_k^+ e^{-\frac{\lambda}{2}}u^+_{\zeta,k},e^{-\frac{\lambda}{2}}u^+_{\zeta,k}\right)_{L^2}+\left(d\theta\lrcorner w,e^{-\frac{\lambda}{2}}u^+_{\zeta,k}\right)_{L^2}.
	\end{aligned}
\end{eqnarray}
	Since the symbol $\mathcal S\left((T-e^k)\tilde\Gamma_k^+\right)\ge 0$, we apply Lemma~\ref{Garding} for the first Levi form in the second line of \eqref{d1} and get
	\begin{eqnarray}\label{a2}
\Re	\left(d\theta\lrcorner(T-e^k)\tilde\Gamma_k^+ e^{-\frac{\lambda}{2}}u^+_{\zeta,k},e^{-\frac{\lambda}{2}}u^+_{\zeta,k}\right)_{L^2}\ge -c\no{e^{-\frac{\lambda}{2}}u^+_{\zeta,k}}^2_{L^2}=-c\no{u^+_{\zeta,k}}^2_{L^2_\lambda}.
	\end{eqnarray}
Since $d\theta$ is a nonnegative $(1,1)$-form,	we can use the Cauchy-Schwarz inequality for the second Levi form in the second line of \eqref{d1} and get
	\begin{eqnarray}\label{a3}
	\begin{aligned}\left|\Re\left(d\theta\lrcorner w,e^{-\frac{\lambda}{2}}u^+_{\zeta,k}\right)_{L^2}\right|\le &\sqrt{(d\theta\lrcorner w,w)(d\theta\lrcorner e^{-\frac{\lambda}{2}}u^+_{\zeta,k},e^{-\frac{\lambda}{2}}u^+_{\zeta,k})}\\
	\le &\frac{1}{2}e^{-k}\left(d\theta\lrcorner w,w\right)_{L^2}+\frac{1}{2}e^k\left(d\theta\lrcorner e^{-\frac{\lambda}{2}}u^+_{\zeta,k},e^{-\frac{\lambda}{2}}u^+_{\zeta,k}\right)_{L^2}\\
	\le &ce^{-k}\no{w}_{L^2}^2+\frac{1}{2}e^k( d\theta\lrcorner u^+_{\zeta,k},u^+_{\zeta,k})_{L^2_\lambda} \end{aligned}	\end{eqnarray}
	where $c$ is independent of $k$ and $\lambda$.\\
	
	From  \eqref{m1}, \eqref{d1},\eqref{a2} and \eqref{a3}, we obtain  
	\begin{align*}
			\Re\left\{\left(d\theta\lrcorner Tu^+_{\zeta,k},u^+_{\zeta,k}\right)_{L^2_\lambda}\right\}\ge \frac{1}{2}e^k\left(d\theta\lrcorner u^+_{\zeta,k},u^+_{\zeta,k}\right)_{L^2_\lambda}-c\no{u^+_{\zeta,k}}^2_{L^2_\lambda}-ce^{-k}\no{w}^2_{L^2}.
	\end{align*}
	The proof is complete by using Lemma~\ref{l1} in Appendix  to estimate $e^{-k}\no{w}^2$. Indeed,
	\begin{align*}
	e^{-k}\no{w}^2_{L^2}&=\no{(T-e^k)[e^{-\frac{\lambda}{2}}\zeta,\tilde{\Gamma}_k^+] \Gamma_k\Psi^+\zeta u}_{L^2}\\
	&\le 2e^{-k}\no{T[e^{-\frac{\lambda}{2}}\zeta,\tilde{\Gamma}k^+] \Gamma_k\Psi^+\zeta u}^2_{L^2}+2e^k\no{[e^{-\frac{\lambda}{2}}\zeta,\tilde{\Gamma}_k^+] \Gamma_k\Psi^+\zeta u}^2_{L^2} \\
	&\underset{Lemma~\ref{l1}}{\le} c_{s, r}e^{-k}\no{\zeta u}^2_{H^{-r}}.
	\end{align*}
\end{proof}

\subsection{A special norm between $\no{\zeta_0\cdot}_{H^s}$ and $\no{\zeta_1\cdot}_{H^s}$}\label{s2.3}
 Let $\sigma,\zeta$ be  cutoff functions in $U$ such that  $\sigma\prec \zeta$. Recall that  $u^\pm_{\zeta,k}=\zeta\Gamma_k\Psi^\pm\zeta u$. Fix $s_o\ge 0$. 
Let $s\ge0$ and $g:\R^+\to\R^+$ such that $g(t)t^{s_o+2}$ is decreasing and $g(t)t^{r}$ is increasing for some $r\ge s_o+2$. We write the next several results for a generic $g$ but we will use $g(t)=t^{-r}$ or $g(t)=t^{-r}f(t)$ in \S\ref{s3} to establish the a-priori estimates. We define the following spaces and norms on $(0,q)$-forms
$$g\T-{\mathcal A}^{\pm,s\sigma}_{\zeta; 0,q}=\left\{u\in H_{0,q}^{-s_o}(M): \no{u}^2_{g\T-{\mathcal A}_\zeta^{\pm,s\sigma}}:=\sum_{k=1}^\infty g^2(e^k)\no{e^{ks\sigma}u^\pm_{\zeta,k}}^2_{L^2}<\infty\right\}.$$
Be Remark~\ref{1r1}, each atom $g^2(e^k)\no{e^{ks\sigma}u^\pm_{\zeta,k} }^2_{L^2}\le c_{k,s}$, however, the finite of the full norm $\no{u}^2_{g\T-{\mathcal A}_\zeta^{\pm,s\sigma}}$ determines  
the regularity of $u$ in the $H^s$-spaces. In particular,
by  $e^{ks\sigma}\zeta\le e^{ks}$ and $
\sum_{k=0}^\infty \gamma_k\equiv 1$, it follows
$$\no{u}^2_{g\T-{\mathcal A}_\zeta^{\pm,s\sigma}}\le\sum_{k=0}^\infty g^2(e^k)e^{2ks}\no{\Gamma_k\Psi^\pm\zeta u}^2_{L^2}\le c_{s,g} \no{\La^{s}g(\La)\Psi^\pm\zeta u}^2_{L^2}\le c_{s,g}  \no{g(\La)\zeta u}^2_{H^{s}},$$
where the second inequality follows by Plancherel's theorem. 
  			 	 
On the other hand, if $\zeta_0\prec \sigma$  then 
$$
\no{g(\La)\zeta_0 u}^2_{H^{s}}\le c_{s,g} \left( \no{u}^2_{g\T-{\mathcal A}_\zeta^{+,s\sigma}}+\no{u}^2_{g\T-{\mathcal A}_\zeta^{-,s\sigma}}+\no{g(\La)\Psi^0\zeta u}_{H^{s}}^2\right).
$$
Thus, the norm defined by $$\no{\cdot}^2_{g\T-{\mathcal A}_\zeta^{s\sigma}}:=\no{\cdot}^2_{g\T-{\mathcal A}_\zeta^{+,s\sigma}}+\no{\cdot}^2_{g\T-{\mathcal A}_\zeta^{-,s\sigma}}+\no{g(\La)\Psi^0\zeta \cdot}_{H^{s}}^2$$ is an intermediate norm between  $\no{g(\La)\zeta_0 \cdot}^2_{H^{s}}$ and $ \no{g(\La)\zeta \cdot}^2_{H^{s}}$ for any $\zeta_0\prec\sigma\prec\zeta$. Instead of working on the norms of $H^s$ we will work on $\no{\cdot}^2_{g\T-{\mathcal A}_\zeta^{\pm,s\sigma}}$. 
For convenience, we also use the following notation\blue{:}
$$\no{u}^2_{g\T-\di_b\sigma\lrcorner {\mathcal A}_{\zeta}^{+,s\sigma}}:=\sum_{k=1}^\infty  g^2(e^k)\no{ e^{ks\sigma}\di_b\sigma\lrcorner u^+_{\zeta,k}}^2_{L^2},$$		
$$\no{u}^2_{g\T-\dib_b\sigma\we {\mathcal A}_{\zeta}^{-,s\sigma}}:=\sum_{k=1}^\infty  g^2(e^k)\no{e^{k(s\sigma)}\dib_b\sigma\we u^-_{\zeta,k}}^2_{L^2},$$
and 
$$(u,v)_{g\T-{\mathcal A}_\zeta^{\pm,s\sigma}}:=\sum_{k=1}^\infty g^2(e^k)(e^{2ks\sigma}u^\pm_{\zeta,k},v^\pm_{\zeta,k})_{L^2},$$
for $u,v\in g\T-{\mathcal A}^{\pm,s\sigma}_{\zeta;0,q}$.  We also define the error space whose norm appears when we estimate  commutators. Set
$$g\T-{\mathcal E}^{s,b}_{\zeta,\tilde \zeta;0,q}:=\{u\in H^{-s_o}(M): \no{u}^2_{g\T-{\mathcal E}^{s,b}_{\zeta,\tilde\zeta}}:=\no{g(\La)\tilde\zeta u}^2_{H^{b}}+\no{g(\La)\tilde \Psi^0\zeta u}^2_{H^{s+b}}<\infty\},$$
where  $\tilde \zeta\succ \zeta$ and the support of $\mathcal S(\tilde{\Psi}^0)$ belongs to $\mathcal C^0$. If $g(t)=t^a\tilde g(t)$, we write $g\T-{\mathcal A}^{\pm,s\sigma}_{\zeta;0,q}=\tilde g\T-{\mathcal A}^{\pm,s\sigma+a}_{\zeta;0,q}$ and $g\T-{\mathcal E}^{s,b}_{\zeta,\tilde \zeta;0,q}=\tilde g\T-{\mathcal E}^{s,a+b}_{\zeta,\tilde \zeta;0,q}$. 
The error norm appears when we estimate the commutators in the following lemma.
\begin{lemma}\label{com1}
s2	Let $X$ be a differential operator of the first order. For $u\in (g\T-{\mathcal A}^{+,s\sigma}_{\zeta;0,q})\cap (g\T-{\mathcal E}^{s,0}_{\zeta,\tilde \zeta;0,q})$, we have 
	$$\sum_{k=1}^n g^2(e^k)\no{e^{ks\sigma}[\zeta\Gamma_k\Psi^\pm\zeta,X]u}_{L^2}^2\le c_{s,g} \left(\no{u}^2_{g\T-{\mathcal A}^{+,s\sigma}_{\zeta}}+\no{u}^2_{g\T-{\mathcal E}^{s,0}_{\zeta,\tilde\zeta}}\right).$$
\end{lemma}
\begin{proof}
	Multiplying with $g^2(e^k)$ to both sides of the estimate in Lemma~\ref{l2}(i) in Appendix, and then taking the summation over $k$, we get the desired estimate.
\end{proof}
	The combination of Lemma~\ref{com1} and \ref{l2}(ii) in Appendix gives us  
	\[\begin{aligned}
	\no{Xu}^2_{g\T-{\mathcal A}_{\zeta}^{\pm,s\sigma}}\le c_{s,g}  \left( \no{u}^2_{g\T-{\mathcal A}_{\zeta}^{\pm,s\sigma+1}}+\no{u}^2_{g\T-{\mathcal E}_{\zeta,\tilde\zeta}^{s,1}}\right),
	\end{aligned}\]
	for  $u\in (g\T-{\mathcal A}^{\pm,s\sigma+1}_{\zeta;0,q})\cap (g\T-{\mathcal E}^{s,1}_{\zeta,\tilde \zeta;0,q})$. Consequently, we have the following corollary.
	\begin{corollary}\begin{enumerate}
			\item[(i)] 	If $u\in (g\T-{\mathcal A}^{\pm,s\sigma+1}_{\zeta;0,q})\cap (g\T-{\mathcal E}^{s,1}_{\zeta,\tilde \zeta;0,q})$ then  $\dib_bu, \dib_b^*u\in (g\T-{\mathcal A}^{\pm,s\sigma}_{\zeta;0,q})\cap (g\T-{\mathcal E}^{s,0}_{\zeta,\tilde \zeta;0,q})$ .		
			\item[(ii)] If $u\in (g\T-{\mathcal A}^{\pm,s\sigma+2}_{\zeta;0,q})\cap (g\T-{\mathcal E}^{s,2}_{\zeta,\tilde \zeta;0,q})$ then  $\dib_b^*\dib_bu,\dib_b\dib_b^*u, \Box_bu, \Box_b^\delta u\in (g\T-{\mathcal A}^{\pm,s\sigma}_{\zeta;0,q})\cap (g\T-{\mathcal E}^{s,0}_{\zeta,\tilde \zeta;0,q})$, where $\Box_b^\delta$ is an elliptic pertubation of $\Box_b$ defined in \S\ref{s3.3}.
		\end{enumerate}
	
	\end{corollary}

\noindent{\bf Note on notation.} In what follows we use notation $A\lesssim B$ for $A\le cB$ where $c$ is a constant depending on indices and cutoff functions of the space $\no{\cdot}_{g\T-{\mathcal A}^{\pm,s\sigma+a}_{\zeta;0,q}}$ and $\no{\cdot}_{g\T-{\mathcal E}^{s\sigma, a}_{\zeta, \tilde{\zeta};0,q}}$  such as  $s$, $s_o$, $a$, $g$, $\zeta$, $\sigma$ but not on $k$.


 			\section{A-priori estimates} \label{s3}
 			In this section we establish the {\it a-priori} estimates for the system $(\dib_b,\dib_b^*)$, $\Box_b$ and $\Box_b^\delta$ which are needed in the proofs of the main theorems.

 			\subsection{A-priori estimates for the system $(\dib_b,\dib_b^*)$}\label{s3.1}
 			The central of this subsection is to prove the following {\it a-priori} estimates for the system $(\dib_b,\dib_b^*)$ in the special norm  $\no{\cdot }_{g\T-{\mathcal A}^{\pm,s\sigma}_\zeta}$. The proof of Theorem~\ref{t0} is also given at the end of this subsection. 
 			\begin{theorem}\label{t41} Suppose that the $\sigma$-superlogarithmic property  on  $(0,q_0)$-forms holds with the pair of rates $(f,f_0)$.  Then the following holds:
 				\begin{enumerate}
 					\item[(i)] For $q\ge q_0$, we have 
 					\begin{eqnarray}\label{est:t41a}
 					\no{u}^2_{gf\T-{\mathcal A}^{+,s\sigma}_\zeta}+\no{u}^2_{gf_0\ln\T-\di_b\sigma\lrcorner {\mathcal A}^{+,s\sigma}_\zeta}\le  c\left(\no{\dib_bu}^2_{g\T-{\mathcal A}^{+,s\sigma}_\zeta}+\no{\dib_b^*u}^2_{g\T-{\mathcal A}^{+,s\sigma}_\zeta}\right)+c_{s,g}\no{u}^2_{g\T-{\mathcal E}^{s,0}_{\zeta,\tilde\zeta}}.
 					\end{eqnarray}    
 					for all $u\in (g\T-{\mathcal A}^{+,s\sigma+1}_{\zeta;0,q})\cap (g\T-{\mathcal E}^{s,1}_{\zeta,\tilde \zeta;0,q})$.
 			\item[(ii)]
 				For  $q\le n-q_0$, we have 
 				\begin{eqnarray}\label{est:t41b}
 				\no{u}^2_{gf\T-{\mathcal A}^{-,s\sigma}_\zeta}+\no{u}^2_{gf_0\ln\T-\dib_b\sigma\we {\mathcal A}^{-,s\sigma}_\zeta}\le c\left( \no{\dib_bu}^2_{g\T-{\mathcal A}^{-,s\sigma}_\zeta}+\no{\dib_b^*u}^2_{g\T-{\mathcal A}^{-,s\sigma}_\zeta}\right)+c_{s,g}\no{u}^2_{g\T-{\mathcal E}^{s,0}_{\zeta,\tilde\zeta}}.
 				\end{eqnarray}  
 			for all $u\in(g\T-{\mathcal A}^{-,s\sigma+1}_{\zeta;0,q})\cap (g\T-{\mathcal E}^{s,1}_{\zeta,\tilde \zeta;0,q})$.
 				\end{enumerate}       
 			\end{theorem}
 				The proof of this theorem is divided into four steps:
 				\begin{enumerate}
 					\item[Step 1.] We point out that if the $\sigma$-superlogarithmic property holds on $(0,q_0)$ forms then it still holds on $(0,q)$-forms with $q_0\le q\le n$.
 					\item[Step 2.] Using the Kohn-Morrey-H\"ormander inequality 
					 in Proposition~\ref{kmh} and the estimate of the Levi form $(d\theta\lrcorner T\cdot,\cdot)_{L^2_\lambda}$ in Proposition~\ref{Local3.2}, we obtain the estimate for each atom  $\no{e^{ks\sigma}u^\pm_{\zeta,k}}_{L^2}$ of the full norm $\no{\cdot }_{g\T-{\mathcal A}^{\pm,s\sigma}_\zeta}$. Indeed,  	\begin{eqnarray}\begin{aligned}\label{est:t31a}
 					f^2(e^k)\no{e^{ks\sigma}u^+_{\zeta,k}}^2_{L^2}+&k^2f_0^2(e^k)\no{e^{ks\sigma}\di_b\sigma \lrcorner u^+_{\zeta,k}}^2_{L^2}\\
 					&\le c\left(\no{e^{ks\sigma}\dib_bu^+_{\zeta,k}}^2_{L^2}+\no{e^{ks\sigma}\dib_b^*u^+_{\zeta,k}}^2_{L^2}\right) +c_{s,r}e^{-k}\no{\zeta u}^2_{H^{-r}} 				\end{aligned}	\end{eqnarray}      
 				for all $u\in H^{-s_o}_{0,q}(M)$, $s,r\in \R^+$ with $r\ge s_o$ and $q\ge q_0$.
 					\item[Step 3.] A similar estimate for negative microlocalization is  obtained by the aid of Hodge-$*$ on complementary degree. In particular, \eqref{est:t31a} is equivalent to 
 					\begin{eqnarray}\begin{aligned}\label{est:t31b}
 					f^2(e^k)	\no{e^{ks\sigma}u^-_{\zeta,k}}^2_{L^2}+&	k^2f_0^2(e^k)\no{e^{ks\sigma}\dib_b\we  u^-_{\zeta,k}}^2_{L^2}\\
 					\le &c\left(\no{e^{ks\sigma}\dib_bu^-_{\zeta,k}}^2_{L^2}+\no{e^{ks\sigma}\dib_b^*u^-_{\zeta,k}}^2_{L^2}\right) +c_{s,r}e^{-k}\no{\zeta u}^2_{H^{-r}}			\end{aligned}		\end{eqnarray}      
 					for all $u\in H^{-s_o}_{0,q}(M)$, $s,r\in \R^+$ with $r\ge s_o$ and $q\le n-q_0$.
 					\item[Step 4.] 	Multiplying with $g^2(e^k)$ to both sides  of \eqref{est:t31a} and \eqref{est:t31b}, and then taking summation over $k$, we get the desired estimates \eqref{est:t41a} and \eqref{est:t41b}, respectively. 
 				\end{enumerate}

 					The proof of Step 1 immediately follows by following proposition.
 					\begin{proposition}\label{p3.6}
 				If the $\sigma$-superlogarithmic property holds on $(0,q)$-forms then it still holds on $(0,q+1)$-forms with the same rates.
 					\end{proposition}
 				\begin{remark}\label{3r3} Let $\M$ be a nonnegative $(1,1)$-form.  
 				It is obvious that the  $(f\T-\M)$-property on $(0,q)$-forms implies the  $(f\T-\M)$-property on $(0,q+1)$-forms. However, in the definition of the $\sigma$-superlogarithmic property, the Levi form $\L_\sigma$  is not a nonnegative $(1,1)$-form in general. So we have to give a direct proof for Proposition~\ref{p3.6}.
 				\end{remark}
 					\begin{proof} We assume that there exist functions $f,f_0:\R^+\to\R^+$ with $f,f_0\gg 1$ and  a family of smooth real-valued functions $\{\lambda_t\}_{t\ge1}$ defined on $U$ such that 
 						\begin{eqnarray}\label{1p3.6}
 						\begin{cases}
 						\left\la(\L_{\lambda_t}+td\theta)\lrcorner v, v\right\ra\ge f^2_0(t)\left(\ln t|\la\L_\sigma\lrcorner v,v\ra|+\ln^2t|\di_b\sigma \lrcorner v|^2\right)+f^2(t)|v|^2,\\
 						|T^m(\lambda_t)|\le c_mt^{\frac{m}{2}}, \T{for all } m\in \mathbb N,
 						\end{cases}
 						\end{eqnarray}
 						holds at any $x\in U$ for all smooth $(0,q)$-forms $v$. 
 						Let 
 						$u=\underset{L\in \I_{q+1}}{{\sum}}u_L\bar \om_L$ is a $(0,q+1)$-form defined in $U$.
 						We rewrite $u$ as a non-ordered sum
 						$$u=\frac{1}{(q+1)!}\underset{|L|=q+1}{{\sum}}u_L\bar \om_L=\frac{(-1)^q}{q+1}\sum_{l=1}^{n}\Big(\frac{1}{q!}\underset{|J|=q}{{\sum}}u_{lJ}\bom_J \Big)\we\bom_l.$$
 						For $l=1,\dots, n$, we define a set of $(0,q)$-forms $v_l$ by $$v_l:=\frac{1}{q!}\underset{|J|=q}{\sum}u_{lJ}\bar\om_{J}=\underset{J\in \I_{q}}{\sum}u_{lJ}\bar\om_{J}.$$ It is easy to check that $$\begin{cases*}
 						\overset{n}{\underset{l=1}{\sum}} |v_l|^2=(q+1)|u|^2;\\					\overset{n}{\underset{l=1}{\sum}}|\lambda\lrcorner v_l|^2=q|\lambda\lrcorner u|^2 \T{  if $\lambda$ is a $(1,0)$-form};\\
 						 \overset{n}{\underset{l=1}{\sum}}\la \lambda\lrcorner v_l,v_l\ra=q \la \lambda\lrcorner u,u\ra\T{   if $\lambda$ is a $(1,1)$-form}.
 						\end{cases*}$$  Thus, 
 						\begin{eqnarray}\label{p33b}\begin{aligned}
 						&\sum_{l=1}^n\Big(f^2_0(t)\left(\ln t|\la\L_\sigma\lrcorner v_l,v_l\ra|+\ln^2t|\di_b\sigma \lrcorner v_l|^2\right)+f^2(t)|v_l|^2\Big )\\
 						&= qf^2_0(t)\left(\ln t|\la\L_\sigma\lrcorner u,u\ra|+\ln^2t|\di_b\sigma \lrcorner u|^2\right)+(q+1)f^2(t)|u|^2.
 						\end{aligned}
 						\end{eqnarray}
 						Substituting $v=v_l$ into \eqref{1p3.6}  and taking summation over $l$, to get
 						\begin{eqnarray}\label{p33a}
 						\begin{aligned}
 						&	\sum_{l=1}^n\Big(f^2_0(t)\left(\ln t|\la\L_\sigma\lrcorner v_l,v_l\ra|+\ln^2t|\di_b\sigma \lrcorner v_l|^2\right)+f^2(t)|v_l|^2\Big )\\
 						\le& \sum_{l=1}^n\Big(	\left\la(\L_{\lambda_t}+td\theta)\lrcorner v_l, v_l\right\ra\Big)					=q\left\la(\L_{\lambda_t}+td\theta)\lrcorner u, u\right\ra.
 						\end{aligned}
 						\end{eqnarray}
 					From \eqref{p33b} and \eqref{p33a}, the $\sigma$-superlogarithmic property on $(0,q+1)$-forms  holds. 
 						
 					\end{proof}
 				The second step in the proof of Theorem~\ref{t41} is contained in the following proposition.
 				\begin{proposition}\label{p31} Suppose that the $\sigma$-superlogarithmic property with the pair of rates $(f,f_0)$  holds on $(0,q)$-forms.  Then
 					\begin{eqnarray}\label{est:p31}\begin{aligned}
 				 f^2(e^k)\no{e^{ks\sigma}u^+_{\zeta,k}}_{L^2}^2+&k^2f_0^2(e^k)\no{\di_b\sigma\lrcorner e^{ks\sigma}u^+_{\zeta,k}}_{L^2}^2\\
 				 	\le &c\left(\no{e^{ks\sigma}\dib_bu^+_{\zeta,k}}_{L^2}^2+\no{e^{ks\sigma}\dib_b^*u^+_{\zeta,k}}_{L^2}^2\right)+c_{s,r}e^{-k}\no{\zeta u}^2_{H^{-r}},
 					\end{aligned}				\end{eqnarray}
 					 for all $u\in H^{-s_o}_{0,q}(M)$,  $k=1,2\dots$, and $s,r\in \R$ with $r\ge s_o$, where $c$ and $c_{s,r}$ are independent of $k$.
 				\end{proposition}
 				\begin{proof}Assume that for any $t\ge1$ there exists a real-valued  function  $\lambda_t$ smooth in a neighborhood of $\supp(\sigma)$ such that 
 						\begin{eqnarray}\label{d3}
 					\begin{cases}
 					\left\la(\L_{\lambda_t}+td\theta)\lrcorner u, u\right\ra\ge f^2_0(t)\left(\ln t|\la\L_\sigma\lrcorner u,u\ra|+\ln^2t|\di_b\sigma \lrcorner u|^2\right)+f^2(t)|u|^2,\\
 					|T^m(\lambda_t)|\le c_mt^{\frac{m}{2}} \T{~~for all } m\in \mathbb N,
 					\end{cases}
 					\end{eqnarray}
 					for all $u\in C^\infty_{0,q}(M)$. We are going to apply Proposition~\ref{kmh} and ~\ref{Local3.2} for 
 					 $$\lambda:=\lambda_{k,s}:=\chi(\lambda_{e^k})-2ks\sigma$$
 					 where $\chi:\R^+\to\R^+$ is a smooth function and will be chosen later. 
 					By the definition of $\L_\lambda$, we have  
 					$$\L_\lambda=\dot\chi\L_{\lambda_{e^k}}+\ddot\chi \di_b(\lambda_{e^k})\we \dib_b(\lambda_{e^k})-2sk\L_\sigma, $$
 					 and hence 
 				\begin{equation}
 				\label{est1}\la\L_{\lambda}\lrcorner u,u\ra=\dot\chi \la\L_{\lambda_{e^k}}\lrcorner u,u\ra+\ddot{\chi}|\di_b \lambda_{e^k}\lrcorner u|^2-2sk \la\L_\sigma\lrcorner u,u\ra.
 				\end{equation}
 					On other hand, it follows from \eqref{dbar*}, 
 				\begin{equation}
 				\label{est2}\dib^*_{b,\lambda}u=\dib^*_bu-\di_b\lambda\lrcorner u=\dib^*_bu-\dot\chi\di_b\lambda_{e^k}\lrcorner u+2sk\di_b\sigma\lrcorner u.	\end{equation}
 	Thus we get from Proposition~\ref{kmh} applied to $u^+_{\zeta,k}\in C^\infty_{0,q}(M)$ supported in $U$ (by Remark \ref{1r1}), under the choice of the weight $\lambda :=\lambda_{k,s}$ and
taking into account \eqref{est1} and \eqref{est2},			\begin{eqnarray}
 					\begin{aligned}
 					\no{e^{-\frac{\lambda}{2}}\dib_b u^+_{\zeta,k}}^2_{L^2}+&3\no{e^{-\frac{\lambda}{2}}\dib_{b}^* u^+_{\zeta,k}}^2_{L^2}+c\no{e^{-\frac{\lambda}{2}} u^+_{\zeta,k}}^2_{L^2}\\
 					\ge&
 				\Re	\int_M \left[\dot\chi \la\L_{\lambda_{e^k}}\lrcorner u^+_{\zeta,k},u^+_{\zeta,k}\ra+\ddot{\chi}|\di_b \lambda_{e^k}\lrcorner u^+_{\zeta,k}|^2+\la d\theta \lrcorner Tu^+_{\zeta,k},u^+_{\zeta,k}\ra\right]e^{-\lambda}dV\\
 					&-\int_M \left[2sk \la \L_\sigma\lrcorner u^+_{\zeta,k},u^+_{\zeta,k}\ra+3\dot\chi^2|\di_b\lambda_{e^k}\lrcorner u^+_{\zeta,k}|^2+12s^2k^2|\di_b\sigma\lrcorner u^+_{\zeta,k}|^2\right]e^{-\lambda}dS.
 					\end{aligned}
 					\end{eqnarray}
 					Since $\lambda_{e^k}$ is uniformly bounded, we can assume that $-1\le \lambda_{e^k}\le 0$. We choose $\chi(a)=\frac{1}{3}e^{a}$. If follows:(i) $\ddot\chi\ge 3\dot{\chi^2}$ so we can remove the term involving  $|\di_b\lambda_{e^k}\lrcorner u^+_{\zeta,k}|^2$; (ii) by the second line of \ref{d3} it follows
 					$$|\lambda|\le c_{0,s}k\quad \T {and }\quad |T^m(\lambda)|\le c_{m,s}e^{\frac{mk}{2}},$$
 					 so we can use Proposition~\ref{Local3.2} for to estimate the Levi form $\la d\theta\lrcorner T\cdot, \cdot\ra$ as
 					 $$\Re \int_M\la d\theta\lrcorner Tu^+_{\zeta,k},u^+_{\zeta,k}\ra e^{-\lambda}dV\ge \frac{1}{2}e^{k}\int_M\la d\theta\lrcorner u^+_{\zeta,k},u^+_{\zeta,k}\ra e^{-\lambda}dV-c\no{e^{\frac{\lambda}{2}}u^+_{\zeta,k}}^2-c_{r,s}e^{-k}\no{\zeta u}^2_{H^{-r}};$$ (iii) $e^{-\chi}\approx 1$ and $\dot\chi\gtrsim 1$, so we remove the factor $e^{-\chi}$ in $e^{-\lambda}=e^{-\chi+2ks\sigma}$. We can conclude that   
 					\begin{eqnarray}
 					\begin{aligned}
 					\no{e^{sk\sigma}\dib_b u^+_{\zeta,k}}^2_{L^2}+&\no{e^{sk\sigma}\dib^*_{b,} u^+_{\zeta,k}}^2_{L^2}+c_{s,r}e^{-k}\no{\zeta u}_{H^{-r}}^2\\
 					\ge&  c\int_M\left[\la\L_{\lambda_{e^k}}\lrcorner u^+_{\zeta,k},u^+_{\zeta,k}\ra +\frac{1}{2}e^k\la d\theta\lrcorner u^+_{\zeta,k},u^+_{\zeta,k}\ra\right]e^{2sk\sigma}dV\\
 					&-C\int_M\left[|u^+_{\zeta,k}|^2+12s^2k^2|\di_b\sigma\lrcorner u^+_{\zeta,k}|^2+ sk \la \L_\sigma\lrcorner u^+_{\zeta,k},u^+_{\zeta,k}\ra\right]e^{2sk\sigma}dV\\
 					\ge & c'\left(f^2(e^k)\no{e^{ks\sigma}u^+_{\zeta,k}}_{L^2}^2+k^2f^2_0(e^k)\no{\di_b\sigma\lrcorner e^{ks\sigma}u^+_{\zeta,k}}_{L^2}^2+kf_0^2(e^k)\int_M|\la\L_\sigma\lrcorner u^+_{\zeta,k},u^+_{\zeta,k}\ra| e^{2ks\sigma}dV\right)\\
 						&-C\int_M\left[|u^+_{\zeta,k}|^2+12s^2k^2|\di_b\sigma\lrcorner u^+_{\zeta,k}|^2+ sk \la \L_\sigma\lrcorner u^+_{\zeta,k},u^+_{\zeta,k}\ra\right]e^{2sk\sigma}dV.
 					\end{aligned}
 					\end{eqnarray}
 				where the last inequality follows the first line of \eqref{d3} for $t=e^k$.
 					We obtain the desired inequality since the last line is absorbed by the line before for large  $k$ since both $f, f_0\gg 1$, otherwise it is absorbed by $c_{s,r}e^{-k}\no{\zeta u}^2_{H^{-r}}$.	\end{proof}

 			The estimates for positive and negative microlocalizations are related, in complementary degree, by the aid of the Hodge-$*$ theory.
 				\bp
 				\label{p3.7} Fix $1\le q\le n$ and $r\ge s_o$. The estimates
 				\begin{eqnarray}\label{est:p38}\begin{aligned}
 			c_{s,r}e^{-k}\no{\zeta u}^2_{H^{-r}}	+&c\left(\no{e^{ks\sigma}\dib_bu^+_{\zeta,k}}^2_{L^2}+\no{e^{ks\sigma}\dib_b^*u^+_{\zeta,k}}^2_{L^2}\right)\\
 			&	\ge f^2(e^k)\no{e^{ks\sigma}u^+_{\zeta,k}}^2_{L^2}+f_0^2(e^k)k^2\no{\di_b\sigma\lrcorner e^{ks\sigma}u^+_{\zeta,k}}^2_{L^2}
 			\end{aligned}	\end{eqnarray}
 				for all  $u\in H^{-s_o}_{0,q}(M)$  and 
 				\begin{eqnarray}\label{est:p39}\begin{aligned}
 			c_{s,r}e^{-k}\no{\zeta u}^2_{H^{-r}}+&	c\left(\no{e^{ks\sigma}\dib_bu^-_{\zeta,k}}^2_{L^2}+\no{e^{ks\sigma}\dib_b^*u^-_{\zeta,k}}^2_{L^2}\right)\\
 	&	\ge f^2(e^k)\no{e^{ks\sigma}u^-_{\zeta,k}}^2_{L^2}+f_0^2(e^k)k^2\no{\dib_b\sigma\we e^{ks\sigma}u^-_{\zeta,k}}^2_{L^2}
 				\end{aligned}		\end{eqnarray}
 				for all  $u\in  H^{-s_o}_{0,n-q}(M)$ are equivalent.
 				\ep
 				\bpf
 				We only need to prove $\eqref{est:p38}\Longrightarrow\eqref{est:p39}$ since the reversed direction follows analogously. 
 				We define the local conjugate-linear duality map $F^{n-q} : C_{0,n-q}
 				^\infty(M)\to C_{0,q}^\infty(M)$ as follows. If $u=\underset{|J|=n-q}{\sum'} u_J\bom_J$ then 
 				$$F^{n-q}u=\sum \epsilon^{\{J ,J'\}}_{\{1,..., n\}} \bar u_J\bom_{J'} , $$
 				where $J'$ denotes the strictly increasing $q$-tuple consisting of all integers in $[1,n]$ which do not belong to $J$ and $\epsilon^{J,J'}_{\{1,n\}}$ is the sign of the permutation $\{J, J'\}\sim \{1,\dots, n\}$. \\
 				
 				By this definition, $F^{q}F^{n-q}u=u$, $\no{e^{ks\sigma}F^{n-q}u}=\no{e^{ks\sigma}u}$ and furthermore 
 				$$\begin{cases} 
 				\dib_b F^{n-q} u=F^{n-q-1}\left(\dib^*_b u +\cdots\right),\\\ 
 				\dib^*_b F^{n-q} u=F^{n-q+1}\left(\dib_b u +\cdots\right),\\
 				\di_b\sigma \lrcorner F^{n-q}u=F^{n-q+1}(\dib_b\sigma\we u),
 				\end{cases}$$
 				for any $(0,n-q)$-form $u$ supported in $U$, where dots refers the term in which $u$ is not differentiated. It follows
 				\begin{eqnarray}\label{a1} 
 				\begin{cases}\no{e^{ks\sigma}\dib_b F^{n-q}u}^2+\no{e^{ks\sigma}\dib_b^* F^{n-q} u}^2\le 2(\no{e^{ks\sigma}\dib_b  u}^2+\no{e^{ks\sigma}\dib_b^*  u}^2)+C\no{e^{ks\sigma}u}^2\\ 
 				\no{e^{ks\sigma}F^{n-q}u}+k^2\no{e^{ks\sigma}\di_b\sigma \lrcorner F^{n-q}u}=\no{e^{ks\sigma}u}+\no{e^{ks\sigma}\dib_b\sigma\we u}^2.
 				
 				\end{cases}
 				\end{eqnarray}
 				On the other hand for each coefficient $u^-_{k,J}$ of $u^-_{\zeta,k}$. we have
 				\begin{eqnarray}
 				\begin{aligned}
 				\overline{u_{k,J}^-(x)}=&\overline{(\zeta\Gamma_k^-\Psi_k^-\zeta u_J)(x)}\\
 				=&(2\pi)^{2n+1}\zeta(x)\overline{\int_{\R^{2n+1}_\xi}e^{ix\xi}\gamma_k(\xi) \psi^-(\xi) \int_{\R^{2n+1}_y}e^{-iy\xi}(\zeta u_J)(y)dyd\xi}\\
 				=&(2\pi)^{2n+1}\zeta(x)\int_{\R^{2n+1}_\xi}e^{-ix\xi}\gamma(-e^{k}\xi_{2n+1}) \psi(-\xi) \int_{\R^{2n+1}_y}e^{iy\xi}\overline{(\zeta u_J)(y)}dyd\xi\\
 				\overset{\xi:=-\xi}=&-(-2\pi)^{2n+1}\zeta(x)\int_{\R^{2n+1}_\xi}e^{ix\xi} \gamma(e^k\xi_{2n+1})\psi(\xi) \int_{\R^{2n+1}_y}e^{-iy\xi}\zeta(y)\overline{u_J(y)}dyd\xi\\
 				=&-(-2\pi)^{2n+1}\zeta(x)\int_{\R^{2n+1}_\xi}e^{ix\xi} \gamma_k(\xi_{2n+1})\psi^+(\xi) \int_{\R^{2n+1}_y}e^{-iy\xi}(\zeta\bar{u}_J)(y)dyd\xi\\
 				=&-(\zeta\Gamma_k^+\Psi^+\zeta \bar v)(x)
 				=-\bar {u}^+_{k,J}(x),
 				\end{aligned}
 				\end{eqnarray}
 				and hence $F^{n-q}(u^-_{\zeta,k})=-(\overline{F^{n-q}u})^+_k$ is a positive microlocalized $(0,q)$-form supported on $U$. So we apply \eqref{est:p38} for $F^{n-q}(u^-_{\zeta,k})$ together with \eqref{a1} for $u$ replaced by $u^-_{\zeta,k}$, we get \eqref{est:p39}. The proof of Proposition \ref{p3.7} is complete.
 		
 				\epf
 				
 				Now we are ready to get the complete the proof of Theorem~\ref{t41}.
 				\begin{proof}[End proof of Theorem~\ref{t41}]
 					Using the estimate \eqref{est:t31a} for $k=1,2,\dots$ with the choice of $r$ such that $t^{-r}\le g(t)$ for all $t$, we multiply by $g^2(e^k)$ and sum over $k$ to establish
 					
 					\begin{eqnarray}\label{est:dib+dib*}
 					\begin{aligned}
 					&	\no{u}^2_{gf\T-{\mathcal A}^{+,s\sigma}_\zeta}+\no{u}^2_{gf_0\ln\T-\di_b\sigma\lrcorner {\mathcal A}^{+,s\sigma}_\zeta}\\
 					=&\sum_{k=1}^ng^2(e^k)\Big(	f^2(e^k)\no{e^{ks\sigma}u^+_{\zeta,k}}^2_{L^2}+f_0^2(e^k)k^2\no{\di\sigma\lrcorner e^{ks\sigma}u^+_{\zeta,k}}^2_{L^2}\Big)	\\
 					\le &c\sum_{k=1}^ng^2(e^k)\Big(	\no{e^{ks\sigma}  \dib_bu^+_{\zeta,k}}^2_{L^2}+\no{e^{ks\sigma}\dib_b^* u^+_{\zeta,k}}^2_{L^2}+c_{s,g}e^{-k}\no{g(\La)\zeta u}^2_{L^2}\Big)\\
 					\le &c\sum_{k=1}^ng^2(e^k)\Big(	\no{e^{ks\sigma}( \dib_b  u)^+_{\zeta,k}}^2_{L^2}+\no{e^{ks\sigma} ( \dib_b^*  u)^+_{\zeta,k}}^2_{L^2}\\
 					&	+\no{e^{ks\sigma}[\zeta\Gamma_k\Psi^+\zeta,\dib_b] u}_{L^2}^2+\no{e^{ks\sigma}[\zeta\Gamma_k\Psi^+\zeta,\dib_b^*] u}_{L^2}^2
 					+c_{s,g}e^{-k}\no{g(\La)\zeta u}^2_{L^2}\Big)\\
 					\le& c\left( \no{\dib_bu}^2_{g\T-{\mathcal A}^{+,s\sigma}_\zeta}+\no{\dib_b^*u}^2_{g\T-{\mathcal A}^{+,s\sigma}_\zeta}\right)+c_{s,g}\left(\no{u}^2_{g\T-{\mathcal A}^{+,s\sigma}_\zeta}+\no{u}^2_{g\T-{\mathcal E}_{\zeta,\tilde\zeta}^{s,0}}\right),
 					\end{aligned}
 					\end{eqnarray}
 					where the last inequality follows by Lemma~\ref{com1}.	
 					Since $f\gg 1$, the term $c_{s,g}\no{u}^2_{g\T-{\mathcal A}^{+,s\sigma}_\zeta}$ can be absorbed by $\no{u}^2_{gf\T-{\mathcal A}^{+,s\sigma}_\zeta}$ in the first line of \eqref{est:dib+dib*}. Then the desired estimate \eqref{est:t31a} follows. The negative microlocal estimate \eqref{est:t31b} is proved analogously.
 				\end{proof}	 	
 			
 			We also have a proof of Theorem~\ref{t0} as a consequence of Remark~\ref{3r3} and Propositions~\ref{p31} and \ref{p3.7} as follows.
 			\begin{proof}[Proof of Theorem~\ref{t0}]By Remark~\ref{3r3} the $(f\T-\M)$ property holds on $(0,q)$-forms with $q\ge q_0$. Similar to the argument in Proposition~\ref{p31} but with simpler calculation by the choice of weight 
			$\lambda=\chi(\lambda_{e^k})$, we get the $(f\T-\M)$-estimate for each $u_{\zeta,k}^+$ as 
 				\begin{equation}
 				\label{u+zeta}
 				f(e^k)^2(\M\lrcorner u^+_{\zeta,k},u^+_{\zeta,k})_{L^2}\lesssim \no{\dib_b u^+_{\zeta,k}}^2_{L^2}+\no{\dib_b^*u^+_{\zeta,k}}^2_{L^2}+\no{u^+_{\zeta,k}}^2_{L^2}+e^{-k}\no{\zeta u}^2_{H^{-1}}.
 				\end{equation}
 				for $u\in C^\infty_{0,q}(M)$.
 				In order to get the estimate for full $u^+_\zeta:=\zeta\Psi^+\zeta u$, we first decompose $u^+_{\zeta}=\sum_{k=0}^\infty u_{\zeta,k}^+$ and hence 
 			\begin{equation} 		\begin{aligned}
 			(\M\lrcorner f(\La)u^+_{\zeta},f(\La)u^+_{\zeta})_{L^2}=	\sum_{k,\hat k}(\M\lrcorner f(\La)u^+_{\zeta,k},f(\La)u^+_{\zeta,\hat k})_{L^2}.
 			\end{aligned}\end{equation} We first consider the case $|k-\hat k|\ge3$. In this case we have $\supp(\gamma_k)\cap \supp(\gamma_{\hat k})=\emptyset$  and hence, by the argument in Lemma~\ref{l1} below we can get 
 				$$\left|(\M\lrcorner f(\La)u^+_{\zeta,k},f(\La)u^+_{\zeta,\hat k})_{L^2}\right|\lesssim e^{-(k+\hat k)}\no{\zeta u}^2_{H^{-1}}\quad\T{for all }|k-\hat k|\ge 3.$$
 				Otherwise for $|k-\hat k|\le 3$, we use the Cauchy-Schwarz inequality since the Hermitian form $(\M\lrcorner \cdot,\cdot)$ is nonnegative. Therefore we obtain 
 					\begin{equation} 		\label{fM1}		\begin{aligned}
 					(\M\lrcorner f(\La)u^+_{\zeta},f(\La)u^+_{\zeta})_{L^2}\lesssim& 	\sum_{k=1}^\infty\Big((\M\lrcorner f(\La)u^+_{\zeta,k},f(\La)u^+_{\zeta,k})_{L^2}+e^{-k}\no{\zeta u}^2_{H^{-1}}\Big)\\
 				\underset{\T{G{\aa}rding's inequality}}{	\lesssim}& \sum_{k=1}^\infty\Big( f^2(e^k)(\M\lrcorner u^+_{\zeta,k},u^+_{\zeta,k})_{L^2}+\no{u^+_{\zeta,k}}^2_{L^2}+e^{-k}\no{\zeta u}^2_{H^{-1}}\Big)\\
 					\underset{\T{by \eqref{u+zeta}}}{\lesssim}& \sum_{k=1}^\infty\Big(\no{\dib_b u^+_{\zeta,k}}^2_{L^2}+\no{\dib_b^*u^+_{\zeta,k}}^2_{L^2}+\no{u^+_{\zeta,k}}^2_{L^2}+e^{-k}\no{\zeta u}^2_{-1}\Big)\\
 						\underset{\T{by Lemma~\ref{com1}}}{\lesssim}& \no{\dib_b u^+_{\zeta}}^2_{L^2}+\no{\dib_b^*u^+_{\zeta}}^2_{L^2}+\no{\tilde \zeta u}^2_{L^2}
 					\end{aligned}\end{equation}			
 				for all $u\in C^\infty_{0,q}(M)$, where $\tilde \zeta\succ\zeta$. By \cite[Theorem 5]{Kha16b} (or analogous proof of Proposition~\ref{p3.7}), \eqref{fM1} is equivalent to 
 				\begin{equation} 			\begin{aligned}
 			\no{\sqrt{\Tr(\M)}\times f(\La)u^-_{\zeta}}^2_{L^2}-(\M\lrcorner f(\La)u^-_{\zeta},f(\La)u^-_{\zeta})_{L^2}\lesssim&  \no{\dib_b u^-_{\zeta}}^2_{L^2}+\no{\dib_b^*u^-_{\zeta}}^2_{L^2}+\no{\tilde \zeta u}^2_{L^2}.
 				\end{aligned}\end{equation}	
 				for all $u\in C^{\infty}_{0,q}(M)$ with $q\le n-q_0$.			
 				This is complete the proof of Theorem~\ref{t0}.
 			\end{proof}	
 				\subsection{A-priori estimates for $\Box_b$}\label{3.2}

 			We need  technique lemmas before going estimates for $\Box_b$ in our norms.
 				\begin{lemma}\label{l41}
 					Let $a\ge 0$ and $\eps>0$.  Then, there exists $c_\eps$ such that the following holds.  	
 					\begin{enumerate}
 						\item[(i)] 	for $v\in \left(g\T-{\mathcal A}_{\zeta;0,q-1}^{+,s\sigma+a}\right)\cap\left(g\T-{\mathcal E}_{\zeta,\tilde\zeta;0,q-1}^{s,a}\right) $ and $w\in \left(g\T-{\mathcal A}_{\zeta;0,q}^{+,s\sigma+1-a}\right)\cap\left( g\T-{\mathcal E}_{\zeta,\tilde\zeta;0,q}^{s,-a}\right)$, we have 
 							\begin{eqnarray}\label{est:l42a}
 					\begin{aligned}
 					\left|\Re\left\{(v,\dib^*_bw)_{g\T-{\mathcal A}_\zeta^{+,s\sigma}}-(\dib_bv,w)_{g\T-{\mathcal A}_\zeta^{+,s\sigma}}\right\}\right|\le
 					\eps\left(\no{v}^2_{g\T-{\mathcal A}_\zeta^{+,s\sigma+a}}+\no{v}^2_{g\T-{\mathcal E}^{s,a}_{\zeta,\tilde{\zeta}}}\right)\\
 				+	c_\eps\left(\no{w}^2_{g\T-{\mathcal A}_\zeta^{+,s\sigma-a}}+\no{w}^2_{g\ln\T-\di_b\sigma\lrcorner {\mathcal A}_\zeta^{+,s\sigma-a}}+\no{w}^2_{g\T-{\mathcal E}^{s,-a}_{\zeta,\tilde \zeta}}\right)
 					\end{aligned}
 					\end{eqnarray}
 				\item[(ii)] For  $v\in \left(g\T-{\mathcal A}_{\zeta;0,q-1}^{-,s\sigma+a}\right)\cap\left(g\T-{\mathcal E}_{\zeta,\tilde\zeta;0,q-1}^{s,a}\right) $ and $w\in \left(g\T-{\mathcal A}_{\zeta;0,q}^{-,s\sigma+1-a}\right)\cap\left( g\T-{\mathcal E}_{\zeta,\tilde\zeta;0,q}^{s,-a}\right)$, we have
\begin{eqnarray}\label{est:l42b}
 					\begin{aligned}
 					\left|\Re\left\{(v,\dib_bw)_{g\T-{\mathcal A}_\zeta^{-,s\sigma}}-(\dib_b^*v,w)_{g\T-{\mathcal A}_\zeta^{-,s\sigma}}\right\}\right|\le 	\eps\left(\no{v}^2_{g\T-{\mathcal A}_\zeta^{-,s\sigma+a}}+\no{v}^2_{g\T-{\mathcal E}^{s,a}_{\zeta,\tilde{\zeta}}}\right)\\
 					+c_\eps\left(\no{w}^2_{g\T-{\mathcal A}_\zeta^{-,s\sigma-a}}+\no{w}^2_{g\ln\T-\dib_b\sigma\we {\mathcal A}_\zeta^{-,s\sigma-a}}+\no{w}^2_{g\T-{\mathcal E}^{s,-a}_{\zeta,\tilde{\zeta}}}\right).
 				\end{aligned}
 				\end{eqnarray}
 			
 					\end{enumerate}	
 				
 				\end{lemma}	
 				\begin{proof}		
 					
 					We first work on each atom of the inner product $(\cdot,\cdot)_{g\T-{\mathcal A}^{+,s\sigma}_\zeta}$,
 					\begin{eqnarray}
 					\label{est:dib*}\begin{aligned}
 				&\left|	(e^{2ks\sigma}v^+_{\zeta,k},(\dib_b^*w)^+_{\zeta,+})_{L^2}
 					-(e^{2ks\sigma}( \dib_bv)^+_{\zeta,k},w^+_{\zeta,k})_{L^2}\right|\\
 					=&\left|(e^{2ks\sigma}[\dib_b,\zeta\Gamma_k\Psi^+\zeta ]v,w_{\zeta,k}^+)_{L^2}
 					+(v_{\zeta,k}^+,[e^{2ks},\dib_b^*]   w_{\zeta,k}^+)_{L^2}+(e^{2ks\sigma} v_{\zeta,k}^+,[\zeta\Gamma_k\Psi^+\zeta ,\dib_b^*]w)_{L^2}\right|\\
 					\le&\eps\left(\no{e^{k(s\sigma+a)}[\dib_b,\zeta\Gamma_k\Psi^+\zeta ]v}^2_{L^2}+ \no{e^{k(s\sigma+a)} v^+_{\zeta,+}}^2_{L^2}\right)\\
 					&+c_\eps\left(\no{e^{k(s\sigma-a)} w^+_{\zeta,k}}^2_{L^2}+k^2\no{e^{k(s\sigma-a}\di_b\sigma\lrcorner w_{\zeta,k}^+}^2_{L^2}+\no{e^{k(s\sigma-a)}[\zeta\Gamma_k\Psi^+\zeta ,\dib_b^*]w}^2_{L^2}\right).
 					\end{aligned}\end{eqnarray}
 					Now multiply  both sides  \eqref{est:dib*} by $g^2(e^k)$ and sum over $k$. The desired estimate \eqref{est:l42a} follows by the estimates of commutator in Lemma~\ref{com1}. The desired estimate \eqref{est:l42b} is proved analogously.
 				\end{proof}		
 		
 		\begin{lemma}\label{l42}
 		 Suppose that  the $\sigma$-superlogarithmic property with the pair of rates $(f,f_0)$ holds on $(0,q)$-forms. Then the following holds:
 		 \begin{enumerate}
 		 	\item[(i)] For $q\ge q_0$, we have
 		 \[\begin{aligned}
 		 \no{u}^2_{gf\T-{\mathcal A}^{+,s\sigma}_\zeta}+&\no{u}^2_{gf_0\ln\T-\di_b\sigma\lrcorner {\mathcal A}^{+,s\sigma}_\zeta}+\no{\dib_b^*u}^2_{g\T-{\mathcal A}^{+,s\sigma}_\zeta}\\
 		 \lesssim&\Re\, (\dib_b\dib_b^* u,u)_{g\T-{\mathcal A}^{+,s\sigma}_\zeta}+\no{\dib_bu}^2_{g\T-{\mathcal A}^{+,s\sigma}_\zeta}+\no{u}^2_{g\T-{\mathcal E}^{s,1}_{\zeta,\tilde{\zeta}}},
 		 \end{aligned}\]
 		 for all $u\in \left(g\T-{\mathcal A}_{\zeta;0,q}^{+,s\sigma+1}\right)\cap\left(g\T-{\mathcal E}_{\zeta,\tilde\zeta;0,q}^{s,1}\right)$.
 		 
\item[(ii)] For $q\le n-q_0$, we have
 		 \[\begin{aligned}
 		 \no{u}^2_{gf\T-{\mathcal A}^{-,s\sigma}_\zeta}+&\no{u}^2_{gf_0\ln\T-\dib_b\sigma\we {\mathcal A}^{,s\sigma}_\zeta}+\no{\dib_bu}^2_{g\T-{\mathcal A}^{-,s\sigma}_\zeta}\\
 		 \lesssim&\Re\, (\dib_b^*\dib_b u,u)_{g\T-{\mathcal A}^{-,s\sigma}_\zeta}+\no{\dib_bu^*}^2_{g\T-{\mathcal A}^{-,s\sigma}_\zeta}+\no{u}^2_{g\T-{\mathcal E}_{\zeta,\tilde{\zeta}}^{s,1}},
 		 \end{aligned}\]
 		 for all $u\in\left(g\T-{\mathcal A}_{\zeta;0,q}^{-,s\sigma+1}\right)\cap\left(g\T-{\mathcal E}_{\zeta,\tilde\zeta;0,q}^{s,1}\right)$ with $q\le n-q_0$.
 		  		 \end{enumerate}
 		\end{lemma}		
 				\begin{proof}
 					 Using Lemma~\ref{l41} for $v=\dib^*u$ and $w=u$ with the choice $a=0$, we get	 
 					 \begin{eqnarray}\label{new41a}
 					 \begin{aligned}
 					 \no{\dib_b^*u}^2_{g\T-{\mathcal A}^{+,s\sigma}_\zeta}=&\Re\left(\dib_b^*u,\dib_b^*u\right)_{g\T-{\mathcal A}^{+,s\sigma}_\zeta}\\
 					 \le&\Re\left(\dib_b\dib_b^*u,u\right)_{g\T-{\mathcal A}^{+,s\sigma}_\zeta}+\left|\Re\left(\dib_b^*u,\dib_b^*u\right)_{g\T-{\mathcal A}^{+,s\sigma}_\zeta}-\Re\left(\dib_b\dib_b^*u,u\right)_{g\T-{\mathcal A}^{+,s\sigma}_\zeta}\right|\\
 					 \le&\Re\left(\dib_b\dib_b^*u,u\right)_{g\T-{\mathcal A}^{+,s\sigma}_\zeta}+\eps\left(\no{\dib^*_bu}^2_{g\T-{\mathcal A}^{+,s\sigma}_\zeta}+\no{\dib^*_bu}^2_{g\T-{\mathcal E}^{s,0}_{\zeta,\tilde\zeta}}\right)\\
 					 &+c_\eps\left(\no{u}^2_{g\T-{\mathcal A}^{+,s\sigma}_\zeta}+\no{u}^2_{g\ln\T-\di_b\sigma\lrcorner {\mathcal A}^{+,s\sigma}_\zeta}+\no{u}^2_{g\T-{\mathcal E}^{s,0}_{\zeta,\tilde\zeta}}
 					 \right).
 					 \end{aligned}
 					 \end{eqnarray}
 					First, we see that for small $\eps>0$, the term $\no{\dib^*_bu}^2_{g\T-{\mathcal A}^{+,s\sigma}_\zeta}$ in the third line of \eqref{new41a} is absorbed. Then, adding $\no{\dib_bu}^2_{g\T-{\mathcal A}^{+,s\sigma}_\zeta}$ to both sides 
					of \eqref{new41a} and applying Theorem~\ref{t41}, we absorb $c_\eps\left(\no{u}^2_{g\T-{\mathcal A}^{+,s\sigma}_\zeta}+\no{u}^2_{g\ln\T-\di_b\sigma\lrcorner {\mathcal A}^{+,s\sigma}_\zeta}\right)$ by  $\no{u}^2_{gf\T-{\mathcal A}^{+,s\sigma}_\zeta}+\no{u}^2_{gf_0\ln\T-\di_b\sigma\lrcorner {\mathcal A}^{+,s\sigma}_\zeta}$ since $f, f_0\gg1$. Therefore, the remainder terms,  after absorbing,  are $\Re\, (\dib_b\dib_b^* u,u)_{g\T-{\mathcal A}^{+,s\sigma}_\zeta}+\no{\dib_bu}^2_{g\T-{\mathcal A}^{+,s\sigma}_\zeta}$ and terms in the error norm $\no{\cdot}^2_{g\T-{\mathcal E}^{s,0}_{\zeta,\tilde\zeta}}$. Finally, we notice that all terms in the error norm are bounded by $\no{u}^2_{g\T-{\mathcal E}^{s,1}_{\zeta,\tilde{\zeta}}}$. This proves the first estimate in Lemma~\ref{l41}. The proof of the second follows analogously.
 				\end{proof}
 			The main theorem in this subsection is the {\it a-priori} estimates for $\Box_b$ on $\pm$-microlocalizations of $(0,q)$-forms in the norm $\no{\cdot}_{g\T-{\mathcal A}^{\pm,s\sigma}_{\zeta}}$.
 				\begin{theorem}\label{t42} Suppose that  the $\sigma$-superlogarithmic property on $(0,q_o)$-forms with the pair of rates $(f,f_0)$ holds. Then, for $s,r\in \R^+$ with $r\ge s_o+2$, the following holds.
 				\begin{enumerate}
 					\item[(i)] 
 					For  $u\in {\mathcal A}^{+,s\sigma+2-r}_{\zeta;0,q}\cap {\mathcal E}^{s,2-r}_{\zeta,\tilde\zeta;0,q}$ with $q\ge q_0$, we have
 					\[\begin{aligned}
 					\no{u}^2_{f^2\T-{\mathcal A}^{+,s\sigma-r}_\zeta}+\no{\dib_bu}^2_{f\T-{\mathcal A}^{+,s\sigma-r}_\zeta}+\no{\dib_b^*u}^2_{f\T-{\mathcal A}^{+,s\sigma-r}_\zeta}
 					+&\no{\dib_b^*\dib_bu}^2_{{\mathcal A}^{+,s\sigma-r}_\zeta}+\no{\dib_b\dib_b^*u}^2_{{\mathcal A}^{+,s\sigma-r}_\zeta}\\
 				\no{u}^2_{ff_0\ln\T-\di_b\sigma\lrcorner {\mathcal A}^{+,s\sigma-r}_\zeta}+\no{\dib_bu}^2_{f_0\ln\T-\di_b\sigma\lrcorner {\mathcal A}^{+,s\sigma-r}_\zeta}	\lesssim&\no{\Box_b u}^2_{{\mathcal A}^{+,s\sigma-r}_\zeta}+\no{u}^2_{{\mathcal E}^{s,2-r}_{\zeta,\tilde{\zeta}}}.
 					\end{aligned}\]
 					\item[(ii)] For $u\in {\mathcal A}^{-,s\sigma+2-r}_{\zeta;0,q}\cap {\mathcal E}^{s,2-r}_{\zeta,\tilde\zeta;0,q}$ with $q\le n-q_0$, we have
 					\[\begin{aligned}
 					\no{u}^2_{f^2\T-{\mathcal A}^{-,s\sigma-r}_\zeta}+\no{\dib_bu}^2_{f\T-{\mathcal A}^{-,s\sigma-r}_\zeta}+\no{\dib_b^*u}^2_{f\T-{\mathcal A}^{-,s\sigma-r}_\zeta}
 					+&\no{\dib_b^*\dib_bu}^2_{{\mathcal A}^{-,s\sigma-r}_\zeta}+\no{\dib_b\dib_b^*u}^2_{{\mathcal A}^{-,s\sigma-r}_\zeta}\\
 					\no{u}^2_{ff_0\ln\T-\dib_b\sigma\we {\mathcal A}^{-,s\sigma-r}_\zeta}+\no{\dib_b^*u}^2_{f_0\ln\T-\dib_b\sigma\we {\mathcal A}^{-,s\sigma-r}_\zeta}	\lesssim&\no{\Box_b u}^2_{{\mathcal A}^{-,s\sigma-r}_\zeta}+\no{u}^2_{{\mathcal E}^{s,2-r}_{\zeta,\tilde\zeta }}.
 					\end{aligned}\]
 			\end{enumerate}	\end{theorem}
 			
 \begin{proof}
 		We only need to prove the positive microlocalization since the negative microlocalization follows analogously.
 	Combining the estimate \eqref{est:t41a} in  Theorem~\ref{t41} with $\dib_b u$ substituted in for ``u" and
	Lemma~\ref{l42}(i) with the choice $g(t)=t^{-r}f(t)$ for some $r\ge s_o+2$, we get 
 		\begin{eqnarray}\label{4a1}\begin{aligned}
\no{u}^2_{f^2\T-{\mathcal A}^{+,s\sigma-r}_\zeta}+&\no{\dib^*_bu}^2_{f\T-{\mathcal A}^{+,s\sigma-r}_\zeta}+ \no{\dib_bu}^2_{f\T-{\mathcal A}^{+,s\sigma-r}_\zeta}+\no{\dib_bu}^2_{f_0\ln\T-\di_b\sigma\lrcorner {\mathcal A}^{+,s\sigma-r}_\zeta}\\+&\no{u}^2_{ff_0\ln\T-\di_b\sigma\lrcorner {\mathcal A}^{+,s\sigma-r}_\zeta}+\no{\dib_bu}^2_{f_0\ln\T-\di_b\sigma\lrcorner {\mathcal A}^{+,s\sigma-r}_\zeta}\\
&\lesssim\Re\, (\dib_b\dib_b^* u,u)_{f\T-{\mathcal A}^{+,s\sigma-r}_\zeta}+\no{\dib^*_b\dib_bu}^2_{{\mathcal A}_\zeta^{+,s\sigma-r}}+\no{u}^2_{{\mathcal E}^{s,2-r}_{\zeta,\tilde\zeta }}.
 		 		\end{aligned}\end{eqnarray}
 		 		Using the Cauchy-Schwarz inequality, 
 		 		$$\left|\Re\, (\dib_b\dib_b^* u,u)_{f\T-{\mathcal A}^{+,s\sigma-r}_\zeta}\right|\le \eps \no{u}^2_{f^2\T-{\mathcal A}^{+,s\sigma-r}_\zeta}+c_\eps \no{\dib_b\dib_b^*u}^2_{{\mathcal A}_\zeta^{+,s\sigma-r}},$$
 		 		 and absorbing $\no{u}^2_{f^2\T-{\mathcal A}^{+,s\sigma-r}_\zeta}$, we only need to bound $\no{\dib^*_b\dib_bu}^2_{{\mathcal A}_\zeta^{+,s\sigma-r}}+\no{\dib_b\dib_b^*u}^2_{{\mathcal A}_\zeta^{+,s\sigma-r}}$.
 		By the definition of $\Box_b$, it is easy to see that
 		\begin{eqnarray}\label{4a2}\begin{aligned}
 		\no{\dib_b^*\dib_b u}^2_{{\mathcal A}_\zeta^{+,s\sigma-r}}+\no{\dib_b\dib_b^*u}^2_{{\mathcal A}_\zeta^{+,s\sigma-r}}
 		=\no{\Box_b u}^2_{{\mathcal A}_\zeta^{+,s\sigma-r}}-2\Re(\dib_b\dib_b^* u,\dib_b^*\dib_b u)_{{\mathcal A}_\zeta^{+,s\sigma-r}}
 		\end{aligned}\end{eqnarray}
 		For the last term, we apply again Lemma~\ref{l42}(i) for $v=\dib_b\dib_b^*u$ and $w=\dib_b u$ with the choice $g(t)=t^{-r}$ and $a=0$ , since $\dib_bv=\dib_b\dib_b\dib_b^*u=0$ it follows
 		\begin{eqnarray}\label{4a3}\begin{aligned}
 		\left|\Re(\dib_b\dib_b^* u,\dib_b^*\dib_b u)_{{\mathcal A}_\zeta^{+,s\sigma-r}}\right|
&\le
\eps\left(\no{\dib_b\dib_b^* u}^2_{{\mathcal A}_\zeta^{+,s\sigma-r}}+\no{\dib_b\dib_b^* u}^2_{{\mathcal E}^{s,-r}_{\zeta,\tilde\zeta }}\right)\\
&+c_\eps\left(\no{\dib_b u}^2_{{\mathcal A}_\zeta^{+,s\sigma-r}}+\no{\dib_b u}^2_{\ln\T-\di_b\sigma\lrcorner {\mathcal A}^{+,s\sigma-r}_\zeta}+\no{\dib_bu}^2_{{\mathcal E}^{s,-r}_{\zeta,\tilde\zeta }}\right).
\end{aligned}\end{eqnarray}	
 		From \eqref{4a1}, \eqref{4a2}, \eqref{4a3}, and all terms in the error norms are bounded by $\no{u}_{{\mathcal E}^{s,2-r}_{\zeta,\tilde{\zeta}}}$, we obtain the desired inequality. It completes the proof of Theorem~\ref{t42}. 
 		\end{proof}

 \subsection{A-priori estimates for $\Box_b^\delta$}\label{s3.3}
For $\delta>0$, we set $$\Box_b^\delta:=\Box_b+\delta\left(\Box_d+I\right),$$
 where $\Box_d:=d^*d+dd^*$ is the elliptic operator associated to the de Rham exterior derivative $d$. We will give more discussion on $\Box_b^\delta$ in
 the next section. By the ellipticity of $\Box_d$, we have the following equivalences
 \begin{eqnarray}\label{000}
 \begin{aligned}
 \no{u}^2_{{\mathcal A}_\zeta^{+,s\sigma+2-r}}+  \no{u}^2_{{\mathcal E}^{s,2-r}_{\zeta,\tilde\zeta }}\approx & \no{du}^2_{{\mathcal A}_\zeta^{+,s\sigma+1-r}}+\no{d^*u}^2_{{\mathcal A}_\zeta^{+,s\sigma+1-r}}+\no{u}^2_{{\mathcal A}_\zeta^{+,s\sigma+1-r}}+\no{u}^2_{{\mathcal E}^{s,2-r}_{\zeta,\tilde\zeta }}\\
 \approx& ((\Box_d+I)u,u)^2_{{\mathcal A}_\zeta^{+,s\sigma+1-r}}+\no{u}^2_{{\mathcal E}^{s,2-r}_{\zeta,\tilde\zeta }}\\
 \approx& \no{ ((\Box_d+I)u}^2_{{\mathcal A}_\zeta^{+,s\sigma-r}}+\no{u}^2_{{\mathcal E}^{s,2-r}_{\zeta,\tilde\zeta }}
 \end{aligned}
 \end{eqnarray}
 for any $u\in {\mathcal A}^{+,s\sigma+2-r}_{\zeta;0,q}\cap {\mathcal E}^{s,2-r}_{\zeta,\tilde\zeta;0,q}$ and $0\le q\le n$. The boundedness of $\Box_b$ by $\Box_b^\delta$ is given in the following lemma.
 		\begin{lemma}\label{l44} Suppose that the $\sigma$-superlogarithmic property on $(0,q_0)$-forms  holds.   Then the following holds uniformly in $\delta$.
 			 
 			\begin{enumerate}
 				\item[(i)] 
 				For $u\in {\mathcal A}^{+,s\sigma+2-r}_{\zeta;0,q}\cap {\mathcal E}^{s,2-r}_{\zeta,\tilde\zeta;0,q}$ with $q\ge q_0$, we have
 			$$\no{\Box_bu}^2_{{\mathcal A}_\zeta^{+,s\sigma-r}}\lesssim\no{\Box_b^\delta u}^2_{{\mathcal A}_\zeta^{+,s\sigma-r}}+\no{\dib_b u}^2_{ {\mathcal A}_\zeta^{+,s\sigma-r}}+\no{\dib_b u}^2_{\ln\T-\di_b\sigma\lrcorner {\mathcal A}_\zeta^{+,s\sigma-r}}+\no{u}^2_{{\mathcal E}^{s,2-r}_{\zeta,\tilde{\zeta}}}.$$
 				\item[ii.] 
 			For $u\in {\mathcal A}^{-,s\sigma+2-r}_{\zeta;0,q}\cap {\mathcal E}^{s,2-r}_{\zeta,\tilde\zeta;0,q}$ with $q\le n-q_0$, we have
 			$$\no{\Box_bu}^2_{{\mathcal A}_\zeta^{-,s\sigma-r}}\lesssim\no{\Box_b^\delta u}^2_{{\mathcal A}_\zeta^{-,s\sigma-r}}+\no{\dib_b ^*u}^2_{ {\mathcal A}_\zeta^{-,s\sigma-r}}+\no{\dib_b^* u}^2_{\ln\T-\dib_b\sigma\we {\mathcal A}_\zeta^{-,s\sigma-r}}+\no{u}^2_{{\mathcal E}^{s,2-r}_{\zeta,\tilde{\zeta}}}.$$
 		\end{enumerate}
 		\end{lemma}
 		\begin{proof}
 Now, we use Lemma~\ref{l42}.(i) for  the choice  $g(t)=t^{-r+1}$ to get 
 		\begin{eqnarray}\label{est:l41m}\begin{aligned}
 		\no{\dib_b^*u}^2_{{\mathcal A}^{+,s\sigma-r+1}_\zeta}
 		\le\Re\, (\dib_b\dib_b^* u,u)_{{\mathcal A}^{+,s\sigma-r+1}_\zeta}+\no{\dib_bu}^2_{{\mathcal A}^{+,s\sigma-r+1}_\zeta}+\no{u}^2_{{\mathcal E}^{s,2-r}_{\zeta,\tilde{\zeta}}}
 		\end{aligned}\end{eqnarray}
To estimate  $\no{\dib_bu}^2_{{\mathcal A}^{+,s\sigma-r+1}_\zeta}$, we apply Lemma~\ref{l41}.(i) for $v=u$ and $w=\dib_bu$ with the choice $g(t)=t^{-r+1}$ and $a=1$, $\eps:=\eps\delta$ to obtain
 			\begin{eqnarray}\label{est:l42m}
 			\begin{aligned}
\no{\dib_bu}^2_{{\mathcal A}_\zeta^{+,s\sigma-r+1}}=&\Re\,(\dib_bu,\dib_bu)_{{\mathcal A}_\zeta^{+,s\sigma-r+1}}\\
\le& \Re\, (\dib_b^*\dib_bu,u)_{{\mathcal A}^{+,s\sigma-r+1}}+\epsilon\delta\left( \no{u}^2_{{\mathcal A}_\zeta^{+,s\sigma-r+2}}+\no{u}^2_{{\mathcal E}_{\zeta,\tilde{\zeta}}^{s,2-r}}
\right)\\
&+c_\epsilon\delta^{-1}\left(\no{\dib_bu}^2_{ {\mathcal A}_\zeta^{+,s\sigma-r}}+\no{\dib_bu}^2_{\ln\T-\di_b\sigma\lrcorner {\mathcal A}_\zeta^{+,s\sigma-r}}  +\no{\dib_bu}^2_{{\mathcal E}_\zeta^{s,-r}}\right).
 			\end{aligned}\end{eqnarray}
Combining \eqref{est:l41m} and \eqref{est:l42m}, we get 
 			\begin{eqnarray}\label{est:l43m}
 			\begin{aligned}
\no{\dib_b^*u}^2_{{\mathcal A}_\zeta^{+,s\sigma-r+1}}\le& \Re\,(\Box_bu,u)_{{\mathcal A}_\zeta^{+,s\sigma-r+1}}+\epsilon\delta \no{u}^2_{{\mathcal A}_\zeta^{+,s\sigma-r+2}}\\
&c_\epsilon\delta^{-1}\left(\no{\dib_bu}^2_{ {\mathcal A}_\zeta^{+,s\sigma-r}}+\no{\dib_bu}^2_{\ln\T-\di_b\sigma\lrcorner {\mathcal A}_\zeta^{+,s\sigma-r}}  
+\no{u}^2_{{\mathcal E}_{\zeta,\tilde{\zeta}}^{s,2-r}}\right).
 			\end{aligned}\end{eqnarray}
Adding $(\delta(\Box_d+I)u,u)_{{\mathcal A}_\zeta^{+,s\sigma+1-r}}$ into both sides of \eqref{est:l43m}  , after that using the Cauchy-Schwarz inequality for $\Re\,(\Box_b^\delta u,u)_{{\mathcal A}_\zeta^{+,s\sigma-r+1}}$, and then multiplying with $\delta$, we obtain
 			\begin{eqnarray}\label{est:l44m}
 			\begin{aligned}
&\delta^2((\Box_d+I)u,u)_{{\mathcal A}_\zeta^{+,s\sigma+1-r}}+\delta\no{\dib_b^*u}^2_{{\mathcal A}_\zeta^{+,s\sigma-r+1}}
\le 2\epsilon\delta ^2\no{u}^2_{{\mathcal A}_\zeta^{+,s\sigma-r+2}}\\
&+c_\epsilon\left(\no{\Box_b^\delta u}^2_{ {\mathcal A}_\zeta^{+,s\sigma-r}}+\no{\dib_bu}^2_{ {\mathcal A}_\zeta^{+,s\sigma-r}}+\no{\dib_bu}^2_{\ln\T-\di_b\sigma\lrcorner {\mathcal A}_\zeta^{+,s\sigma-r}}  
+\no{u}^2_{{\mathcal E}_{\zeta,\tilde{\zeta}}^{s,2-r}}\right).
 			\end{aligned}\end{eqnarray}
 			 Using \eqref{000} to absorb the first term in the second line of \eqref{est:l44m}, we get the estimate for $\delta^2(\Box_d+I)$ in the following inequality
 			 \begin{eqnarray*}\label{est:l45m}
 			 \begin{aligned}
 			 \delta^2\no{(\Box_d+I)u}^2_{{\mathcal A}_\zeta^{+,s\sigma-r}}
 			 \lesssim\no{\Box_b^\delta u}^2_{ {\mathcal A}_\zeta^{+,s\sigma-r}}+\no{\dib_bu}^2_{ {\mathcal A}_\zeta^{+,s\sigma-r}}+\no{\dib_bu}^2_{\ln\T-\di_b\sigma\lrcorner {\mathcal A}_\zeta^{+,s\sigma-r}}  
 			 +\no{u}^2_{{\mathcal E}_{\zeta,\tilde{\zeta}}^{s,2-r}}.
 			 \end{aligned}\end{eqnarray*}
 			The proof of the first part of this lemma is complete by the simple inequality 
 			$$\no{\Box_bu}^2_{{\mathcal A}_\zeta^{+,s\sigma-r}}\le 2 \delta^2\no{(\Box_d+I)u}^2_{{\mathcal A}_\zeta^{+,s\sigma-r}}+\no{\Box_b^\delta u}^2_{ {\mathcal A}_\zeta^{+,s\sigma-r}}.$$	The proof of the second part goes in analogously way.
	\end{proof}
	The estimates for elliptic regularization method follows immediately by Theorem~\ref{t42} and Lemma~\ref{l44}.
 			\begin{theorem}
 				\label{t43} Suppose that  the $\sigma$-superlogarithmic property on $(0,q_o)$-forms holds.  Then, the following holds uniformly in $\delta$.
 				\begin{enumerate}
 					\item[(i)] 
 					For any  $u\in {\mathcal A}^{+,s\sigma+2-r}_{\zeta;0,q}\cap {\mathcal E}^{s,2-r}_{\zeta,\tilde\zeta;0,q}$ with $q\ge q_0$, we have
 					\[\begin{aligned}
 					\no{u}^2_{{\mathcal A}^{+,s\sigma-r}_\zeta}
 					\lesssim&\no{\Box_b^\delta u}^2_{{\mathcal A}^{+,s\sigma-r}_\zeta}+\no{u}^2_{{\mathcal E}_{\zeta,\tilde{\zeta}}^{s,2-r}}.
 					\end{aligned}\]
 					\item[ii.] For any $u\in {\mathcal A}^{-,s\sigma+2-r}_{\zeta;0,q}\cap {\mathcal E}^{s,2-r}_{\zeta,\tilde\zeta;0,q}$ with $q\ge q_0$, we have
 					\[\begin{aligned}
 				\no{u}^2_{{\mathcal A}^{-,s\sigma-r}_\zeta}
 				\lesssim&\no{\Box_b^\delta u}^2_{{\mathcal A}^{+,s\sigma-r}_\zeta}+\no{u}^2_{{\mathcal E}_{\zeta,\tilde{\zeta}}^{s,2-r}}.
 				\end{aligned}\]
 				\end{enumerate}	
 			\end{theorem}
 	We finish this section by the elliptic estimate for $0$-microlocalization. 
 	\begin{proposition}\label{Psi0}
 		Let $\Psi^0_0\prec \Psi^0_1$ be a pair of homogeneous  pseudodifferential operators of order zero whose symbols have support in  in $\mathcal C^0$ and let $\zeta\prec \zeta_1\prec \zeta_2$ be a triple of cutoff functions. Then, for any $s,s_o\ge 0$ and $\delta\ge 0$, we have
 		\begin{eqnarray}	
 		\no{\Psi^0_0\zeta u}^2_{H^{s+1}}&\lesssim&\no{\Psi^0_1\zeta_1 \dib_bu}^2_{H^s}+\no{\Psi^0_1\zeta_1 \dib_b^*u}^2_{H^s}+\no{\zeta_2 u}^2_{H^{-s_o}},\\
 		 		\no{\Psi^0_0\zeta u}^2_{H^{s+2}}&\lesssim& \no{\Psi^0_1\zeta_1 \Box_b u}^2_{H^s}+\no{\zeta_2 u}^2_{H^{-s_o}},\\
 		\no{\Psi^0_0\zeta u}^2_{H^{s+2}}&\lesssim& \no{\Psi^0_1\zeta_1 \Box_b^\delta u}^2_{H^s}+\no{\zeta_2 u}^2_{H^{-s_o}},
 		\end{eqnarray}	
 		for all $u\in C^\infty_{0,q}(\supp(\zeta_1))\cap H^{-s_o}_{0,q}(M)$ with any $0\le q\le n$.
 	\end{proposition}
 	The proof of this proposition is omitted since it   follows straightforwardly from the elliptic estimate for $0$-microlocal.

 				\section{Proof of the main theorems}\label{s4}
 			\subsection{Local regularity of the complex Green operator}\label{s4.1}
 					Let $Q_b$ be defined by 
 					$$Q_b(u,u)=\no{\dib_b u}^2+\no{\dib_b^*u}^2,\quad u\in \T{Dom}(\dib_b)\cap\T{Dom}(\dib_b).$$
 				The closed range of $\dib_b$ in $L^2$-spaces for all degrees of forms implies that 
 				\begin{eqnarray}
 					\label{closed range}
 					\no{u}^2\le c\left(Q_b(u,u)+\no{H_qu}^2\right),
\end{eqnarray} for all $(0,q)$ forms $u\in \T{Dom}(\dib_b)\cap \T{Dom}(\dib_b^*)\cap L^2_{0,q}(M)$ where $H_q$ is the project onto $\H_{0,q}$ and $0\le q\le n$. Here we recall that $\H_{0,q}$ is the space of harmonic forms of degree $(0,q)$. As a consequence of \eqref{closed range},  $\Box_{b}$ is invertible over $\H_{0,q}^\perp$ and its inverse $G_q$ is $L^2$ bounded. as mentioned in Section~1, we extend the operator $G_q$ to a linear operator on $L^2_{0,q}(M)$ be setting it equal to zero on $\H_{0,q}$. Then $G_q$ is bounded and self adjoint in $L^2_{0,q}(M)$, and 
$$\Box_bG_q=I-H_q.$$
As a consequence from the $L^2$ boundedness of $G_q$, the related operators $\dib_b^* G_q$ , $G_q\dib^*$, $\dib_b G_q$, $G_q\dib_b$, $I-\dib_b^* \dib_bG_{q}$, $I-\dib_b^*G_{q}\dib_b$, $I-\dib_b \dib_b^*G_{q}$, $I-\dib_bG_{q}\dib_b^*$, $\dib_bG_q^2\dib_b^*$  and $\dib_b^*G_q^2\dib_b$ are also $L^2$ bounded.

For higher regularity of $G_q$, by Theorem \ref{t42}, Proposition~\ref{Psi0} and the comparison of the $\no{\cdot}_{{\mathcal A}^{s\sigma}_\zeta}$ with $\no{\zeta_0\cdot}_{H^s}$ and $\no{\zeta_1\cdot}_{H^s}$ and the  the choice of cutoff functions $\zeta_0\prec\sigma\prec\zeta\prec\tilde\zeta\prec\zeta_1\prec\zeta_2$ , we obtain the  a-priori estimate in $H^s$: for any $s\ge0$, there exists $c_{s}>0$ so that for every $(0,q)$ form $u$ smooth in $U$,
 				$$\no{\zeta_0u}^2_{H^s}\le c_s\left(\no{\zeta_1\Box_b u}^2_{H^s}+\no{u}^2_{L^2}\right).$$
 			for all $u\in C^\infty_{0,q}(\supp(\zeta_1))\cap L^2_{0,q}(M)$	with $q_0\le q\le n-q_0$. Unlike the $L^2$-estimate  \eqref{closed range}, this  does not imply $\zeta_0 G_q\varphi\in H^s_{0,q}(M)$ for $\varphi\in H^s_{0,q}(\supp(\zeta_1))\cap L^2_{0,q}(M)$ since we do not know whether $u=G_q\varphi\in {\mathcal A}^{s\sigma}_{\zeta;0,q}$ even that $\varphi$ is smooth. We need to work with the family of regularized operators $G_q^\delta$ defined as follows. \\
 			
 				For a small $\delta>0$, we define 
 				$$Q_b^\delta(u,v):=Q_b(u,v)+\delta Q_d(u,v) \quad u,v\in H_{0,q}^1(M),$$
 					where  $Q_d(u,v)=(du,dv)+(d^*u,d^*v)+(u,v)$ the Hermitian inner product associated to the de Rham exterior derivative $d$. 
 						This quadratic form $Q_b^\delta$ gives a unique, self-adjoint in $L^2$, elliptic operator $\Box_b^\delta:=\Box_b+\delta\left(\Box_d+I\right)$ which is mentioned in \S\ref{s3.3}. Since 
 						$\delta\no{u}^2\le Q^\delta_b(u,u)$ for all $u\in H^1_{0,q}(M)$, $\Box_b^\delta$ has an inverse $G_{q}^\delta$ such that $$\Box^\delta_bG^\delta_q=I\quad\T{since $\ker(\Box_b^\delta)=\{0\}$}.$$
 							Furthermore, by the theory of elliptic regularity, it is well-known that $G^\delta_q$ is locally regular. More precisely, if $\varphi\in L^2_{0,q}(M)\cap H^s_{0,q}(V)$ then $G_q^\delta\varphi\in H^{s+2}_{0,q}(V)$ for any open set $V$ of $M$. Consequently, $G^\delta_q$ is globally regular, i.e., its mapping is continuous in $C^\infty_{0,q}(M)$.\\
 							
 				Let $H_{0,q}^{-\infty}(M)=\cup_{s\in\R}H_{0,q}^s(M)$ be the dual space of $C^\infty_{0,q}(M)$. Since $G_q^\delta$ is self-adjoint  and its mapping is continuous in the $C^\infty_{0,q}(M)$ topology, it extends to a  continuous in the space $H_{0,q}^{-\infty}(M)$. Thus, the local regularity property can be replaced as: if $\varphi\in H^{-\infty}_{0,q}(M)\cap H^s_{0,q}(V)$ then $G_q^\delta\varphi\in H^{s+2}_{0,q}(V)$ for any open set $V$ of $M$. The following proposition is to show that $G_q\varphi\in {\mathcal A}^{\pm,s\sigma-r}_\zeta$ by the advantage of $G^\delta_q\varphi.$

 		\begin{theorem}\label{t43b}Suppose that the $\sigma$-superlogarithmic property on $(0,q_0)$-forms holds and $\sigma\prec \zeta\prec\tilde\zeta\prec \zeta_1\prec\zeta_2$.  
 		Let  $\varphi\in H^{-\infty}_{0,q}(M)\cap C^\infty_{0,q}(\supp(\zeta_1))$ such that  $\no{\zeta_2G^\delta_q\varphi}^2_{H^{-s_o}}<\infty $ uniformly in $\delta$ for some $s_0\ge 0$.  Then, for any $r\ge s_o+2$ and $s\ge 0$, we have 
 		\begin{enumerate}
 			\item[(i)] $G_q\varphi\in {\mathcal A}^{+,s\sigma-r}_{\zeta;0,q}$ for  $q_0\le q\le n$; 
 			 \item[(ii)] $G_q\varphi\in {\mathcal A}^{-,s\sigma-r}_{\zeta;0,q}$ for  $0\le q\le n-q_0$;
 			 \item[(iii)] $G_q\varphi\in {\mathcal E}^{s,-r}_{\zeta,\tilde\zeta;0,q}$ for all $q=0,\dots,n$. 
 		\end{enumerate}   
 		\end{theorem}	
 		\begin{proof} By the local regularity of $G^\delta_q$, we have  $G^\delta_q\varphi\in C^\infty_{0,q}(\supp(\zeta_1))$. Consequently, $G^\delta_q\varphi\in {\mathcal A}^{+,s\sigma-r+2}_{\zeta;0,q}\cap {\mathcal E}^{s,-r+2}_{\zeta,\tilde{\zeta};0,q}$ for all $s,r\ge 0$ with $r\ge s_o+2$.
 		So we can apply Theorem~\ref{t43}(i) and Proposition~\ref{Psi0}(iii) for $u=G_q^\delta\varphi$ with $q\ge q_o$ to get the uniform estimate in $\delta$
 		\[\begin{aligned}
 		\no{G^\delta_q\varphi}^2_{{\mathcal A}_\zeta^{+,s\sigma-r}}\lesssim& \no{\varphi}^2_{{\mathcal A}_\zeta^{+,s\sigma-r}}+\no{G^\delta_q\varphi}^2_{{\mathcal E}^{s,-r+2}_{\zeta,\tilde{\zeta}}}\\
 		=&\no{\varphi}^2_{{\mathcal A}_\zeta^{+,s\sigma-r}}+\no{\tilde \Psi^0\zeta G^\delta_q\varphi}^2_{H^{s-r+2}}+\no{\tilde\zeta G^\delta_q\varphi}^2_{H^{-r+2}}\\
 		\lesssim&\no{\zeta_1\varphi}^2_{H^{s-r}}+\no{\zeta_2G_q^\delta \varphi}^2_{H^{-r+2}}.
 		\end{aligned}\]
 		  Thus the assumption that $\no{\zeta_2G^\delta_q\varphi}^2_{H^{-r+2}}$ is bounded uniformly in $\delta$ forces  $\no{G^\delta_q\varphi}^2_{{\mathcal A}_\zeta^{+,s\sigma-r}}$  to be  also uniform bounded in 
 		$\delta$. Consequently, the family $\{ G^{\delta} _{q}\varphi\}_{0<\delta \le 1}$ has a subsequence which converges weakly to $\hat u$ in the norm $\no{\cdot}_{{\mathcal A}^{+,s\sigma-r}_\zeta}$. We claim $\hat u= G_q\varphi$ in $V_0^{\zeta_2}$ the interior of $\{x\in M: \zeta_2(x)=1\}$. 
 		Let $ Q_{b,V_0^{\zeta_2}}$, be the restrictions of  $ Q_b$ to $V_0^{\zeta_2}$. 
 		For any $v\in C^\infty_{0,q}(M)$ supported in $V_0^{\zeta_2}$,
 		we have 
 		\[\begin{aligned}
 		|Q_{b,V_0^{\zeta_2}}(G_q^\delta \varphi-G_q\varphi,v)|=&|Q_{b}(G_q^\delta \varphi-G_q\varphi,v)|=\delta |Q_d(G^\delta_q \varphi, v)|=\delta |Q_d(\zeta_2G^\delta_q \varphi, v)|\\
 		\lesssim& \delta \no{\zeta_2G^\delta_q\varphi}_{H^{2-r}}\no{v}_{H^{r}}\lesssim \delta \no{v}_{H^{r}}\to 0\quad\T{as}\quad \delta\to 0.
 		\end{aligned}\] It means that $G^\delta_q\varphi$ converges to $G_q\varphi$ weakly in the $Q_{b,V_0^{\zeta_2}}$-norm. Therefore, we must have $\hat u=G_q\varphi$ in $V_0^{\zeta_2}\supset\supp(\zeta)$ and hence $$\no{G_q\varphi}_{{\mathcal A}^{+,s\sigma-r}_\zeta}\le \liminf \no{G^\delta_q\varphi}_{{\mathcal A}^{+,s\sigma-r}_\zeta}<\infty.$$
 		Therefore, $G_q\varphi\in {\mathcal A}^{+,s\sigma-r}_{\zeta;0,q}$ for  $q\ge q_0$. Similarly, $G_q\varphi\in {\mathcal A}^{-,s\sigma-r}_{\zeta;0,q}$ for  $q\le n-q_0$; and  $ G_q\varphi\in {\mathcal E}^{s,-r}_{\zeta,\tilde\zeta;0,q}$ for all $0\le q\le n$. 	
 		\end{proof}
 		In the following theorem we give the proof of local hypoellipticity of $\Box_b$ (the first part of Theorem~\ref{maintheorem1}).
\begin{theorem}\label{t44}
	Let $M$ be a pseudoconvex CR manifold of dimension $(2n+1)$ with $n\ge 2$ 
	such that  $\dib_b$ has closed range in $L^2$ spaces for all degrees of forms. 
	Suppose that  the $\sigma$-superlogarithmic property   holds on $(0,q_0)$-forms. 	Then, for all $q_0\le q\le n-q_0$, if $u\in L^2_{0,q}(M)$ such that $\Box_bu\in L^2_{0,q}(M)\cap C^\infty_{0,q}(V_1^\sigma)$ then  $u\in C^\infty_{0,q}(V_0^\sigma) $. 
\end{theorem}
\begin{proof} Set $\varphi:=\Box_bu\in \H^\perp_{0,q}(M)\cap  L^2_{0,q}(M)\cap C^\infty_{0,q}(V_1^\sigma)$. Then $u=G_q\varphi+H_qu$ where we recall that $H_q$ is the project onto $\H_{0,q}(M)$.
In order to apply Theorem~\ref{t43} above for this $G_q\varphi$, we have to show that  that $G^\delta_q\varphi$ is also  bounded in $L_{0,q}^2(M)$ uniformly in $\delta$.
 					Since $\H_{0,q}$	is finite dimensional if $1\le q\le n-1$, the inequality \ref{closed range} implies
 					\begin{eqnarray}
 					\label{L2delta}	\no{u}^2_{L^2}\lesssim Q_b^\delta(u,u)+\no{H^\delta_qu}_{L^2}^2
 					\end{eqnarray}
 					uniformly in $\delta$ where $H_q^\delta$ is the project onto $\H^\delta_{0,q}=\H_{0,q}\cap \{u: Q_d(u,u)=0\}=\{0\}$ (see \cite[Lemma 5.3]{KhRa16a} for the proof of this claim). This implies
 					\begin{eqnarray}
 					\label{L2}
 					\no{G^\delta_q\varphi}_{L^2}\lesssim \no{\varphi}_{L^2}
 					\end{eqnarray}
 					holds uniformly in $\delta$ for any $\varphi\in L^2_{0,q}(M)$ with $1\le q\le n-1$. Thus, $\no{\zeta_2G^\delta_q\varphi}^2_{H^{0}}<\infty $ uniformly in $\delta$. Now, we apply Theorem~\ref{t43b} for $s_o=0$ and $r=2$  to obtain that
 					\begin{enumerate}
 						\item[(i)] $G_q\varphi\in {\mathcal A}_{\zeta;0,q}^{+,s\sigma-2}$ for $q_0\le q\le n-1$;
 						\item[(ii)]  $G_q\varphi\in {\mathcal A}_{\zeta;0,q}^{-,s\sigma-2}$ for $1\le q\le n- q_0$;
 						\item[(iii)] $G_q\varphi\in {\mathcal E}^{s,-2}_{\zeta,\tilde{\zeta};0,q}$ for $1\le q\le n-1$.
 					\end{enumerate}  Subsisting $G_q\varphi\in {\mathcal A}_{\zeta;0,q}^{+,s\sigma-2}\cap {\mathcal A}_{\zeta;0,q}^{-,s\sigma-2}\cap {\mathcal E}^{s,-2}_{\zeta,\tilde{\zeta};0,q} $ with $q_o\le q\le n-q_o$ into the estimates in Theorem~\ref{t41}, we get
 					\begin{eqnarray}\begin{aligned}
 					\no{\zeta_0u}^2_{H^{s-4}}\lesssim&	\no{\zeta_0G_q\varphi}^2_{H^{s-4}}+\no{\zeta_0H_qu}^2_{H^{s-4}}\\
 					\lesssim& \no{G_q\varphi}^2_{{\mathcal A}^{+,s\sigma-4}_{\zeta}}+\no{G_q\varphi}^2_{{\mathcal A}^{-,s\sigma-4}_{\zeta}}+\no{G_q\varphi}^2_{{\mathcal E}^{s-4}_{\zeta,\tilde{\zeta}}}+\no{H_qu}^2_{H^{s-4}}\\
 					\lesssim& \no{\varphi}^2_{{\mathcal A}^{+,s\sigma-4}_{\zeta}}+\no{\varphi}^2_{{\mathcal A}^{-,s\sigma-4}_{\zeta}}+\no{G_q\varphi}^2_{{\mathcal E}^{s-4}_{\zeta,\tilde{\zeta}}}+\no{u}^2_{L^2}\\
 					\lesssim& \no{\zeta_1\varphi}^2_{H^{s-4}}+\no{G_q\varphi}^2_{L^2}+\no{u}^2_{L^2}\\
 					\lesssim&\no{\zeta_1\Box_bu}^2_{H^{s-4}}+\no{u}^2_{L^2}.\\
 				\end{aligned}	\end{eqnarray}
 				Here, we have use $\no{H_qu}^2_{H^{s-4}}\lesssim \no{u}^2_{L^2}$ since $\H_{0,q}(M)$ is finite dimensional. We conclude that $\zeta_0u\in H^{s-4}$ for any $s\ge 0$ and any $\zeta_0\in V_0^\sigma$, this proves Theorem~\ref{t44}. 
 			\end{proof}
 			\begin{remark}As in the proof of this theorem, we see that 
 			even {\it a-priori} estimates hold on the top degrees, we only have the regularity of $G_q$ excluding the top degrees $q=0$ and $q=n$ since 	 $\H_{0,q}$ is finte dimensional only for $1\le q\le n-1$.  
 			\end{remark}
 			The local regularity of $G_q$ is equivalent to the local hypoellipticity of $\Box_b$ and implies the local regularity of $\dib_bG_q$ and other operators but with loss derivatives in $H^s$-spaces. In the following we prove the exact $H^s$ regularity or the $H^s$ regularity with a small gain derivative for the relative operators of $G_q$.	
 			\begin{theorem}
 				\label{t45}  Let $M$ be a pseudoconvex CR manifold of dimension $(2n+1)$, $n\ge 2$
 				such that  $\dib_b$ has closed range in $L^2$ spaces for all degrees of forms. 
 				Suppose that  the $\sigma$-superlogarithmic property with independent rate $f$ holds on $(0,q_0)$-forms. 	Then, for all $s\ge0$, $q_0\le q\le n-q_0$ and $\zeta_0\prec\sigma\prec \zeta_1$, the following holds.
 				\begin{enumerate}
 					\item[(i)] If $\varphi\in L^2_{0,q}(M)\cap H_{0,q}^s(\supp(\zeta_1))$,  then 
 					\begin{eqnarray}
 					\label{main-est10}
 					\begin{aligned}
 					&\no{f^2(\Lambda)\zeta_0G_q\varphi}_{H^{s}}+\no{f(\Lambda)\zeta_0\dib_bG_q\varphi}_{H^{s}}+\no{f(\Lambda)\zeta_0\dib_b^*G_q\varphi}_{H^{s}}\\
 					&+\no{\zeta_0\dib_b\dib_b^*G_q\varphi}_{H^{s}}+\no{\zeta_0\dib_b^*\dib_bG_q\varphi}_{H^{s}}
 					\lesssim \no{ \zeta_1 \varphi}_{H^{s}}+\no{\varphi}_{L^2}.
 					\end{aligned}
 					\end{eqnarray}
 					\item[(ii)] If $\varphi\in L^2_{q-1}(M)\cap H^s_{0,q-1}(\supp(\zeta_1))$ then 
 				
 					\begin{eqnarray}
 					\label{main-est20}
 					\begin{aligned}
 					\no{f^2(\Lambda)\zeta_0\dib_b^*G_{q}^2\dib_b\varphi}_{H^{s}}+&\no{f(\Lambda)\zeta_0G_{q}\dib_b\varphi}_{H^{s}} 				+\no{\zeta_0(I-\dib_b^*G_{q}\dib_b)\varphi}_{H^{s}}\\
 					& \no{ \zeta_1 \varphi}_{H^{s}}+\no{\varphi}_{L^2}.
 					\end{aligned}
 					\end{eqnarray}
 					\item[(iii)] If  $\varphi\in L^2_{0,q+1}(M)\cap H^s_{0,q+1}(\supp(\zeta_1))$
 					\begin{eqnarray}
 					\label{main-est30}
 					\begin{aligned}
 					\no{f^2(\Lambda)\zeta_0\dib_bG_{q}^2\dib^*_b\varphi}_{H^{s}}+&\no{f(\Lambda)\zeta_0G_{q}\dib_b^*\varphi}_{H^{s}} 				+\no{\zeta_0(I-\dib_bG_{q}\dib_b^*)\varphi}_{H^{s}}\\
 					& \no{ \zeta_1 \varphi}_{H^{s}}+\no{\varphi}_{L^2}.
 					\end{aligned}
 					\end{eqnarray}
 				\end{enumerate} 	
 			\end{theorem}
 			\begin{proof}
 				In the proof of Theorem~\ref{t44} we have already proven that $G_q\varphi\in {\mathcal A}_{\zeta;0,q}^{+,s\sigma-2}$ with $q_0\le q\le n-1$, $G_q\varphi\in {\mathcal A}_{\zeta;0,q}^{-,s\sigma-2}$ with $1\le q\le n- q_0$, and also $G_q\varphi\in {\mathcal E}^{s,-2}_{\zeta,\tilde{\zeta};0,q}$ for $1\le q\le n-1$ if provided  $\varphi\in L^2_{0,q}(M)\cap H_{0,q}^s(\supp(\zeta_1))$.
 		Hence, $G_q\varphi$ satisfies  the hypothesis of Theorem~\ref{t41} for $r=4$, therefore, the estimates
 			\begin{eqnarray}
 			\label{z1}\begin{aligned}
 			\no{G_q\varphi}^2_{f^2\T-{\mathcal A}^{+,s\sigma-4}_\zeta}+&\no{\dib_bG_q\varphi}^2_{f\T-{\mathcal A}^{+,s\sigma-4}_\zeta}+\no{\dib_b^*G_q\varphi}^2_{f\T-{\mathcal A}^{+,s\sigma-4}_\zeta}\\
 			+&\no{\dib_b^*\dib_bG_q\varphi}^2_{{\mathcal A}^{+,s\sigma-4}_\zeta}+\no{\dib_b\dib_b^*G_q\varphi}^2_{{\mathcal A}^{+,s\sigma-4}_\zeta}
 			\lesssim\no{\zeta_1\varphi}^2_{H^{s-4}}+\no{\varphi}^2_{L^2}
 			\end{aligned}
 			\end{eqnarray}
 		holds on degrees $q_0 \le q\le n-1 $;	and dually, 
 		\begin{eqnarray}
 		\label{z2}\begin{aligned}
 			\no{G_q\varphi}^2_{f^2\T-{\mathcal A}^{-,s\sigma-4}_\zeta}+&\no{\dib_bG_q\varphi}^2_{f\T-{\mathcal A}^{-,s\sigma-4}_\zeta}+\no{\dib_b^*G_q\varphi}^2_{f\T-{\mathcal A}^{-,s\sigma-4}_\zeta}\\
 			+&\no{\dib_b^*\dib_bG_q\varphi}^2_{{\mathcal A}^{-,s\sigma-4}_\zeta}+\no{\dib_b\dib_b^*G_q\varphi}^2_{{\mathcal A}^{-,s\sigma-4}_\zeta}
 			\lesssim\no{\zeta_1\varphi}^2_{H^{s-4}}+\no{\varphi}^2_{L^2}\\
 			\end{aligned}	\end{eqnarray}
 			holds on degrees	$1\le q\le n-q_0.$\\
 			\noindent \underline{\it Proof of (i).} Fix $q_0\le q\le n-q_0$. If $\varphi\in L^2_{0,q}(M)\cap H_{0,q}^s(\supp(\zeta_1))$, then both \eqref{z1} and \eqref{z2} hold, so taking summation to get 
 			\begin{eqnarray}
 			\label{z3}\begin{aligned}
 			\no{f^2(\La)\zeta_0G_q\varphi}^2_{H^{s-4}}&+\no{f(\La)\zeta_0\dib_bG_q\varphi}^2_{H^{s-4}}+\no{f(\La)\zeta_0\dib_b^*G_q\varphi}^2_{H^{s-4}}\\
 			&+\no{\zeta_0\dib_b^*\dib_bG_q\varphi}^2_{H^{s-4}}+\no{\zeta_0\dib_b\dib_b^*G_q\varphi}^2_{H^{s-4}}
 			\lesssim\no{\zeta_1\varphi}^2_{H^{s-4}}+\no{\varphi}^2_{L^2}.
 			\end{aligned}		\end{eqnarray}
 				Since $H^s_{0,q}(M)$ is dense in $H^{s-4}_{0,q}(M)$, the estimate \eqref{z3} is still holds for $\varphi\in L^2_{0,q}(M)\cap H_{0,q}^{s-4}(\supp(\zeta_1))$. This proves the first estimate in Theorem~\ref{t45}.\\
 	\noindent \underline{\it Proof of (ii).} 
Now we deal $\varphi\in L^2_{q-1}(M)\cap H^s_{0,q-1}(\supp(\zeta_1))$. By \eqref{L2delta}, $\no{G^\delta_{q}\dib_b\varphi}_{L^2}$ is uniformly bounded and hence, by Theorem~\ref{t44}, $G_{q}\dib_b\varphi\in {\mathcal A}^{\pm,s\sigma-2}_{\zeta;0,q}\cap {\mathcal E}^{s,-2}_{\zeta,\tilde{\zeta};0,q}$. We first give estimate on the positive norm $\no{\cdot}_{{\mathcal A}^{+,s\sigma-4}_\zeta}$. We use Lemma~\ref{l42}.(i) for $u=G_{q}\dib_b\varphi$ and with the choice $g(t)=t^{-4}$,
\begin{eqnarray}\label{a11}
\begin{aligned}
\no{G_{q}\dib_b\varphi}^2_{f\T-{\mathcal A}_\zeta^{+,s\sigma-4}}+&\no{G_{q}\dib_b\varphi}^2_{f_0\ln\T-\di_b\sigma\lrcorner {\mathcal A}_\zeta^{+,s\sigma-4}}+\no{\dib_b^*G_{q}\dib_b\varphi}^2_{{\mathcal A}_\zeta^{+,s\sigma-4}}\\
\lesssim&\Re\,(\dib_b\dib_b^*G_{q}\dib_b\varphi,G_{q}\dib_b\varphi)_{{\mathcal A}_\zeta^{+,s\sigma-4}}+\no{\dib_bG_{q}\dib_b\varphi}+^2_{{\mathcal A}^{+,s\sigma-4}_\zeta}+\no{G_{q}\dib_b\varphi}^2_{{\mathcal E}^{s,-3}_{\zeta,\tilde{\zeta}}}\\
=&\Re\,(\dib_b\varphi,G_{q}\dib_b\varphi)_{{\mathcal A}_\zeta^{+,s\sigma-4}}+\no{\varphi}^2_{{\mathcal E}^{s,-4}}+\no{G_q\dib_b\varphi}^2_{{\mathcal E}^{s,-3}}.
\end{aligned}
\end{eqnarray} 
Using Lemma~\ref{l41}.(i) for $v=\varphi$, $w=G_{q}\dib_b\varphi$ with the choice $g=t^{-4}$ and $a=\eps=1$ 
\begin{eqnarray}\label{a12}
\begin{aligned}\Re\,(\dib_b\varphi,G_q\dib_b\varphi)_{{\mathcal A}_\zeta^{+,s\sigma-4}}\le& \Re\,(\varphi,\dib_b^*G_q\dib_b\varphi)_{{\mathcal A}_\zeta^{+,s\sigma-4}}+\no{\varphi}_{{\mathcal A}_\zeta^{+,s\sigma-4}}
+\no{\varphi}_{{\mathcal E}^{s,-4}_{\zeta,\tilde{\zeta}}}\\
+&c\left(\no{G_{q}\dib_b\varphi}^2_{{\mathcal A}_\zeta^{+,s\sigma-4}}+\no{G_{q}\dib_b\varphi}^2_{\ln\T-\di_b\sigma\lrcorner {\mathcal A}_\zeta^{+,s\sigma-4}}+\no{G_{q}\dib_b\varphi}^2_{{\mathcal E}^{s,-4}_{\zeta,\tilde{\zeta}}}\right)\\
\le&\epsilon\no{\dib_b^*G_q\dib_b\varphi}^2_{{\mathcal A}_\zeta^{+,s\sigma-4}}+c_\eps\no{\varphi}_{{\mathcal A}_\zeta^{+,s\sigma-4}}
+\no{\varphi}_{{\mathcal E}^{s,-4}_{\zeta,\tilde{\zeta}}}\\
+&c\left(\no{G_{q}\dib_b\varphi}^2_{{\mathcal A}_\zeta^{+,s\sigma-4}}+\no{G_{q}\dib_b\varphi}^2_{\ln\T-\di_b\sigma\lrcorner {\mathcal A}_\zeta^{+,s\sigma-4}}+\no{G_{q}\dib_b\varphi}^2_{{\mathcal E}^{s,-4}_{\zeta,\tilde{\zeta}}}\right).
\end{aligned}
\end{eqnarray} 
Combining \eqref{a11} and \eqref{a12} and the elliptic estimate in the error norm, we conclude that  
\begin{eqnarray}\label{a13}
\begin{aligned}
\no{G_{q}\dib_b\varphi}^2_{f\T-{\mathcal A}_\zeta^{+,s\sigma-4}}+\no{\dib_b^*G_{q_o}\dib_b\varphi}^2_{{\mathcal A}_\zeta^{+,s\sigma-4}}
\lesssim\no{\zeta_1\varphi}^2_{H^{s-4}}+\no{\varphi}^2_{L^2}.
\end{aligned}
\end{eqnarray}
Now we  give estimate on the negative norm $\no{\cdot}_{{\mathcal A}^{-,s\sigma-4}_\zeta}$. We consider two case of $q$. If $1\le q-1\le n-q_0$,
 we use the fact that $G_{q}\dib_b=\dib_bG_{q-1}$ and \eqref{z2} holds for all degree between $1$ and $n-q_0$, then
\begin{eqnarray}\label{a14}
\begin{aligned}
\no{G_{q}\dib_b\varphi}^2_{f\T-{\mathcal A}_\zeta^{-,s\sigma-4}}+\no{(I-\dib_b^*G_{q}\dib_b)\varphi}^2_{{\mathcal A}_\zeta^{-,s\sigma-4}}
\lesssim\no{\zeta_1\varphi}^2_{H^{s-4}}+\no{\varphi}^2_{L^2}.
\end{aligned}
\end{eqnarray}
Otherwise, $q=1$, we give a directly proof of \eqref{a14} as the following
\begin{eqnarray}\label{a15}
\begin{aligned}\no{G_{1}\dib_b\varphi}^2_{f\T-{\mathcal A}_\zeta^{-,s\sigma-4}}\lesssim& 
\no{\dib_b^*G_{1}\dib_b\varphi}^2_{{\mathcal A}_\zeta^{-,s\sigma-4}}+\no{G_{1}\dib_b\varphi}^2_{{\mathcal E}^{s,-4}_{\zeta,\tilde{\zeta}}}\\
\lesssim&\no{(I-\dib_b^*G_{1}\dib_b)\varphi}^2_{{\mathcal A}_\zeta^{-,s\sigma-4}}+\no{\varphi}^2_{{\mathcal A}_\zeta^{-,s\sigma-4}}+\no{G_{1}\dib_b\varphi}^2_{{\mathcal E}^{s,-4}_{\zeta,\tilde{\zeta}}}\\
\lesssim&\no{\varphi}^2_{{\mathcal A}_\zeta^{-,s\sigma-4}}+\no{(I-\dib_b^*G_{1}\dib_b)\varphi}^2_{{\mathcal E}^{s,-4}_{\zeta,\tilde{\zeta}}}+\no{G_{1}\dib_b\varphi}^2_{{\mathcal E}^{s,-4}_{\zeta,\tilde{\zeta}}}.
\end{aligned}
 \end{eqnarray}
 Here, we have used Theorem~\ref{t41}.(ii) twice: the first one for the first inequality and the second one for the third inequality with notice that $\dib_b(I-\dib_b^*G_1\dib_b)\varphi=\dib_b^*(I-\dib_b^*G_1\dib_b)\varphi=0$. So \eqref{a14} holds for both case of $q\le n-q_0$. \\

Finally, the estimate of the term $\dib_b^*G_{q}^2\dib_b\varphi$ follows by the estimates of $\dib_bG_{q}$ and $G_{q}\dib_b$. Summarizing, we get  
	\begin{eqnarray}
	\label{m20}
	\begin{aligned}
	\no{f^2(\Lambda)\zeta_0\dib_b^*G_{q}^2\dib_b\varphi}_{H^{s-4}}+&\no{f(\Lambda)\zeta_0G_{q}\dib_b\varphi}_{H^{s-4}} 				+\no{\zeta_0(I-\dib_b^*G_{q}\dib_b)\varphi}_{H^{s-4}}\\
	& \no{ \zeta_1 \varphi}_{H^{s-4}}+\no{\varphi}_{L^2},
	\end{aligned}
	\end{eqnarray}
	for all $\varphi\in L^2_{0,q-1}(M)\cap H^s_{0,q-1}(\supp(\zeta_1))$. By the density of $H^{s}_{0,q}(M)$ in $H^{s-4}_{0,q}(M)$, we can replace the assumption $\varphi\in L^2_{0,q-1}(M)\cap H^s_{0,q-1}(\supp(\zeta_1))$ by $\varphi\in L^2_{0,q-1}(M)\cap H^{s-4}_{0,q-1}(\zeta_1)$. This proves the part (ii) in Theorem~\eqref{t45}. The part (iii) follows analogously.
	\end{proof}
\begin{remark}
	 It is known that on the top degrees, the Green operators $G_0$ and $G_n$ agree with $\dib_b^*G_1^2\dib_b$ and $\dib_bG_{n-1}^2\dib_b^*$, respectively; and also $\dib_bG_0=G_1\dib_b$, $\dib_b^*G_n=G_{n-1}\dib_b^*$. Therefore, if $q_0=1$ then 		\begin{eqnarray}
	 \begin{aligned}
	 \no{f^2(\Lambda)\zeta_0G_{0}\varphi}_{H^{s}}+\no{f(\Lambda)\zeta_0\dib_bG_0\varphi}_{H^{s}}\lesssim
	  \no{ \zeta_1 \varphi}_{H^{s}}+\no{\varphi}_{L^2}
	 \end{aligned}
	 \end{eqnarray}
	 holds 	for all $\varphi\in L^2(M)\cap H^s(\supp(\zeta_1))$; and 
	 \begin{eqnarray}
	 \begin{aligned}
	 \no{f^2(\Lambda)\zeta_0G_{n}\varphi}_{H^{s}}+\no{f(\Lambda)\zeta_0\dib^*_bG_n\varphi}_{H^{s}}\lesssim \no{ \zeta_1 \varphi}_{H^{s}}+\no{\varphi}_{L^2}
	 \end{aligned}
	 \end{eqnarray}
	 holds 	for all $\varphi\in L_{0,n}^2(M)\cap H_{0,n}^s(\supp(\zeta_1))$.
\end{remark}
For the case $n\ge 2$, the local regularity of $G_0$ (resp. $G_n$) is inherited from the local regularity of $G_1$ (resp. $G_{n-1}$). In the case $n=1$, this procedure does not work. In this case, in order to use {a-priori} estimates, we assume that $G_0$ (and hence $G_1$) is globally regular.
	\begin{theorem}\label{t46}
		Let $M$ be a $3$-dimensional pseudoconvex CR manifold such that $\dib_b$ has closed range on functions and $G_0$ is globally regular. Suppose that  $\sigma$-superlogarithmic property with dependent rate $f$ holds on $(0,1)$-forms. Then, for $q=0,1$, if  $u\in L^2_{0,q}(M)$ such that $\Box_b u\in  L^2_{0,q}(M)\cap C^\infty_{0,q}(V_1^\sigma)$ then  $u\in C^\infty_{0,q}(V_0^\sigma) $.
		  Furthermore, 
		\begin{eqnarray}
		\label{main-est40}
		\begin{aligned}
		\no{f^2(\Lambda)\zeta_0G_{0}\varphi}_{H^{s}}+\no{f(\Lambda)\zeta_0\dib_bG_0\varphi}_{H^{s}}+\no{\zeta_0(I-\dib_b^*G_1\dib_b)\varphi}_{H^{s}}\lesssim
		\no{ \zeta_1 \varphi}_{H^{s}}+\no{\varphi}_{L^2};
		\end{aligned}
		\end{eqnarray}
		holds 	for all $\varphi\in L^2(M)\cap H^s(\supp(\zeta_1))$, and 
		\begin{eqnarray}
		\label{main-est50}
		\begin{aligned}
		\no{f^2(\Lambda)\zeta_0G_{1}\varphi}_{H^{s}}+\no{f(\Lambda)\zeta_0\dib^*_bG_1\varphi}_{H^{s}}+\no{\zeta_0(I-\dib_bG_0\dib_b^*)\varphi}_{H^{s}}\lesssim \no{ \zeta_1 \varphi}_{H^{s}}+\no{\varphi}_{L^2},
		\end{aligned}
		\end{eqnarray}
		holds 	for all $\varphi\in L_{0,1}^2(M)\cap H_{0,1}^s(\supp(\zeta_1))$.
	\end{theorem}
The idea of the proof of this theorem follows by the argument in \cite[page 70]{Str10}.
\begin{proof}  We remark that the global regularity property of $G_0$ implies that $\dib_bG_0$, $(I-\dib_bG_0\dib_b^*)$, $G_1$, $\dib_b^*G_1$ and $(I-\dib_b^*G_1\dib_b)$ are also global regularity. Thus, for $\varphi\in C^\infty(M)$,  we can use the {\it a-priori} estimates in Theorem~\ref{t41} and \ref{t42}, and follow the proof in Theorem~\ref{t45}(ii) (for the case $n\ge 2$, but it still hold for $n=1$) together with the closed range of $\dib_b$ to get \eqref{main-est40} and \eqref{main-est50} for this globally smooth  function $\varphi$. So we have to pass the smoothness of $\varphi$ from global to local.  In the following we only give the proof for $G_0$ since other operators follows analogously.  By the fact that the local hypoellipticity of $\Box_b$ on functions follows by the local regularity of $G_0$, we have obtained 
		\begin{eqnarray}
		\label{main-est51}\no{f^2(\La)\zeta_0G_0\varphi}^2_{H^s}\lesssim \no{\tilde \zeta \varphi}^2_{H^s}+\no{\varphi}^2_{L^2}	\end{eqnarray}
		for all $\varphi\in C^\infty(M)$.
		Here we replace $\zeta_1$ by $\tilde{\zeta}$ such that $\sigma\prec\tilde\zeta\prec\zeta_1$ since $\zeta_1$ can be chosen arbitrary (but dominating $\sigma$). \\

 For a general $\varphi\in H^s(\supp(\zeta_1))\cap L^2(M)$, we choose a sequence of functions $\varphi_m\in C^\infty(M)$ such that $\varphi_m\to \varphi$ in $L^2(M)$. Since $\supp(\zeta_1)$ and $\supp((1-\zeta_2))$ are disjoint, we use \eqref{main-est51} for the smooth datum function $(1-\zeta_2)(\varphi_l-\varphi_m)$ and get 
 $$\no{f^2(\La)\zeta_0G_0(1-\zeta_2)(\varphi_l-\varphi_m)}_{H^s}^2\lesssim \no{\varphi_l-\varphi_m}^2.$$
Consequently, the sequence $\{f^2(\La)\zeta_0G_0(1-\zeta_2)\varphi_m\}$ is Cauchy in $H^s(M)$. Since it convergences to $f^2(\La)\zeta_0G_0(1-\zeta_2)\varphi$ in $L^2(M)$, thus $f^2(\La)\zeta_0G_0(1-\zeta_2)\varphi\in H^s(M)$ and 
\begin{eqnarray}
\label{a121}\no{f^2(\La)\zeta_0G_0(1-\zeta_2)\varphi}_{H^s}^2\lesssim \no{\varphi}^2.	\end{eqnarray}
Since $\zeta_1\varphi\in C^\infty(M)$, we used \eqref{main-est51} for the datum $\zeta_1\varphi$,
\begin{eqnarray}
\label{a122}\no{f^2(\La)\zeta_0G_0\zeta_1\varphi}_{H^s}^2\lesssim \no{\zeta_1\tilde\zeta\varphi}^2_{H^s}+\no{\zeta_1\varphi}^2_{L^2}\lesssim\no{\zeta_1、\varphi}^2_{H^s}+\no{\varphi}^2_{L^2} .	\end{eqnarray}
Combining \eqref{a121} and \eqref{a122}, we get the desired estimate:
 $$\no{\zeta_0G_0\varphi}_{H^s}^2\lesssim\no{f^2(\La)\zeta_0G_0\zeta_1\varphi}_{H^s}^2+ \no{f^2(\La)\zeta_0G_0(1-\zeta_2)\varphi}_{H^s}^2\lesssim\no{\zeta_1\varphi}^2_{H^s}+\no{\varphi}^2_{L^2},$$
 for all $\varphi\in H^s(\supp(\zeta_1))\cap L^2(M)$.
 
 \end{proof}

\subsection{Smoothness of the complex Green kernel}\label{s4.2}

 Let $\G_q(\hat x,x)$,  $0\le q\le n$, be the integral kernel of the complex Green operator $G_q$, it means, $G_q$ has the integral representation 
 $$G_q\varphi(x)=\int_M\la\varphi(\hat x),\G_q(\hat x,x) dV(\hat x)=(\varphi( \bullet),\G_q( \bullet, x))_{L^2}$$
 provided the integral exists; this imposes regularity conditions on $\varphi$ and $\G_q$.
 Thus, $\G_q$ is a double distribution form of degree $(0,q;q,0)$ on the product manifold $M\times M$. $\G_q$ can be expressed in local coordinate $\hat U\times  U$ as
 $$\G_q(\hat x,x)=\sum_{\hat I,I\in \I_q}\G_{\hat I,I}(\hat x,x)\bom_{\hat I}(\hat x)\we \om_I(x). $$
 To show the smoothness of $\G_q$, we will prove that functions $\G_{\hat I,I}$ are smooth for all $\hat I,I\in \I_q$. Fix $\hat I\in \I_q$. Let $\delta_{\hat x_0}$ be the delta Dirac ``function" with support in $\hat x_0\in M$. We choose a  $(0,q)$-form datum $$\varphi_{\hat x_0}(\hat x)=(D^{\hat \alpha})^*\delta_{\hat x_0}(\hat x)\bom_{\hat I}(\hat x),$$
 where $(D^{\hat\alpha})^*$ is the $L^2$ adjoint of the differential
 operator  $D^{\hat{\alpha}}$.
 This data is smooth on $M\setminus\{\hat x_0\}$ and bounded in $H^{-(n+1+|\hat\alpha|)}$-norm since
 	\begin{eqnarray}\label{810}
 	\begin{aligned}
\no{\varphi_{\hat x_0}}_{H^{-(n+1+|\hat{\alpha}|)}}^2\lesssim & \no{\delta_{\hat x_0}}_{H^{-(n+1)}}^2\\ =& \int_{\R^{2n+1}}(1+|\xi|^2)^{-(n+1)}\F\{\delta_{\hat x_0}\}(\xi)d\xi\\
=&\int_{\R^{2n+1}}(1+|\xi|^2)^{-(n+1)}d\xi\lesssim 1.
 	\end{aligned}
 	\end{eqnarray}
 Thus, 
 \begin{eqnarray}\label{811a}
 \begin{aligned}G_q\varphi_{\hat x_0}(x)=&
 (\varphi_{\hat x_0}( \bullet),\G_q( \bullet,x))_{L^2}\\
 =&\sum_{I\in \I_q}\left((D^{\hat \alpha})^*\delta_{\hat x_0}( \bullet),\G_{\hat I, I}( \bullet,x)\right)_{L^2}\bom_I(x)\\
 =&\sum_{I\in \I_q}\left(\delta_{\hat x_0}( \bullet),D_{ \bullet}^{\hat\alpha}\G_{\hat I,I}( \bullet,x)\right)_{L^2}\bom_I(x)\\
 =&\sum_{I\in \I_q}\overline{D^{\hat \alpha}_{\hat x_0}\G_{I,\hat I}(\hat x_0,x)}\bom_I(x),
 \end{aligned}
 \end{eqnarray}
 and consequently,  
  \begin{eqnarray}\label{811b}
 \sum_{I\in \I_q}\left|D^\alpha_{x}D^{\hat{\alpha}}_{\hat x_0}\G_{\hat I,I}(\hat x_0,x)\right|^2=\left|\overline{D^\alpha_{x}} G_q\varphi_{\hat x_0}(x)\right|^2. \end{eqnarray}
Let $\zeta_0$ be a cutoff function  in $U$, we choose $x_0\in V_0^{\zeta_0}$ then by the Sobolev lemma
 \begin{eqnarray}\label{811c}\left|D^\alpha_{x_0}D^{\hat{\alpha}}_{\hat x_0}\G_{\hat I,I}(\hat x_0,x_0)\right|\le \sup_{x\in M}\left|\overline{D^\alpha_{x}} \zeta_0G_q\varphi_{\hat x_0}(x)\right|\lesssim\no{\zeta_0 G_q\varphi_{\hat x_0}}_{n+1+|\alpha|}, \end{eqnarray}
 for all $I,\hat I\in \I_q$. 
These equalities \eqref{811a}, \eqref{811b} and inequalities \eqref{811c} hold if provided that $\no{\zeta_0G_q\varphi_{\hat x_0}}_{H^{n+1+|\alpha|}}<\infty$. To show this boundedness, we need the hypothesis of weakly superlogarithmic properties. Theorem~\ref{maintheorem3} is splitted into two following theorems.
\begin{theorem}\label{tm1}
	 Let $M$ be a pseudoconvex, CR manifold of dimension $(2n+1)$ such that $\dib_b$ has closed range in $L^2$ spaces for all degrees of forms. Let $\sigma$ and $\hat\sigma$ be cutoff functions with disjoint supports ans suppose that the $\sigma$- and $\hat \sigma$-superlogarithmic properties on $(0,q_0)$ forms holds. 
	  Then 
	 	$$\mathcal G_{q}\in C^\infty_{0,q;q,0}\left(\left( V_0^\sigma\times V_0^{\hat\sigma}\right)\cup \left( V_0^{\hat\sigma}\times V_0^{\sigma}\right)\right)\quad\T{	 for all $q_0\le q\le n-q_0$}.$$	If $q_0=1$ and $n\ge 2$, or $q_0=n=1$ with the extra assumption that $G_0$ is globally regular, then this conclusion also holds for $q=0$ and $q=n$. 
	\end{theorem}
\begin{proof}
 We choose $x_0\in V_0^\sigma$ and $\hat x_0\in V_0^{\hat{\sigma}}$ and two triplets of cutoff functions $\{\zeta_j\}_{j=0}^2$ and $\{\hat\zeta_j\}_{j=0}^2$ such that $\zeta_0\prec\sigma\prec\zeta_1\prec\zeta_2$, $\hat\zeta_0\prec\hat\sigma\prec\hat\zeta_1\prec\hat\zeta_2$, $\supp(\zeta_2)\cap\supp(\hat\zeta_2)=\emptyset$, and $x_0\in V_0^{\zeta_0}$, $x_0\in V_0^{\hat \zeta_0}$. Since $\zeta_1\varphi_{\hat x_0}=0$, Theorem~\ref{t43b} implies that the desired inequality  $\no{\zeta_0 G_q\varphi_{\hat x_0}}_{n+1+|\alpha|}<\infty$ holds if  $\no{\zeta_2G_q^\delta\varphi_{\hat x_0}}_{L^2}<\infty$ uniformly in $\delta>0$.
By the self-adjoint property of $G^\delta$, the H\"older  inequality and the estimate \eqref{810}, we get 
 						
 					\begin{eqnarray}
 					\label{m200}\begin{aligned}
 					\no{\zeta_2G_q^\delta\varphi_{\hat x_0}}_{L^2}=&\sup\{|(\zeta_2G_q^\delta\varphi_{\hat x_0},v)_{L^2}|: \no{v}_{L^2}\le 1\}\\
 					=&\sup\{|(\varphi_{\hat x_0},\hat\zeta_0G_q^\delta\zeta_2v)_{L^2}|: \no{v}_{L^2}\le 1\}\\
 					\le& \sup\{\no{\varphi_{\hat x_0}}_{H^{-(n+1+|\hat{\alpha}|)}}\no{\hat \zeta_0G_q^\delta\zeta_2v}_{H^{n+1+|\hat\alpha|}}: \no{v}_{L^2}\le 1\}\\
 					\le& \sup\{\no{\hat \zeta_0G_q^\delta\zeta_2v}_{H^{n+1+|\hat\alpha|}}: \no{v}_{L^2}\le 1\}.
 					\end{aligned}
 					\end{eqnarray}
We apply Theorem~\ref{t43} for $u=G^\delta_q\zeta_2v$ with system of cutoff functions $\{\hat \zeta_j\}$ to obtain
 						\begin{eqnarray}
 						\label{m201}\begin{aligned}\no{\hat \zeta_0G_q^\delta\zeta_2v}^2_{H^{n+1+|\hat\alpha|}}\lesssim& \no{\hat \zeta_1\zeta_2v}^2_{H^{n+1+|\hat\alpha|}}+\no{\hat \zeta_2G_q^\delta\zeta_2v}^2_{L^2}\\
 						\le& \no{G_q^\delta\zeta_2v}_{L^2}\lesssim \no{\zeta_2v}^2_{L^2}\le \no{v}^2\le 1.
 						\end{aligned}
 						\end{eqnarray}
 						uniformly in $\delta$. From \eqref{m200} and \eqref{m201}, we get 
 							$\no{\zeta_2G_q^\delta\varphi_{\hat x_0}}_{L^2}$  is uniformly bounded to $\delta$. Therefore, $\no{\zeta_0 G_q\varphi_{\hat x_0}}_{n+1+|\alpha|}<\infty$ and hence by \eqref{811c},
 							$$\G_q\in C^\infty_{0,q;q,0}(V_0^{\hat \sigma}\times V_0^\sigma).$$
 							We also notice that $\G_q$ is smoothness on symmetric sets of $M\times M$ since $G_q$ is self adjoint. That means if $\G_q\in C^\infty_{0,q;q,0}(V_0^{\hat \sigma}\times V_0^\sigma)$ then  $\G_q\in C^\infty_{0,q;q,0}(V_0^{ \sigma}\times V_0^{\hat\sigma})$.
 							This proves Theorem~\ref{tm1}.

 					\end{proof}

 					\begin{theorem}\label{tm2} Let $M$ be a pseudoconvex, CR manifold of dimension $(2n+1)$ such that $\dib_b$ has closed range in $L^2$ spaces for all degrees of forms. Suppose that the $\sigma$-superlogarithmic property holds on $(0,q_0)$-forms and $G_q$ is globally regular for some $q_0\le q\le n-q_0$. Then 
 							$$\mathcal G_{q}\in C^\infty_{0,q;q,0}\left(\left(V^\sigma_0\times (M\setminus \overline{\supp(\sigma)})\right)\cup\left(  (M\setminus \overline{\supp(\sigma)})\times V_0^\sigma \right)\right).$$	If $q_0=q=1$ (resp. $q_0=1, q=n$), this conclusion also holds for $G_0$ and (resp. $G_n$).
 					\end{theorem}
 					\begin{proof}	
 				 We choose $x_0\in V_0^\sigma$ and $\hat x_0\in M\setminus \overline{\supp(\sigma)}$ and  $\{\zeta_j\}_{j=0}^2$  such that $\zeta_0\prec\sigma\prec\zeta_1\prec\zeta_2$, and $\hat x_0\in V_0^{\zeta_0}$, $\hat x_0\not\in\supp(\zeta_1)$.
 				 This means  $\zeta_1\varphi_{\hat x_0}=0$. By the global regularity of  $G_q$, it follows for any multi-indice $\hat\alpha$ there exists $m_{\hat\alpha}$ such that 
 				\begin{eqnarray}
 				\label{m202}\begin{aligned}\no{G_qv}^2_{H^{n+1+|\hat\alpha|}}\lesssim \no{v}^2_{H^{m_{\hat\alpha}}} \quad\T{for all  } v\in H^{m_{\hat\alpha}}_{0,q}(M);	\end{aligned}
 				\end{eqnarray}
 				 and  $G_q$ can be extended to a continuos in the space of $H^{-\infty}_{0,q}(M)$. Thus we can estimate $\zeta_2G_q\varphi_{\hat x_0}$ as follows
 					
 					\[\begin{aligned}
 					\no{\zeta_2G_q\varphi_{\hat x_0}}_{H^{-m_{\hat\alpha}}}=&\sup\{|(\zeta_2G_q\varphi_{\hat x_0},v)_{L^2}|: \no{v}_{H^{m_{\hat\alpha}}}\le 1\}\\
 					=&\sup\{|(\varphi_{\hat x_0},G_q\zeta_2v)_{L^2}|: \no{v}_{L^2}\le 1\}\\
 					\le& \sup\{\no{\varphi_{\hat x_0}}_{H^{-(n+1+|\hat{\alpha}|)}}\no{ G_q\zeta_2v}_{H^{n+1+|\hat\alpha|}}:\no{v}_{H^{m_{\hat\alpha}}}\le 1\}\lesssim 1, 				
 					\end{aligned}\]
 where the last estimate follows by \eqref{m202} and \eqref{810}. Thus we can apply Theorem~\ref{t42} for $u=G_q\zeta_2\varphi_{\hat x_0}$, $s=n+1+|\hat \alpha|$, $s_o=m_{\hat{\alpha}}$ and $r=s_o+2$, it follows
\[\begin{aligned}\no{\zeta_0G_q\varphi_{\hat x_0}}^2_{H^{n+1+|\alpha|}}\lesssim&\no{\zeta_1\varphi_{\hat x_0}}^2_{H^{n+1+|\alpha|}}+ \no{\zeta_2G_q\varphi_{\hat x_0}}^2_{H^{-s_o}}
\lesssim 1.
\end{aligned}\]
By \eqref{811c}, it follows
$$\G_q\in C^\infty_{0,q;q,0}\left(  (M\setminus \overline{\supp(\sigma)})\times V_0^\sigma \right).$$ 
The proof of Theorem~\ref{tm2} follows by the symmetric of $\G_q$.
\end{proof}

 						\section{The $\sigma$-superlogarithmic property on hypersurface models}\label{s5}
 						In this section we show that the $\sigma$-superlogarithmic property admits on the  hypersurfaces model \eqref{model5}  and prove Theorem~\ref{t1.9}. Let $M$ be the  hypersurfaces passing the origin defined by
 \begin{eqnarray}
 \label{e1}
 M=\left\{(z,z_{n+1})\in \C^n\times \C:\Im\,z_{n+1}=\sum_kH_k\left(\sum_{l\in I_k}|h_{l}(z)|^2\right)F_k\left(\sum_{l\in I_k}|Re\,h_{l}(z)|^2\right)\right\}.
 \end{eqnarray}
 Here, \begin{enumerate}
 	\item $H_k,F_k:\R^+\to \R^+$  are increasing and convex functions such that  $H_k(0)F_k(0)=0$ and. If $H_k(0)=0$ (resp. $F_k(0)=0$) it is added an extra assumption that $H_k(\delta)/\delta$ (resp. $F_k(\delta)/\delta$) is increasing. Here, we assume that $F_k(0)=0$, otherwise the calculation is simpler.
 	\item $h_{l}$ with $l=1,\dots,N$ are holomorphic functions in $\C^n$ with an isolated zero at the origin and $I_k$'s are blocks of $(1,\dots,N)$ such that $\cup_{k}I_k\subset \{1,\dots,N\}$. 
 \end{enumerate} 
 \begin{theorem}\label{t51m} Let $M$ be the hypersurface in $\C^{n+1}$ defined by \eqref{e1}. Then, if  
 	\begin{eqnarray}\label{ht51}
 	\lim_{\delta\to 0}\delta\ln\left(F_k(\delta^2)\right)=0\quad\T{ for all } k,
 	\end{eqnarray} then the origin of $\C^{n+1}$ is a weakly superlogarithmic point.
 \end{theorem} 
\begin{proof} 
	After the setting, the proof of this theorem is divided into three steps. In Step 1, we give a lower bounded of the Levi form of $M$ by the $(1,1)$-forms $$\M_{l}=\frac{H(|h_{l}|^2)F(|\Re\, h_{l}|^2)}{|\Re\, h_{l}|^2}\L_{|h_{l}|^2},$$
	for all $l=1,\dots, N$ where $H:=\min_k H_k$ and $F:=\min_{k}F_k$. In the language of the $(f\T-\M)$-property, the $(t^{1/2}\T-\M_{l})$-properties hold.  Using Proposition~\ref{p21},  we interchange magnitude of $f$ and $\M$ in the $(t^{1/2}\T-\M_l)$-properties in  Step 2. Finally, in Step 3 we construction a cutoff function $\sigma$ in any small neighborhood in $M$ of the origin $x_o$ so that the $\sigma$-superlogarithmic property holds, and hence $x_o$ is a weakly superlogarithmic point.  \\
	
	For setting, let    $$\rho:=\sum_k\rho_k\quad \T{where }\quad \rho_k=H_k\left(\sum_{l\in I_k}|h_{l}|^2\right)F_k\left(\sum_{l\in I_k}|Re\,h_{l}|^2\right).$$
 We identify $M$ with $\C^n\times \R$ so that the point $(z,\tau+\sqrt{-1}\rho)\in M$ corresponds to the point $(z,\tau )\in \C^n\times\R$. 
 Set $$L_j:=\frac{\di}{\di z_j}+\sqrt{-1}\frac{\di \rho}{\di z_j}(z)\frac{\di}{\di \tau},\quad  T:=\sqrt{-1}\frac{\di}{\di \tau}, \quad
 \theta:=-\sqrt{-1}d\tau-\sum_{j=1}^n\left(\frac{\di \rho}{\di z_j}dz_j-\frac{\di \rho}{\di \bar z_j}d\bar z_j\right).$$		 		
 It is easy to verify that $L_j\in T^{1,0}M$ with its dual $dz_j$ and $\theta$ is a contact form with its dual $T$.\\ 
 
 {\it Step 1.}
 The Levi form of $M$ is given by 
 $$d\theta=\sum_{ij=1}^n\frac{\di^2\rho}{\di z_i\di\bar z_j}dz_i\we d\bar z_j=\frac{1}{2}(\di_b\dib_b-\dib_b\di_b)\rho=\L_\rho=\sum_{k}\L_{\rho_k}.$$

 We use notation $|h^k|^2=\sum_{j\in I_k}|h_{l}|^2$ and $|\Re\, h^k|^2=\sum_{l\in I_k}|\Re\,h_{l}|^2$. We can check that 
 $$ \L_{|h^k|^2}=\sum_{l\in I_k}\L_{|h_k|^2}=\sum_{l\in I_k}\di_b h_{l}\we \dib_b\bar h_{l}=2\L_{|\Re\,h^k|^2}
 $$
 and \[ \begin{aligned} \di_b|h^k|^2\we \dib_b|\Re\,h^k|^2+\di_b|\Re h^k|^2\we \dib_b|h^k|^2=&2\sum_{l,\hat l\in I_k} \Re\,h_{l}\di_b h_{l}\we\Re\,h_{\hat l} \dib_b\bar h_{\hat l}\\=&4\di_b|\Re\, h^k|^2\we \dib_b |\Re\,h^k|^2\ge 0. \end{aligned}\]
 Thus, by direct computation, we get  \begin{eqnarray}\label{b10}
 \begin{aligned}
 \L_{\rho_k}=&F_k\L_{H_k}+H_k\L_{F_k}+\di_b H_k\we\dib_b F_k+\di_b F_k\we\dib_b H_k\\
 =&F_k\left(\dot H_k\L_{|h^k|^2}+\ddot H_k\di_b|h^k|^2\we \dib_b|h^k|^2\right)+H_k\left(\dot F_k\L_{|\Re\,h^k|^2}+\ddot F_k\di_b|\Re\,h^k|^2\we \dib_b|\Re\,h^k|^2\right)\\
 &+\dot F_k\dot H_k\left(\di_b|h^k|^2\we \dib_b|\Re\,h^k|^2+\di_b|\Re h^k|^2\we \dib_b|h^k|^2\right).
  \end{aligned}
  \end{eqnarray}
 From the last line of \eqref{b10}, it follows that $\L_{\rho_k}\ge \di_b\rho_k\we \dib_b\rho_k$; and from the second line  of \eqref{b10}, it follows
 \begin{eqnarray}
 \begin{aligned} \L_{\rho_k}\ge& (F_k\dot H_k+\frac{1}{2}\dot F_kH_k)\L_{|h^k|^2}\\
 \ge&\frac{1}{2}\frac{H_k(|h^k|^2)F_k(|\Re\,h^k|^2)}{|\Re\,h^k|^2}\L_{|h^k|^2}\\\ge& \frac{1}{2}\sum_{l\in I_k}\frac{H_k(|h_{l}|^2)F_k(|\Re h_{l}|^2)}{(|\Re h_{l}|^2)}\L_{|h_l|^2}
 \end{aligned}
 \end{eqnarray}
 where the inequalities follow by $\frac{F_k(\delta)}{\delta}$ is increasing. 
 Consequently,
 \begin{eqnarray}\label{k11}
 \begin{aligned} d\theta\gtrsim\sum_{l=1}^N\frac{H(|h_{l}|^2)F(|\Re h_{l}|^2)}{(|\Re h_{l}|^2)}\L_{|h_{l}|^2}+\di_b\rho\we\dib_b\rho.
 \end{aligned}
 \end{eqnarray}
  where $H:=\min_k H_k$ and $F:=\min_{k}F_k$.\\
  
 {\it Step 2.}
 In another words of the expression in \eqref{k11}, the $(f_{1}\T-\M_{l,1})$-property holds for 
 $$f_{1}(t)=t^{1/2}
 \quad \T{and}\quad  \M_{l,1}=\frac{H(|h_{l}|^2)F(|\Re h_{l}|^2)}{(|\Re h_{l}|^2)}\L_{|h_{l}|^2}.$$ 
Using Proposition~\ref{p21} below, the  
 $(f_{2}\T-\M_{l,2})$-property holds for 
 $$f_{2}(t)=(F^*(t^{-1}))^{-1/2}\quad\T{ and }\quad \M_{l,2}=H(|h_l|^2)\L_{|h_l|^2},$$
 where $F^*$ is denoted the inverse of $F$.
 By the hypothesis of $F_k$ in \eqref{ht51} and $F(\delta)=\min_k F_k(\delta)$, we get $\lim_{\delta\to0}\delta\ln F(\delta^2)=0$. This implies $f_2\gg \ln$.
Denote $\tilde H(\delta)=\delta H(\delta)$. Since $\M_{l,2}\ge \dfrac{\tilde H(|\Re \,h_{l}|^2)}{|\Re\,h_{l}|^2}\L_{|h_{l}|^2}$,  
 we apply Proposition \ref{p21} again and obtain the $(f_{3}\T-\M_{l,3})$-property holds for $$f_{3}(t)=(\tilde H^*(F^*(t^{-1}))^{-1/2}\gg1\quad \T{ and } \quad \M_{l,3}=\L_{|h_{l}|^2}.$$
 Set $$
 \M_2:=\sum_{l=1}^N H(|h_l|^2)\L_{|h_l|^2}\qquad \T{and} \quad \M_3:=\sum_{l=1}^N\L_{|h_{l}|^2}.$$ 
 Thus, the $(f_2\T-\M_2)$- and $(f_3\T-\M_3)$-properties holds. By Proposition~\ref{p53} below, 
 there exists $m\in \mathbb N$ such that $(f\T-Id)$-property holds where $f(t)=(f_3(t))^{1/m}$ and $Id=\sum_{i=1}^ndz_i\we d\bar z_i$ the identity $(1,1)$-form of $M$.  We conclude that in this step we have proven there exists a family of smooth real value function $\{\lambda_t\}$ such that  $|T^m(\lambda_t)|\le c_{m}e^{mt}$ such that 
  \begin{eqnarray}\label{55m}
  \begin{aligned}
  \L_{\lambda_t}+td\theta\gtrsim f^2(t)Id+f_0^2(t)\ln^2t\sum_{l=1}^NH(|h_{l}|^2)\L_{|h_{l}|^2}+t^\frac{1}{2}\di_b\rho\we\dib_b\rho.
  \end{aligned}
  \end{eqnarray}
  where $f,f_0\gg1$ where $f_0=f_2/\ln$.\\

 {\it Step 3}.	
 Let $\chi$ be a cutoff function in $\R$ such that $\chi(a)=\begin{cases}1 & \T{if } |a|\le 1, \\ 0& \T{if } |a|\ge 2.\end{cases}$. Denote by $\chi_\eps(a)=\chi(a/\eps)$. We set $$h_0:=\tau \quad\T{and}\quad\sigma_\eps:=\prod_{l=0}^N\chi_\eps(|h_{l}|^2).$$ Since the common zero of $h_{l}$, for $l=1,\dots,N$,
 is at only the origin of $\C^n$. Thus $\sigma_\eps$ is a cutoff function on $M$. Furthermore, for any small neighborhood $U$ in $M$ of $x_o$ (the origin of $M$) there exists $\eps>0$ such that $\supp(\sigma_\eps)\subset U$ and $\sigma_\eps=1$ in another neighborhood of $x_o$. In order to prove that $x_o$ is the weakly superlogarithmic point, we shall show that the $\sigma_\eps$-superlogarithmic property holds in $U$. Before giving the upper bounds of $\di_b\sigma_\eps\we \dib_b\sigma_\eps$ and $\L_{\sigma_\eps}$, we compute $$\L_{|h_0|^2}=2\di_bh_0\we \dib_bh_0+2h_0\L_{h^0}=2\di_b\tau\we \dib_b\tau+2\tau\L_\tau.$$  We have 
 $$\di_b\tau=\sum_{j=1}^nL_j\tau dz_j=\sqrt{-1}\sum_{j=1}^n\frac{\di\rho}{\di z_j}dz_j=\sqrt{-1}\di_b\rho.$$
 Thus, $\dib_b\tau=-\sqrt{-1}\dib_b\rho$ and hence 
 $$\di_b\tau\we \dib_b\tau =\di_b\rho\we \dib_b\rho\underset{\T{by  \eqref{k11}}}{\lesssim}\L_\rho=d\theta.$$ 
 $$\L_\tau=\frac{1}{2}\left(\di_b\dib_b\tau-\dib_b\di_b\tau\right)=\frac{1}{2}\sum_{ij=1}^n\left(-\sqrt{-1}\frac{\di^2\rho}{\di z_i\di\bar z_j}dz_i\we d\bar z_j-\sqrt{-1}\frac{\di^2\rho}{\di \bar z_i\di  z_j}d\bar z_i\we d z_j \right)\equiv0.$$
 Thus $\L_{|h_0|^2}\lesssim d\theta$. 
 We are ready to estimate $\di_b\sigma_\eps\we \dib_b\sigma_\eps$ and $\L_{\sigma_\eps}$
	$$\di_b\sigma_\eps=\sum_{l=0}^N\Big(\prod_{\hat l\not=l}\chi(|h_{\hat l}|^2)\Big)\Big(\dot \chi(|h_l|^2)\di_b |h_l|^2\Big)$$
	and \[\begin{aligned}
	\L_{\sigma_\eps}=&\sum_{l=0}^N\left\{\Big(\prod_{\hat l\not=l}^N\chi(|h_{\hat l}|^2)\Big)\Big(\dot{\chi}(|h_l|^2)\L_{|h_l|^2}+\ddot{\chi}(|h_l|^2)\di|h_l|^2\we\dib_b|h_l|^2\Big)\right\}\\
	&+\sum_{k,l=0,k\not=l}^N\left\{\Big(\prod_{\hat l\not=k,l}\chi(|h_{\hat l}|^2)\Big)\Big(\dot{\chi}(|h_k|^2)\dot{\chi}(|h_l|^2)\di |h_k|^2\we\dib_b|h_l|^2\Big)\right\}.
	\end{aligned}\]
	Thus,  it follows
	\begin{eqnarray}
	\label{56m}\begin{aligned}
	\ln t|\la\L_\sigma\lrcorner u,u\ra|+\ln^2t|\di_b\sigma\lrcorner u|^2
	\lesssim &\ln^2t\sum_{l=0}^NI_{\supp(\dot\chi_\eps(|h_l|^2))} |\la\L_{|h_l|^2}\lrcorner u,u\ra|\\
	\le &\ln^2t\left(\la d\theta\lrcorner u,u\ra+(H(\eps^2))^{-1}\ln^2 t\sum_{l=1}^NH(|h_l|^2)\la\L_{|h_l|^2}\lrcorner u,u\ra\right)
	\end{aligned}
	\end{eqnarray}
	for any $t\ge 1$ and $u\in C^\infty_{0,1}(M)$. 
	where  the second inequality follows by the fact that 
	$$I_{\supp(\cdot\chi_\eps(|h_l|^2))}\le I_{\{x\in M: \eps\le|h_l(x)|\le 2\eps\}}\le H(|\eps|^2)^{-1}H(|h_l|^2).$$
	Combining \eqref{55m} and \eqref{56m}, we can conclude that the $\sigma_\eps$-superlogarithmic property holds for any $\eps>0$. 
\end{proof}
The proof of Theorem~\ref{t51m} is complete but we have skipped a crucial technical point in Step 2 of the proof that we face now. First, the interchange sizes of the rate $f$ and the multiplier $\M$ is expressed by the following proposition. 
\begin{proposition}\label{p21} Let $M$ be the hypersurface  defined by \eqref{e1} and $h$ be a holomorphic function in $\C^n$. 
	Assume that the $(f_1\T-\M_1)$-property holds with a rate $1\ll f_1\lesssim t^{1/2}$ and  multiplier $$\M_1\ge \frac{H(|h|^2)F(|\Re\,h|^2)}{|\Re\,h|^2}\L_{|h|^2}.$$ Then the $(f_2\T-\M_2)$-property holds for $f_2(t)=(F^*(f_1(t)^{-2}))^{-1/2}$ and $\M_2=H(|h|^2)\L_{|h|^2}$. Here  $F,H:\R^+\to \R^+$ are smooth increasing, convex functions with $F(0)=0$, $\frac{F(\delta)}{\delta}\searrow0$; and $F^*$ is denoted the inverse of $F$.
\end{proposition}	

\begin{proof}Assume that for any $t\ge 1$ there exists a function $\lambda_{t,1}$ such that $$|T^s(\lambda_{t,1})|\le c_{s,1}e^{ts/2},\quad\T{for all $s\in\mathbb N$ and  }
	\L_{\lambda_{t,1}}+td\gamma\ge f_1(t)^2\M_1.$$
	 Now we construct $\lambda_{t,2}$ such that 
	 \begin{eqnarray}
	 \label{new1}
	 T^s(\lambda_{t,2})|\le c_{s,2}e^{ts/2},\quad\T{for all $s\in\mathbb N$ and  }
	 \L_{\lambda_{t,2}}+td\gamma\ge f_2(t)^2\M_2.
	 \end{eqnarray}
	We define $$\tilde\lambda_t(z,\tau)=\lambda_t(|\Re\, h(z)|^2)H(|h(z)|^2) 
	\quad \T{where}\quad \lambda_t(a):=\left(1-e^{-\frac{1}{2}f^2_2(t)a}\right).$$  Since $\dot \lambda_t(a)=\frac{1}{2}f_2^2(t)\lambda_t(a)\ge 0$ and $\ddot\lambda_t(a)=-\frac{1}{4}f_2^4(t)\lambda_t(a)$, it follows \begin{eqnarray}\label{1611}
	\begin{aligned}
	\L_{\tilde \lambda_t}=&H\L_{\lambda_t}+\lambda_t\L_{H}+\di_b(H)\we \dib_b(\lambda_t)+\di_b(\lambda_t)\we \dib_b(H)\\
\ge &H\left(\dot \lambda_t+|\Re\,h|^2\dot \lambda_t)\right)\di_bh\we\dib_b\bar h\\
	\ge&\frac{1}{4} f^2_2(t)He^{-\frac{1}{2}f^2_2(t)|\Re\,h|^2}\left\{\Big(2-f_2^2(t)|\Re\,h|^2\Big)\di_bh\we \dib_b\bar h\right\}
	\end{aligned}
	\end{eqnarray}
	If $f_2|\Re\,h|\le 1$, then $\exp\left(-\frac{1}{2}f^2_2(t)h^2\right)\ge e^{-1/2}$ uniformly in $t$ and hence 
	$$\L_{\tilde\lambda_t}\ge \frac{1}{4e^{1/2}} f^2_2(t)H(|h|^2)\di_bh\we\dib_b\bar h\ge \frac{1}{8e^{1/2}}f^2_2(t)\M_2.$$
	Otherwise, assume that $f_2|\Re\,h|\ge 1$. Using our assumption that $F(\delta)/\delta$ increasing for $|\Re\,h|\ge f_2^{-1} $ and the definition of $f_2(t)=(F^*(f_1(t))^{-2})^{-1/2}$, we get $$f_1^2(t)\frac{F(|\Re\,h|^2)}{|\Re\,h|^2}\ge f_1^2(t)\frac{F((f_2(t))^{-2})}{(f_2(t))^{-2}}=f_2^2(t).$$ In this case, $\L_{\tilde \lambda}$ can get negative; however, by using the fact that 
	$\min_{a\ge 1}\Big\{(2-a)e^{-\frac{1}{2}a}\Big\}=-2e^{-2}$  at $a=4$ for $a=f_2(t)|\Re,h|\ge 1$, we have $$\L_{\tilde\lambda_t}\ge -\frac{1}{2e^{2}}f_2^2(t)\M_2$$
	Thus 
	$$f_1^2(t)\M_1+\L_{\tilde\lambda_t}\ge (1-\frac{1}{2e^{2}})f_2(2)^2\M_2.$$ 
	In both cases, we have proved
	$$f_1^2(t)\M_1+\L_{\tilde\lambda_t}\ge \frac{1}{8e^{1/2}}f_2(t)^2\M_2.$$
	Finally, we set $\lambda_{t,2}=\lambda_{t,1}+\tilde{\lambda}_t$. It is easy to check that both $|T^s(\lambda_{t,2})|\le c_{s,2}e^{ts/2}$ for some $c_{s,2}>0$. Therefore,\ $\lambda_{t,2}$ satisfies \ref{new1} and hence Proposition \ref{p21} is proven.
\end{proof}	
\begin{proposition}\label{p53}Let $M$ be the hypersurface  defined by \eqref{e1} and $h_1,\dots, h_N$ be holomorphic functions defined in a neighborhood of the origin in $\C^n$ such that  and 
	\begin{enumerate}
		\item the common zero of $h_l$, $l=1,\dots,N$ is isolated at the origin of $\C^n$;
		\item the $(f\T-\L_{|h^2_l|})$-properties  hold for all $l=1,\dots, N$ with some $f\gg 1$.
	\end{enumerate} 
	Then, there exist an integer $m$ such that  the property $(\sqrt[m]{f}\T-Id)$ holds. 
\end{proposition}
	  Following the guidelines of Kohn in \cite{Koh79,Koh00}, we consider the ideas of germ of holomorphic function $\mathcal J_p$ defined inductively as follows. If $\G$ is  a $P$-tuple  $(g_1,\dots, g_P)$ is  of germs of holomorphic functions in $\C^n$ let $\mathcal B(\G)$ denote the $n\times (N+P)$ matrix
	  
  \[\mathcal B(\G) = 
  \begin{pmatrix}
  	\frac{\di h_{1}}{\di z_1} & \frac{\di h_{1}}{\di z_2}& \cdots & \frac{\di h_{1}}{\di z_n} \\
  	\vdots &\vdots&\ddots&\vdots\\
  \frac{\di h_{N}}{\di z_1} & \frac{\di h_{N}}{\di z_2}& \cdots & \frac{\di h_{N}}{\di z_n} \\
  \frac{\di g_{1}}{\di z_1} & \frac{\di g_{1}}{\di z_2}& \cdots & \frac{\di g_{1}}{\di z_n} \\
  	\vdots  & \vdots  & \ddots & \vdots  \\
  \frac{\di g_{P}}{\di z_1} & \frac{\di g_{P}}{\di z_2}& \cdots & \frac{\di g_{P}}{\di z_n} \\
  \end{pmatrix}.
 	\]
 Denote by $\det \mathcal{B}(\G)$ the ideal generated by all determinants of all the $n\times n$ submatrices of $\mathcal{B}(\G)$. Then we define 
 $$\mathcal J_0=\sqrt{\det\mathcal B(\{0\})}$$ 
 and for $p\ge 1$
  $$\mathcal J_p=\sqrt{\left(\mathcal J_{p-1},\det\mathcal B(\mathcal J_{p-1})\right)}.$$ 
 Here, $\sqrt{\mathcal J}$ is a ideal of holomorphic functions such that  if $g\in \sqrt{\mathcal J}$ then there exists $\tilde g\in \mathcal J$ and $m\in \N$ such that $|g|^m\le |\tilde g|$, where $J$ is a deal of germs of holomorphic functions.\\

For the proof of Proposition~\ref{p53}, we also need the following lemma.
\begin{lemma}\label{l54}Let $g$ be a holomorphic function and $f\gg 1$. Then
 $$\T{The $\left(f\T-|g|^{2m}\Id\right)$-property}\quad \Longrightarrow\quad  \T{ the $\left(f^{1/m}\T-|g|^2\Id\right)$-property}.$$
\end{lemma}\begin{proof}
The proof follows immediately by the fact that $\L_{|z|^2}=\Id$ (i.e., the $(1\T-\Id)$-property holds) and
  $$f^{2/m}(t)|g(z)|^2\le \left[f^{2/m}(t)|g(z)|^2\right]^{m}+1=f^2(t)|g|^{2m}+1.$$
  \end{proof}
We are ready to prove Proposition~\ref{p53}. 
\begin{proof}  It was proved in \cite[Theorem 7.13]{Koh79} that $1\in \mathcal J_p$ for some integer $p$ if and only if the common zero of the functions $h_1,\dots,h_N$ is the origin of $\C^n$. Thus, we only need to prove that if $g\in \mathcal J_p$ then there exists an integer $m_p$ such that $(f^{1/m_p}\T-|g|^2\Id)$-property holds. We prove this statement by the inductive method. The argument for checking the first step $p=0$, by using the hypothesis that the $(f\T-\L_{|h_j|^2})$-property holds for all $j=1,\dots,N$,  is similar for a general step $p$ so we omit here the step $p=0$. Now we assume inductively that the $(f^{\frac{1}{m_{p-1}}}\T-|\hat g|^2\Id)$ holds for all $\hat g\in \mathcal J_{p-1}$. By Proposition~\ref{p21}, the $(f^{\frac1{2m_{p-1}}}\T-\L_{|\hat g|^2})$-property holds for all $\hat g\in \mathcal J_{p-1}$ and hence by the definition of $\det \mathcal B(\mathcal J_{p-1})$, we get $(f^{\frac1{2m_{p-1}}}\T-|\tilde g|^2\Id)$-property holds for all $\tilde g\in \left(J_{p-1},\det\mathcal B(\mathcal J_{p-1}) \right)$. Now taking $ g\in \mathcal J_p $ there exist an integer $\tilde m_q$ and $\tilde g\in \left(\mathcal J_{p-1},\det\mathcal B(\mathcal J_{p-1}) \right)$ such that $| g|^{\tilde m_p}\le |\tilde g|$. By Lemma~\ref{l54}, the $(f^{\frac1{m_p}}\T-|g|^2\Id)$-property holds for $m_p:=2m_{p-1}\tilde m_p$. This proves Proposition~\ref{p53}.

\end{proof}

	Now we are ready for the proof of Theorem~\ref{t1.9}.
	\begin{proof}[Theorem~\ref{t1.9}] By the argument in Theorem~\ref{t51m}, we can prove that any $x\in M$ is a weakly superlogarithmic point. In order to use Theorems~\ref{maintheorem1} and ~\ref{maintheorem2}, we have to show that $\dib_b$ has closed range in $L^2$ for all degree of forms and $G_0$ is globally regular in the case $n=1$. Up the author's knowledge,  the closed range of $\dib_b$ and global regularity of $G_0$ have been established only on compact CR manifolds (cf. \cite{Bar12, Koh86, KoNi06, Nic06, Sha85, StZe15}) nevertheless $M$ is not a compact CR manifold. Since $M$ is a hypersurface in $\C^{n+1}$,  we can compactify $M$  by \cite[Proposition 2.1]{McN92} as follows. Let $V$ be an open set in $M$, there exists a smooth, compact, pseudoconvex hypersurface $\tilde M$ in $\C^{n+1}$ satisfying the following properties:
	\begin{itemize}
		\item $V\subset M\cap \tilde M$,
		\item all points in $\tilde M\setminus M$ are points of strong pseudoconvexity,
		\item the relative boundary $S$ of $\tilde M\cap M$ and $\tilde M\setminus M$ is the intersection of $M$ with a sphere centered at the origin of $\C^{n+1}$.
	\end{itemize}
	Thus, the $\dib_{b,\tilde M}$ has closed range where the subscript $\tilde M$ means the operator of $\tilde M$.  Analogously to \cite[Proposition 4.1]{KhZa12}, we have $G_{q,\tilde M}$ is globally exactly regular for all $0\le q\le n$. Indeed, since the compactness estimate for the system $(\dib_{b,M},\dib^*_{b,M})$ holds on any open set in $M$ (see Step 2 in the proof of Theorem~\ref{t51m}),  similar to the argument in the proof of \cite[Proposition 4.1]{KhZa12}, we can obtain the (global) compactness estimate for the system $(\dib_{b,\tilde M},\dib^*_{b,\tilde M})$ and hence the global regularity of $G_{q,\tilde M}$ follows (see \cite{Rai10, KhPiZa12a, Str12,StZe15}). Now we can apply Theorems~\ref{maintheorem1} and \ref{maintheorem2}, we obtain $\Box_{b,\tilde M}$ is locally hypoelliptic. On local set $V\subset M\cap\tilde M$, the Kohn-Laplacians $\Box_{b,M}$ and $\Box_{b,\tilde M}$ coincide. Therefore, $\Box_{b,M}$ is local hypoelliptic on $V$. This completes the proof of Theorem~\ref{t1.9}.   
		\end{proof}
	 

 								\section*{Appendix}
 								In this appendix, we give estimates of commutators to support the work  in \S\ref{s2}. The following lemma has been used in the proof of Proposition~\ref{Local3.2}
 								\begin{lemma}\label{l1} Let $\lambda:=\lambda_{k,s}$ be a real-valued $C^\infty$ function defined on an open set $U$ of $M$ such that $|\lambda|\le c_{0,s}k$ and $|T^m(\lambda)|\le c_{m,s} e^{\frac{mk}{2}}$  for any $m\in\mathbb N$. Then, for any $p,r\in\mathbb \mathbb N$ and $
 									\alpha\in \R$, there exists a constant $c_{s,r,p,\alpha}$ independent of $k$
 									such that	$$e^{\alpha k}\no{T^p[ e^{-\frac{\lambda}{2}}\zeta,\tilde{\Gamma}_k^+] \Gamma_k\Psi^+ \zeta u}_{L^2}\le c_{s,r,p,\alpha}\no{ \zeta u}_{H^{-r}}$$
 									holds for all $u\in H^{-s_o}(M)$ with $r\ge s_o$.  Here, $\mathcal S(\tilde\Gamma^+_k)\succ \mathcal S(\Gamma_k\Psi^+)$.						
 						\end{lemma}
						We first remark that the supports of  the symbols of $[ e^{-\frac{\lambda}{2}}\zeta,\tilde{\Gamma}_k^+] $ and $\Gamma_k\Psi^+$ are disjoint, the operator $[ e^{-\frac{\lambda}{2}}\zeta,\tilde{\Gamma}_k^+] \Gamma_k\Psi^+$ is order $-\infty$. This implies $\no{T^p[ e^{-\frac{\lambda}{2}}\zeta,\tilde{\Gamma}_k^+] \Gamma_k\Psi^+ \zeta u}^2$ can be controlled by any negative Sobolev norm with multiplying by a constant. However this constant might depend on $k$. The Lemma \ref{l1} is to verify that the upper bounded of  $e^{\alpha k}\no{T^p[ e^{-\frac{\lambda}{2}}\zeta,\tilde{\Gamma}_k^+] \Gamma_k\Psi^+ \zeta u}^2$ is independent on $k$. 
 						\begin{proof}							
 							Let $$K(\xi_{2n+1},\eta_{2n+1})=\xi^p_{2n+1}[\tilde \gamma(e^{-k}\xi_{2n+1})-\tilde \gamma(e^{-k}\eta_{2n+1})]\F_{2n+1}\{e^{-\frac{\lambda}{2}}\zeta\}(\xi_{2n+1}-\eta_{2n+1})\eta^r_{2n+1}$$
 							on support of $\gamma(e^{-k}\eta_{2n+1})\F_{2n+1}(\Psi^+\zeta v)(\eta_{2n+1})$. From Taylor's theorem, we have for any $m\ge 1$ there exists $a_m\in [0,1]$ such that
 							\begin{eqnarray}\begin{aligned}
 								\tilde\gamma(e^{-k}\xi_{2n+1})-\tilde\gamma(e^{-k}\eta_{2n+1})
 									=&\sum_{j=1}^{m-1}\frac{1}{j!}e^{-jk}(\xi_{2n+1}-\eta_{2n+1})^j\tilde\gamma^{(j)}(e^{-k}\eta_{2n+1})\\
 									&+\frac{1}{m!}e^{-mk}(\xi_{2n+1}-\eta_{2n+1})^m\tilde\gamma^{(m)}\left(a_me^{-k}\xi_{2n+1}+(1-a_m)e^{-k}\eta_{2n+1}\right).
 								\end{aligned}
 							\end{eqnarray}
 							On support of $\gamma(e^{-k}\eta_{2n+1})$, the summation $\sum_{j=1}^{m-1}\bullet$ is zero since $\mathcal S(\tilde\Gamma^+_k)\succ \mathcal S(\Gamma_k\Psi^+)$ and hence
 							\begin{eqnarray}
 								\begin{aligned}
 									|K(\xi_{2n+1},\eta_{2n+1})|\le& c_m e^{-mk}|\eta_{2n+1}^{r}|\left|\xi_{2n+1}^p(\xi_{2n+1}-\eta_{2n+1})^m\F_{2n+1}\{e^{-\frac{\lambda}{2}}\zeta\}(\xi_{2n+1}-\eta_{2n+1})\right|\\
 									\le & c_{m,p} e^{-mk}|\eta_{2n+1}^{r}|\Big[\big|\underbrace{(\xi_{2n+1}-\eta_{2n+1})^{m+p}\F_{2n+1}\{e^{-\frac{\lambda}{2}}\zeta\}(\xi_{2n+1}-\eta_{2n+1})}_{I_{m+p}}\big|\\
 									&+\big|\eta^{p+r}_{2n+1}\underbrace{(\xi_{2n+1}-\eta_{2n+1})^{m}\F_{2n+1}\{e^{-\frac{\lambda}{2}}\zeta\}(\xi_{2n+1}-\eta_{2n+1})}_{I_m}\big|\Big].
 								\end{aligned}
 							\end{eqnarray}
 							For $l=m+p$ or $m$, we rewrite $I_l$ as
 							\begin{eqnarray}
 								\begin{aligned}
 									I_l=&(|\xi_{2n+1}-\eta_{2n+1}|^2+1)^{-1}\F_{2n+1}\{[(\xi_{2n+1}-\eta_{2n+1})^{l+2}+(\xi_{2n+1}-\eta_{2n+1})^{l}]e^{-\frac{\lambda}{2}}\zeta\}(\xi_{2n+1}-\eta_{2n+1})\\
 									=& (|\xi_{2n+1}-\eta_{2n+1}|^2+1)^{-1}\F_{2n+1}\{(e^{-\frac{\lambda}{2}}\zeta)^{(l+2)}+(e^{-\frac{\lambda}{2}}\zeta)^{(l)}\}(\xi_{2n+1}-\eta_{2n+1})\\
 									=& (|\xi_{2n+1}-\eta_{2n+1}|^2+1)^{-1}\left[\F_{2n+1}\{(e^{-\frac{\lambda}{2}}\zeta)^{(l+2)}\}(\xi_{2n+1}-\eta_{2n+1})+\F_{2n+1}\{(e^{-\frac{\lambda}{2}}\zeta)^{(l)}\}(\xi_{2n+1}-\eta_{2n+1})\right].
 								\end{aligned}
 							\end{eqnarray}
 							On other hand, 
 							\begin{eqnarray}
 								\begin{aligned}
 									|\F_{2n+1}\{(e^{-\frac{\lambda}{2}}\zeta)^{l}\}(\xi_{2n+1}-\eta_{2n+1})|=&\left|\int_\R e^{-i(\xi_{2n+1}-\eta_{2n+1})x_{2n+1}}T^{l}(e^{-\frac{\lambda}{2}}\zeta)dx_{2n+1}\right|\\
 									\le&\int_{\R} \left|T^{l}(e^{-\frac{\lambda}{2}}\zeta)\right|dx_{2n+1}\\
 									\le& c_{l,s} e^{\frac{lk}{2}+c_{0,s}k},
 								\end{aligned}
 							\end{eqnarray}
 							where the last inequality follows by the assumption on $\lambda:=\lambda_{k,s}$.
 							Thus, $|I_l|\le c_{l,s} e^{\frac{k(l+2)}{2}+c_{0,s}k}$.  Note that on support of $\gamma(e^{-k}\eta_{2n+1})$, $|\eta_{2n+1}|\le e^k$. 
 							Therefore, 
 							$$|K(\xi_{2n+1},\eta_{2n+1})|\le c_{m,p,s}(|\xi_{2n+1}-\eta_{2n+1}|^2+1)^{-1}e^{-mk+\frac{(m+2)k}{2}+(r+p+c_{0,s})k},$$
 						 holds for any $m\in\mathbb N$. Fix $\alpha\in \R$ choose $m$ sufficiently large, we got
 								$$|K(\xi_{2n+1},\eta_{2n+1})|\le c_{\alpha, p, r,s}e^{-\alpha k}(|\xi_{2n+1}-\eta_{2n+1}|^2+1)^{-1}.$$
 					This boundedness of the kernel $K(\xi_{2n+1},\eta_{2n+1})$ implies that 
 					$$e^{\alpha k}\no{T^p[ e^{-\frac{\lambda_k}{2}}\zeta,\tilde{\Gamma}_k^+] \Gamma_k\Psi^+ \zeta u}_{L^2}\le c_{\alpha,p,r,s}\no{ T^{-r}\Gamma_k\Psi^+ \zeta u}_{L^2}\le c_{\alpha,p,r,s}\no{  \zeta u}_{H^{-r}}. $$
 					This completes the proof of this lemma.
 						\end{proof}
 					The following lemma has been used to estimate the commutators
 					$$\sum_k\no{e^{ks\sigma}[\dib_b,\zeta\Gamma_k\Psi^+\zeta]u}^2_{L^2}+\no{e^{ks\sigma}[\dib_b^*,\zeta\Gamma_k\Psi^+\zeta]u}^2_{L^2}$$ in Lemma~\ref{com1}  in \S\ref{s2}.
 						
 						\begin{lemma}\label{l2} Let  $X$ be a differential operator of order one with smooth coefficients and $u\in H^{-s_o}(M)$. Then, the following holds uniformly in $k$.
 							\begin{eqnarray*}
 							(i)\quad\no{e^{ks\sigma}[X,\zeta\Gamma_k\Psi^\pm\zeta]u}^2_{L^2}\le & c_s\left( \no{e^{(k-1)s\sigma}\zeta\Gamma_{k-1}\Psi^\pm\zeta u}^2_{L^2} +\no{e^{(k+1)s\sigma}\zeta\Gamma_{k+1}\Psi^\pm\zeta u}^2_{L^2}\right)\\ &+c\left(\no{\tilde \Gamma_k\tilde{\zeta}u}^2_{L^2}+e^{2ks}\no{\tilde \Gamma_k\tilde \Psi^0\zeta u}^2_{L^2}\right);
 							\end{eqnarray*}
 						$\quad(ii)\quad\no{Xe^{ks\sigma}\zeta\Gamma_k\Psi^\pm\zeta u}^2_{L^2}\le c\left( \no{e^{ks\sigma+k}\zeta\Gamma_k\Psi^\pm\zeta u}^2_{L^2}+e^{2k}\no{\Gamma_k\zeta u}^2_{L^2}\right).$\\\
 					Here  we recall that $\zeta$ is a cutoff function such that $\zeta\prec \tilde\zeta$; $\tilde\Psi^0$ is defined in \S\ref{s2.2} such that $\mathcal S([\Psi^\pm,\alpha])\prec \mathcal S(\tilde \Psi^0)$; and $\tilde\Gamma_k$ is a pseudodifferential operator such that $ \mathcal S(\Gamma_k) \prec \mathcal S(\tilde\Gamma_k)$.
 						\end{lemma}
 						\begin{proof}
 						
 						For the proof of (i),	we first use the Jacobi identity to get 
\begin{eqnarray}\label{n1}
\begin{aligned}
	\no{e^{ks\sigma}[X,\zeta\Gamma_k\Psi^\pm\zeta]u}^2_{L^2}\le& 4\Big(\no{e^{ks\sigma}\zeta[X,\Gamma_k]\Psi^\pm\zeta u}^2_{L^2}+\no{e^{ks\sigma}\zeta\Gamma_k[X,\Psi^\pm]\zeta u}^2_{L^2}\\
	+&\no{e^{ks\sigma}[X,\zeta]\Gamma_k\Psi^\pm\zeta u}^2_{L^2}+\no{e^{ks\sigma}\zeta\Gamma_k\Psi^\pm[X,\zeta]u}^2_{L^2}\Big).
\end{aligned}
\end{eqnarray} 												
 							Since $\supp(\dot{\zeta})\cap \supp(\sigma)=\emptyset$, the two terms in the second line of \eqref{n1} are bounded by $\no{\tilde \Gamma_k\tilde\zeta u}^2_{L^2}$. By the choice of $\tilde\Psi^0$ and $\tilde\Gamma_k$, it follows $\mathcal S(\Gamma_k[X,\Psi^\pm])\prec\mathcal S(\tilde{\Gamma}_k\tilde\Psi^0)$, thus the last term in the first line of \eqref{n1} can be bounded as follows 
 					\begin{eqnarray}\label{n2}
 					\begin{aligned}\no{e^{ks\sigma}\zeta\Gamma_k[X,\Psi^\pm]\zeta u}^2_{L^2}\le e^{2ks}\no{\Gamma_k[X,\Psi^\pm]\tilde \Psi^0\zeta u}^2_{L^2}
 						\le	c e^{2ks} \no{\tilde{\Gamma}_k\tilde \Psi^0\zeta u}^2_{L^2}.
 							 \end{aligned}
 							 \end{eqnarray}  
 					 The first term in the RHS. of the first line of \eqref{n1} is treated as follows. Since $(\gamma_{k+1}+\gamma_{k-1})=1$ on the support of $\dot\gamma_k$, we have
 					\begin{eqnarray}\label{n3}
 					\begin{aligned}\no{e^{ks\sigma}\zeta[X,\Gamma_k]\Psi^\pm\zeta u}^2_{L^2}=&
 					\no{e^{ks\sigma}\zeta[X,\Gamma_k]\tilde\Psi^\pm\left(\Gamma_{k-1}+\Gamma_{k+1}\right)\Psi^\pm\zeta u}^2_{L^2}\\
 				\le& 						\no{[X,\Gamma_k]\tilde\Psi^\pm e^{ks\sigma}\zeta\left(\Gamma_{k-1}+\Gamma_{k+1}\right)\Psi^\pm\zeta u}^2_{L^2}\\
 				&\no{[e^{ks\sigma},[X,\Gamma_k]\tilde\Psi^\pm]\zeta\left(\Gamma_{k-1}+\Gamma_{k+1}\right)\Psi^\pm\zeta u}^2_{L^2}\\
 				&+\no{e^{ks\sigma}[\zeta,[X,\Gamma_k]\tilde\Psi^\pm]\left(\Gamma_{k-1}+\Gamma_{k+1}\right)\Psi^\pm\zeta u}^2_{L^2}.
 					\end{aligned}
 					\end{eqnarray}  
 					The term in the last line of \eqref{n3} is bounded by $\no{\tilde \Gamma_k\tilde\zeta u}^2_{L^2}$ since $\supp(\dot{\zeta})\cap \supp(\sigma)=\emptyset$. 
 		Sine the symbol of $[X,\Gamma_k]\tilde \Psi^\pm$ is uniformly bounded by a constant, the term in second line of \eqref{n3} is bounded by $\no{ e^{ks\sigma}\zeta\left(\Gamma_{k-1}+\Gamma_{k+1}\right)\Psi^\pm\zeta u}^2_{L^2}$. We also notice that the support of  $\mathcal S([X,\Gamma_k]\tilde \Psi^\pm)$ is a subset of  $\supp(\dot\gamma_k\tilde \psi^+)$ and hence
 			$$\no{[e^{ks\sigma},[X,\Gamma_k]\tilde\Psi^\pm]v}^2_{L^2}\le c_s k^2e^{-2k}\no{e^{ks\sigma}v}^2_{L^2}\le c_s\no{e^{ks\sigma}v}^2_{L^2}. $$ Thus, the term in the third line of \eqref{n3} is also bounded by $\no{ e^{ks\sigma}\zeta\left(\Gamma_{k-1}+\Gamma_{k+1}\right)\Psi^\pm\zeta u}^2_{L^2}$. This proves the first part of Lemma~\ref{l2}.
 			
 			For the proof of (ii), we have 
 			\begin{eqnarray}\label{n4}
 			\begin{aligned}\no{Xe^{ks\sigma}\zeta\Gamma_k\Psi^\pm\zeta u}^2_{L^2}\lesssim& \no{\La \tilde \zeta e^{ks\sigma}\zeta\tilde \Gamma_k\tilde \Psi^\pm\Gamma_k\Psi^\pm\zeta u}_{L^2}^2\\
 			\lesssim& \no{\La \tilde \Gamma_k\tilde \Psi^\pm\tilde \zeta e^{ks\sigma }\zeta\Gamma_k\Psi^\pm\zeta u}^2_{L^2}+\no{\La [\tilde \zeta e^{ks\sigma}\zeta,\tilde \Gamma_k\tilde \Psi^\pm]\Gamma_k\Psi^\pm\zeta u}^2_{L^2}.
 				\end{aligned}
 				\end{eqnarray}
 			The symbol $\mathcal S(\La \tilde \Gamma_k\tilde \Psi^\pm)$ is bounded by $e^k$ since, so the first term in the second line of \eqref{n4} is bounded by $\no{e^{ks\sigma+1}\zeta\Gamma_k\Psi^\pm\zeta u}^2$. Similar to the argument in the previous part, the second term in the second line of \eqref{n4} is bounded by 
 			$k^2\no{e^{ks\sigma}\zeta\Gamma_k\Psi^\pm\zeta u}^2_{L^2}+e^{2k}\no{ \Gamma_k \zeta u}^2_{L^2}$. This proves (ii) and then the proof of Lemma~\ref{l2} is complete.
 						\end{proof}

 								\bibliographystyle{alpha}
 							
 							\end{document}